\documentclass[oneside,12pt,letterpaper]{Classes/CUEDthesisPSnPDF}
 
\usepackage{amsthm,amsmath,amssymb,amsbsy,amsfonts,graphicx,graphics,appendix,hyperref,lscape,listings,courier}

\usepackage[hyperpageref]{backref} 

\hypersetup{bookmarksopen=false} 

\usepackage{enumitem} 

\newcommand{\acode}{

\footnotesize \ttfamily}
\newcommand{\zcode}{\rmfamily \normalsize}
\def\listofsymbols{
\begin{tabbing}
$\R$~~~~~~~~~~~~~\=\parbox{5in}{Field of real numbers\dotfill \pageref{symbol:RR}}\\
\addsymbol \C: {Field of complex numbers}{symbol:CC}
\addsymbol |z|: {Modulus of the complex number $z$}{symbol:modulus}
\addsymbol \Re(z): {Real part of the complex number $z$}{symbol:real}
\addsymbol \Im(z): {Imaginary part of the complex number $z$}{symbol:imag}

\addsymbol s_n(f;z): {$n^{\text{th}}$ partial sum of the Maclaurin series for $f$}{symbol:section}
\addsymbol t_n(f;z): {$n^{\text{th}}$ tail of the Maclaurin series for $f$}{symbol:tail}

\addsymbol \sharp_n^{\angle}(\theta_1,\theta_2): {Number of zeros of the $n^{\text{th}}$ section in the sector $\theta_1 \leq \arg z \leq \theta_2$}{symbol:numsec}
\addsymbol \sharp_{n}^{\circ}(R): {Number of zeros of the $n^{\text{th}}$ section in the disk $|z| \leq R$}{symbol:numdisk}
\addsymbol \{N\}: {An appropriate subsequence of the indices $\{n\}$}{symbol:Nindices}

\addsymbol \rho: {Order of an entire function}{symbol:rho}
\addsymbol \rho_n: {}{symbol:rhon}

\addsymbol \binom{n}{k}: {Coefficient of $x^k$ in the expansion of $(1+x)^n$}{symbol:binom}
\addsymbol \Gamma: {Gamma function}{symbol:gammafunc}
\addsymbol \log: {Natural logarithm}{symbol:log}

\addsymbol \longrightarrow: {Converges to}{symbol:arrow}
\addsymbol \approx: {Is approximately}{symbol:approx}
\addsymbol O(~~): {Big O notation}{symbol:bigo}
\addsymbol \sim: {Asymptotically equivalent}{symbol:sim}

\end{tabbing}
 \clearpage}
\def\addsymbol #1: #2#3{$#1$ \> \parbox{5in}{#2 \dotfill \pageref{#3}}\\} 
\def\newnot#1{\label{#1}} 
\newtheorem{theorem}{Theorem}[chapter]

\newtheorem{lemma}[theorem]{Lemma}
\newtheorem{corollary}[theorem]{Corollary}

\def\C{\mathbb{C}}
\def\R{\mathbb{R}}

\DeclareMathOperator{\dist}{dist}
\DeclareMathOperator{\maxdist}{maxdist}
\DeclareMathOperator{\erfc}{erfc}

\ifpdf
    \pdfcatalog { /PageMode (/UseOutlines)
                  /OpenAction (fitbh)  }
\fi

\degree{Doctor of Philosophy}
\degreedate{Yet to be decided}

\usepackage{StyleFiles/watermark}

\begin{document}




\pagenumbering{roman}
\setcounter{page}{0}
\renewcommand{\footnotesize}{\small}
\renewcommand{\footnoterule}{\relax}
\thispagestyle{empty}

\null\vskip0.5in
   \begin{center}
      \hyphenpenalty=10000\LARGE\bfseries {Zeros of Sections of Some Power Series}
   \end{center}
   \vfill
   \begin{center}
      \large by\\
      Antonio R. Vargas \\
      ${}$ \\
      \normalsize
      \texttt{antoniov@mathstat.dal.ca}
   \end{center}
   \vfill
   \begin{center}
      Submitted in partial fulfillment of the requirements \\
      for the degree of Master of Science \\[2.5ex]
      at \\[2.5ex]
      Dalhousie University \\
      Halifax, Nova Scotia \\
      August 2012
   \end{center}
   \vskip0.75in
   \begin{center}
      \rmfamily \copyright\ Copyright by Antonio R. Vargas, 2012
   \end{center}

\setcounter{secnumdepth}{3}
\setcounter{tocdepth}{3}

\frontmatter 

\begin{acknowledgements}      

First and foremost I would like to thank my girlfriend Amelia, for without her love and support this thesis would never have been conceivable.  Her shoulder has borne my full weight on my worst days.  I wish to thank my family for their unwavering affection and appreciation which I have not reciprocated nearly enough during my periods of study.  The same gratitude extends to my dear friends who I miss greatly.  I also want to thank my advisor Dr. Karl Dilcher for his patience and valuable editorial remarks.

\end{acknowledgements}




\begin{abstracts}        

For a power series which converges in some neighborhood of the origin in the complex plane, it turns out that the zeros of its partial sums---its sections---often behave in a controlled manner, producing intricate patterns as they converge and disperse.  We open this thesis with an overview of some of the major results in the study of this phenomenon in the past century, focusing on recent developments which build on the theme of asymptotic analysis.  Inspired by this work, we derive results concerning the asymptotic behavior of the zeros of partial sums of power series for entire functions defined by exponential integrals of a certain type.  Most of the zeros of the $n^{\text{th}}$ partial sum travel outwards from the origin at a rate comparable to $n$, so we rescale the variable by $n$ and explicitly calculate the limit curves of these normalized zeros.  We discover that the zeros' asymptotic behavior depends on the order of the critical points of the integrand in the aforementioned exponential integral.

Special cases of the exponential integral functions we study include classes of confluent hypergeometric functions and Bessel functions.  Prior to this thesis, the latter have not been specifically studied in this context.

\end{abstracts}



\tableofcontents

\chapter*{List of Symbols\hfill} \addcontentsline{toc}{chapter}{List of Symbols}
\begin{tabbing}
$\R$~~~~~~~~~~~~~\=\parbox{5in}{Field of real numbers\dotfill \pageref{symbol:RR}}\\
\addsymbol \C: {Field of complex numbers}{symbol:CC}
\addsymbol |z|: {Modulus of the complex number $z$}{symbol:modulus}
\addsymbol \Re(z): {Real part of the complex number $z$}{symbol:real}
\addsymbol \Im(z): {Imaginary part of the complex number $z$}{symbol:imag}

\addsymbol s_n(f;z): {$n^{\text{th}}$ partial sum of the Maclaurin series for $f$}{symbol:section}
\addsymbol t_n(f;z): {$n^{\text{th}}$ tail of the Maclaurin series for $f$}{symbol:tail}

\addsymbol \sharp_n^{\angle}(\theta_1,\theta_2): {Number of zeros of the $n^{\text{th}}$ section in the sector $\theta_1 \leq \arg z \leq \theta_2$}{symbol:numsec}
\addsymbol \sharp_{n}^{\circ}(R): {Number of zeros of the $n^{\text{th}}$ section in the disk $|z| \leq R$}{symbol:numdisk}
\addsymbol \{N\}: {An appropriate subsequence of the indices $\{n\}$}{symbol:Nindices}

\addsymbol \rho: {Order of an entire function}{symbol:rho}
\addsymbol \rho_n: {}{symbol:rhon}

\addsymbol \binom{n}{k}: {Coefficient of $x^k$ in the expansion of $(1+x)^n$}{symbol:binom}
\addsymbol \Gamma: {Gamma function}{symbol:gammafunc}
\addsymbol \log: {Natural logarithm}{symbol:log}

\addsymbol \longrightarrow: {Converges to}{symbol:arrow}
\addsymbol \approx: {Is approximately}{symbol:approx}
\addsymbol O(~~): {Big O notation}{symbol:bigo}
\addsymbol \sim: {Asymptotically equivalent}{symbol:sim}

\end{tabbing}
 \clearpage

\mainmatter 
\chapter{Introduction}
\label{intro}

\ifpdf
\graphicspath{{Chapter1/Chapter1Figs/PNG/}{Chapter1/Chapter1Figs/PDF/}{Chapter1/Chapter1Figs/}}
\else
\graphicspath{{Chapter1/Chapter1Figs/EPS/}{Chapter1/Chapter1Figs/}}
\fi

The story begins with G\'{a}bor Szeg\H{o}, a leading figure in the field of analysis of polynomials.  In 1924 he published a paper \cite{szego:exp} in which he examined the radial and angular distribution of the zeros of the partial sums---the sections---of the power series for the exponential function $\exp(z)$.  Let
\[
	s_n(\exp;z) = \sum_{k=0}^{n} \frac{z^k}{k!}
\]
be the $n^\text{th}$ such section.  The polynomial $s_n(\exp;z)$ has exactly $n$ complex zeros while the exponential function has none, so Hurwitz's theorem (see, e.g., \cite[p. 4]{marden:geom}) tells us that the zeros of $s_n(\exp;z)$ must move farther and farther away from the origin as $n$ goes to infinity.
\begin{figure}[htb]
	\centering
	\includegraphics[width=0.98\textwidth]{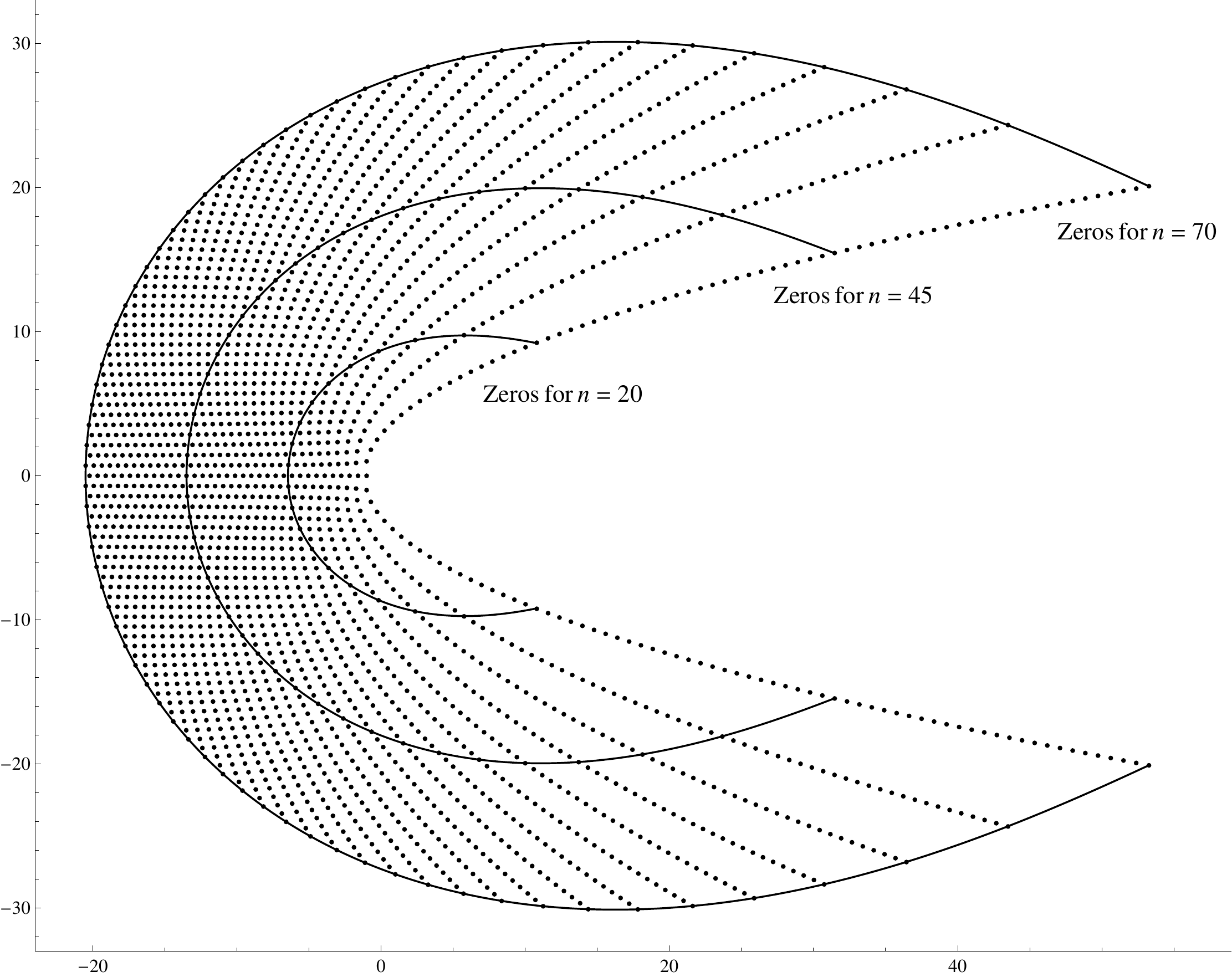}
	\caption{Zeros of the sections $s_n(\exp;z)$ ($n=1,2,\ldots,70$).  Lines have been added to indicate the zeros corresponding to $n=20$, $n=45$, and $n=70$.}
\label{ez_zeros}
\end{figure}
However, they move in a controlled manner: according to the Enestr\"om-Kakeya Theorem (see Theorem \ref{kakene} in Section \ref{prelims:strat}), the zeros of $s_n(\exp;z)$ all lie in the region $|z| \leq n$\phantomsection\newnot{symbol:modulus}.  Seeing this, Szeg\H{o} studied the behavior of the zeros of the polynomials $s_n(\exp;nz)$.  The zeros of these normalized sections all lie in the closed unit disk.

The primary result of Szeg\H{o}'s work was that the zeros of the normalized sections $s_n(\exp;nz)$ have as their limit points the simple closed loop
\begin{equation}
	D = \left\{z \in \C \,\colon |z| \leq 1 \,\,\,\text{and}\,\,\, \left|z e^{1-z}\right| = 1 \right\}.
\label{szegocurve}
\end{equation}
This \newnot{symbol:CC} curve, and its analogues for other power series, is usually referred to as the Szeg\H{o} curve.

\begin{figure}[htb]
	\centering
	\includegraphics[width=0.98\textwidth]{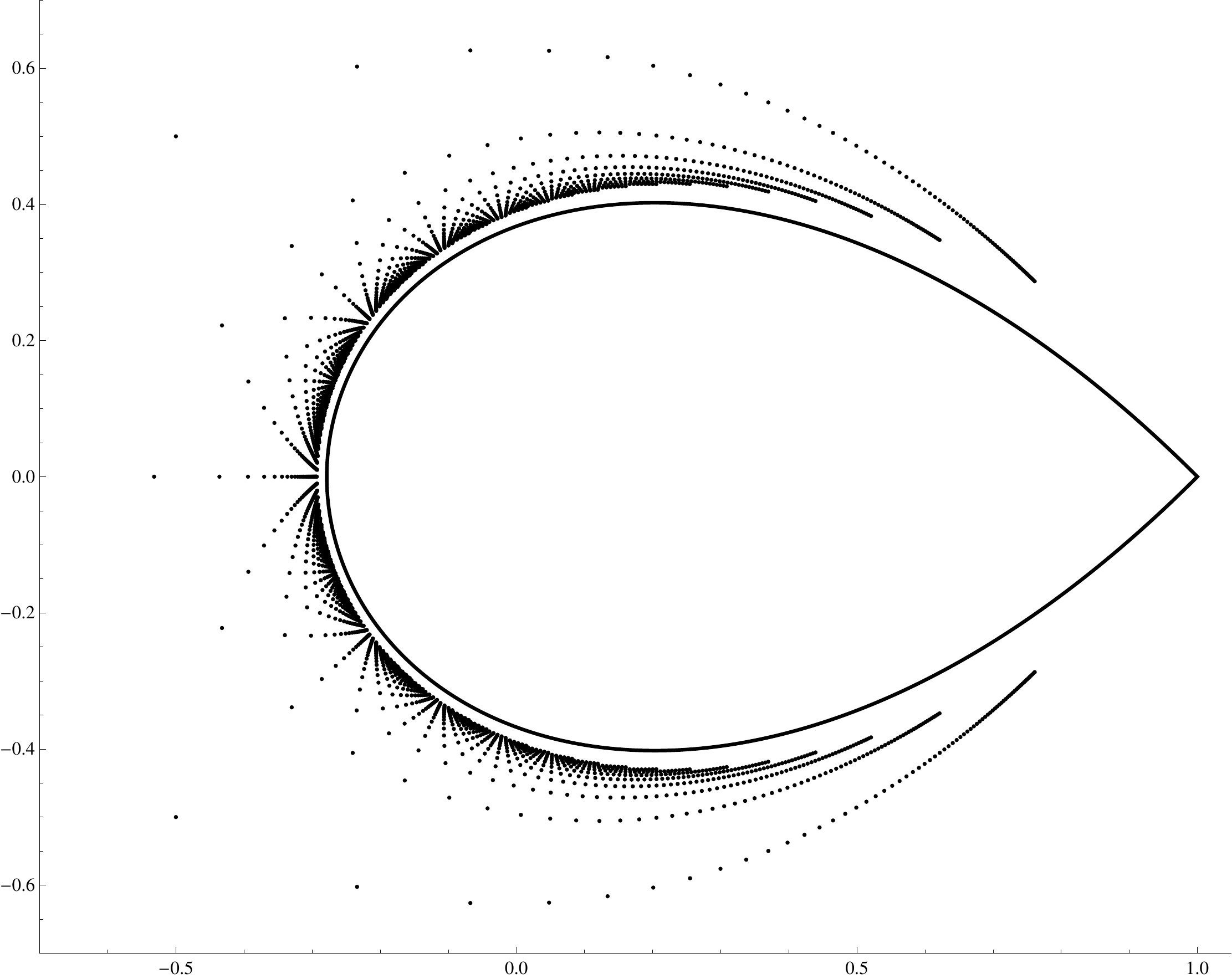}
	\caption{Zeros of the normalized sections $s_n(\exp;nz)$ ($n=1,2,\ldots,70$) with the Szeg\H{o} curve $D$ in equation \eqref{szegocurve}.}
\label{ez_zeros_norms}
\end{figure}

Szeg\H{o} also studied the angular distribution of the zeros.  He first showed that the mapping $w = z e^{1-z}$ takes the curve $D$ to the unit circle in the $w$-plane in a regular manner: $\arg w$ increases monotonically from $0$ to $2 \pi$ as $z$ traverses $D$ from $z = 1$ in the counterclockwise direction.  So, if $\sharp_n^{\angle}(\theta_1,\theta_2)$\phantomsection\newnot{symbol:numsec} is the number of zeros of $s_n(\exp;z)$ in the sector $\theta_1 \leq \arg z \leq \theta_2$ and if $z_1$ and $z_2$ are the points of $D$ with arguments $\theta_1$ and $\theta_2$, respectively, then
\[
	\lim_{n \to \infty} \frac{\sharp_n^{\angle}(\theta_1,\theta_2)}{n} = \frac{w(z_2) - w(z_1)}{2 \pi}.
\]
Essentially this says that, modulo the weight function $w$, the zeros of the sections are uniformly radially distributed.

Though in this work we are only concerned with the zeros of sections of power series, Szeg\H{o} studied the more general question of the roots of the equation
\begin{equation}
	s_n(\exp;nz) = \lambda e^{nz},
\label{szegolincomb}
\end{equation}
where $0 \leq \lambda \leq 1$.  For $\lambda \neq 0,1$ the roots are no longer restricted to $|z| \leq 1$ and may accumulate on any part of the curve $\left|ze^{1-z}\right| = 1$.  For $\lambda = 1$ the roots accumulate on the ``arms'' of this curve, i.e. the points on $\left|ze^{1-z}\right| = 1$ with $|z| \geq 1$.  See Figure \ref{szegocurveplot} for a view of the unrestricted curve $\left|ze^{1-z}\right| = 1$.

\begin{figure}[htb]
	\centering
	\includegraphics[width=0.6\textwidth]{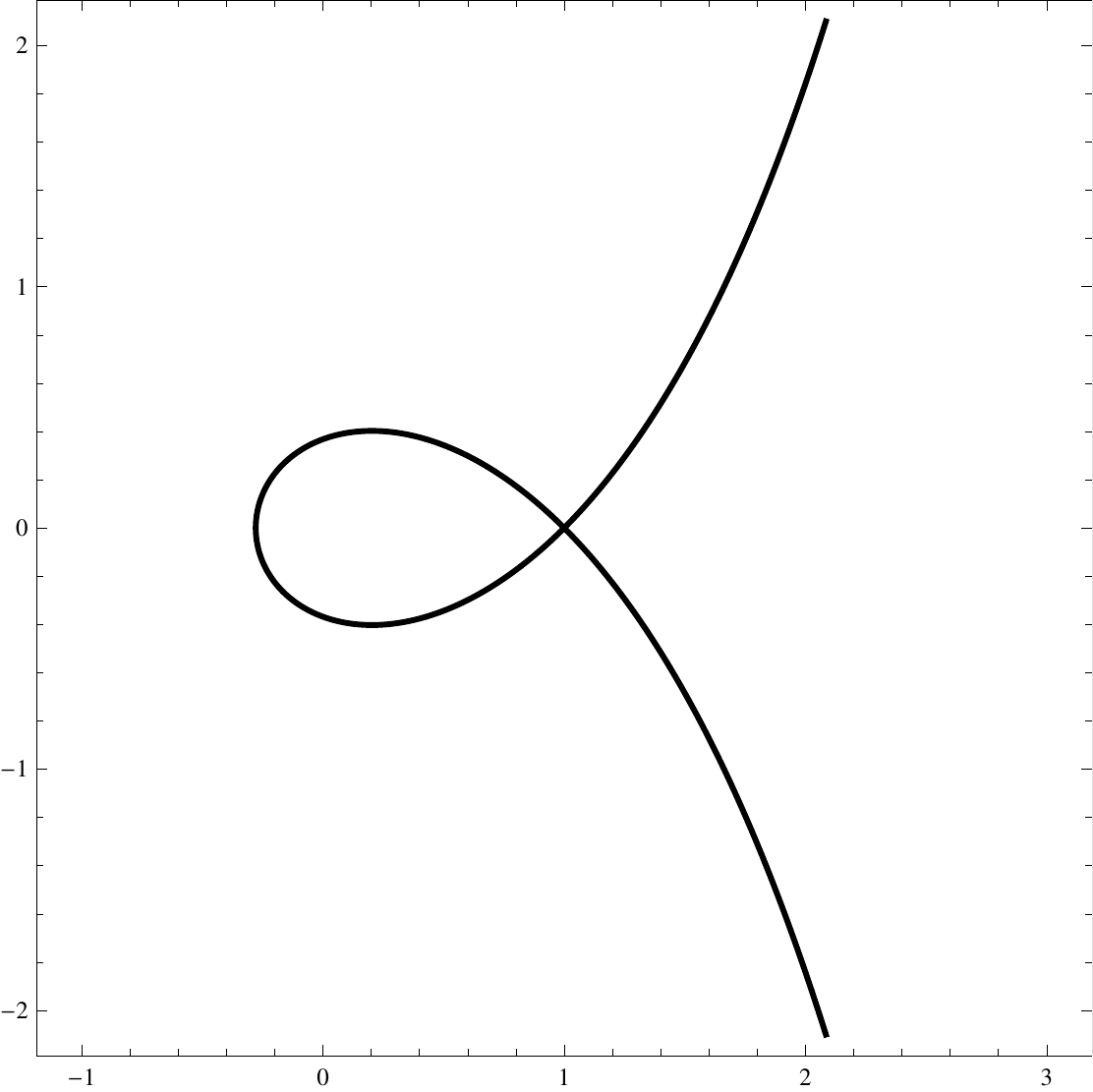}
	\caption{Part of the curve $\left|z e^{1-z}\right|=1$.}
\label{szegocurveplot}
\end{figure}

Szeg\H{o} also studied the behavior of the analogous question for the power series for sine and cosine.  By studying the roots of the equation
\[
	\Bigl(s_n(\exp;inz)-\lambda e^{inz}\Bigr) + \Bigl(s_n(\exp;-inz)-\lambda e^{-inz}\Bigr) = 0
\]
he was able to apply to these series what he had discovered about the exponential series.  In particular he deduced that the set of limit points of the zeros of the sections $s_n(\sin;nz)$ and $s_n(\cos;nz)$ consists of two rotated copies of the part of the limit curve D in the right half-plane.  Indeed, if\phantomsection\newnot{symbol:real} $D^+ = D \cap \{z \in \C \,\colon \Re(z) \geq 0\}$ then the limit curve associated with the sine and cosine series is the set\phantomsection\newnot{symbol:RR}
\[
	iD^+ \cup -iD^+ \cup \left\{x \in \R \,\colon \!-1/e \leq x \leq 1/e\right\}.
\]
Figure \ref{coszeros} illustrates the convergence of the zeros of the normalized sections of the cosine function to this curve.

\begin{figure}[h!tb]
	\centering
	\begin{tabular}{cc}
		\includegraphics[width=0.45\textwidth]{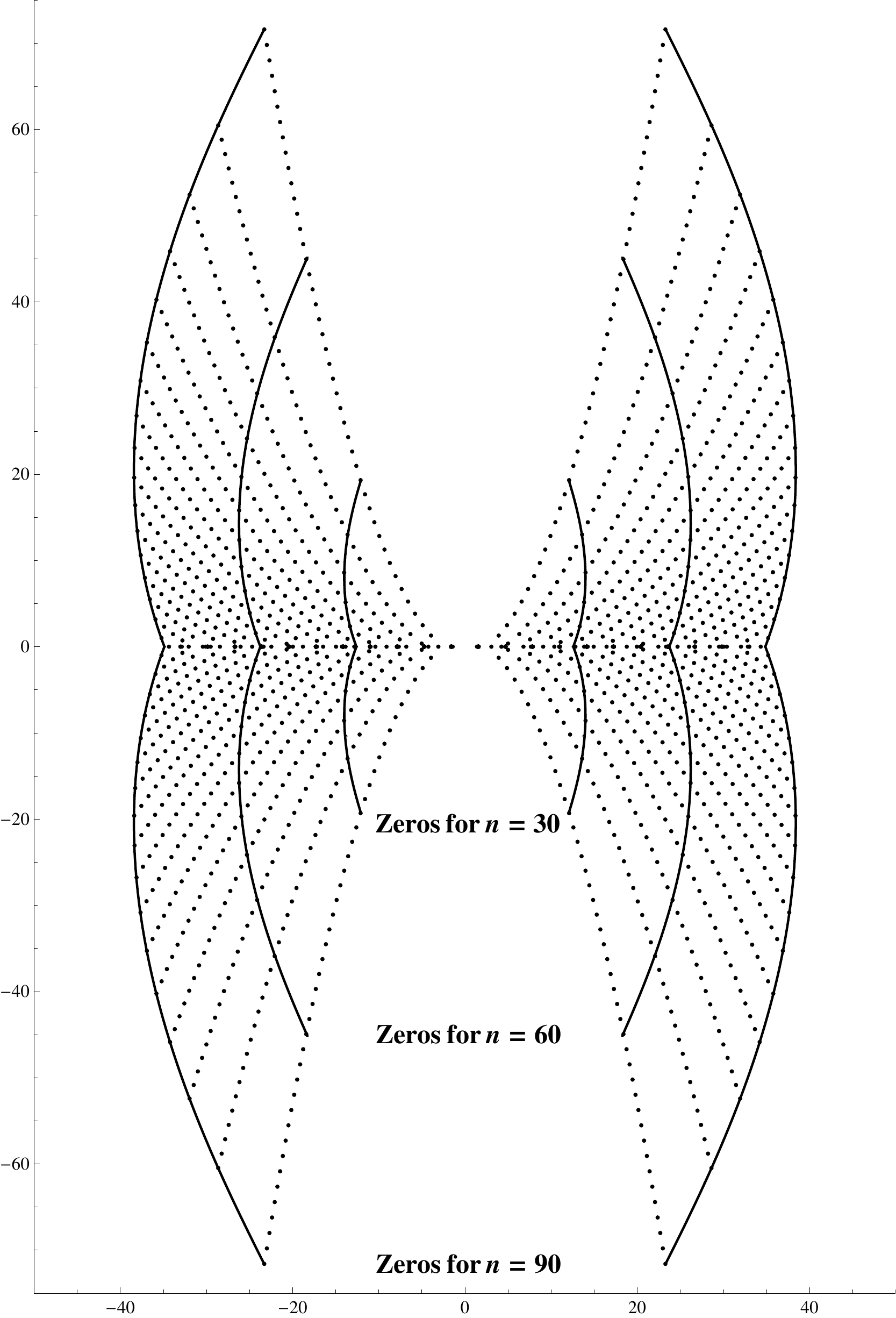}
			& \includegraphics[width=0.45\textwidth]{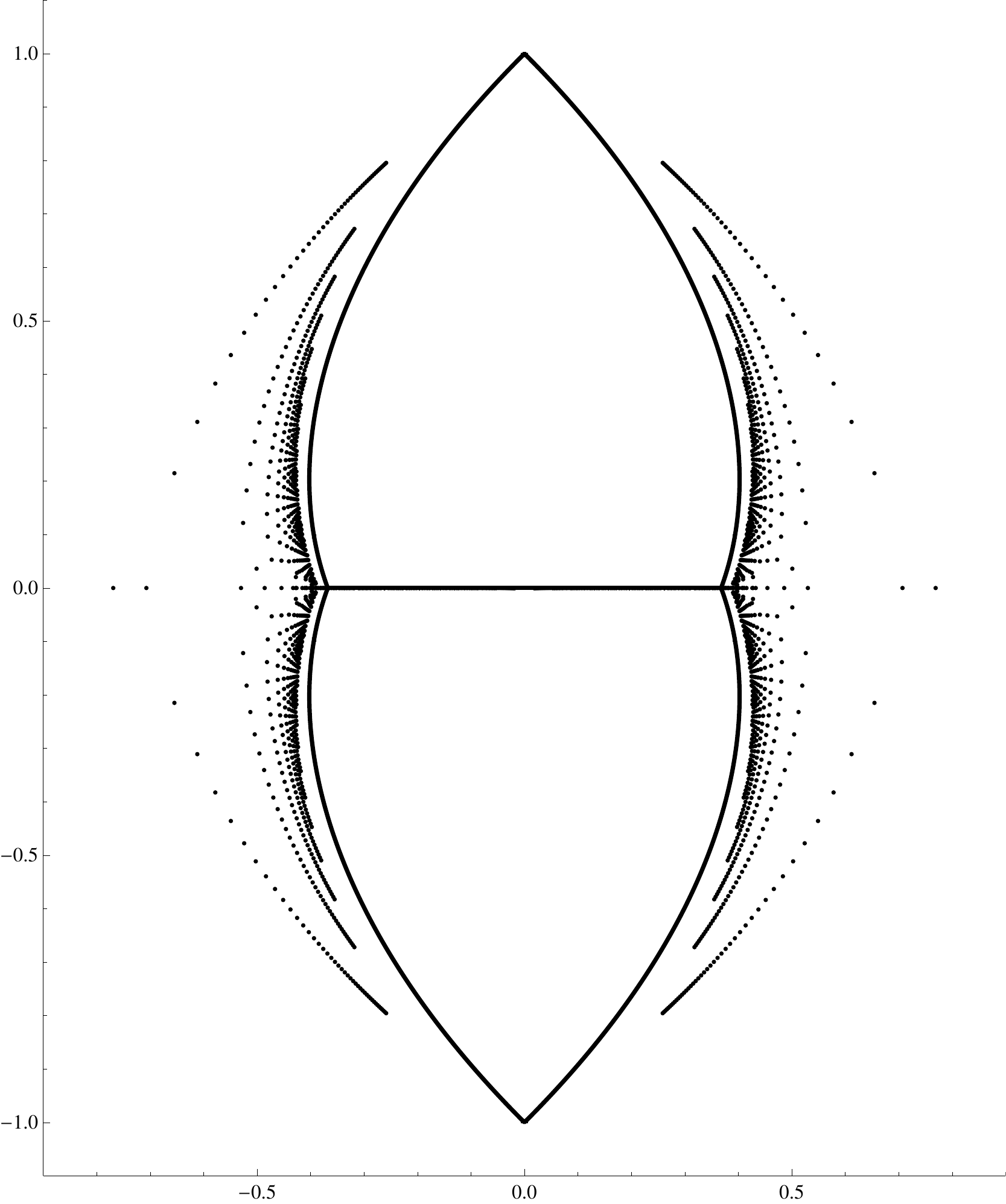}
	\end{tabular}
	\caption{LEFT: Zeros of the sections $s_n(\cos;z)$ $(n=2,4,\ldots,90)$.  Lines have been added to indicate the zeros corresponding to $n=30$, $n=60$, and $n=90$.  RIGHT: Zeros of the normalized sections $s_n(\cos;nz)$ $(n=2,4,\ldots,90)$ with their Szeg\H{o} curve.}
\label{coszeros}
\end{figure}

Most of the results in Szeg\H{o}'s paper were rediscovered by Dieudonn\'e \cite{dieu:expsections} in 1935.

In 1944, Paul C. Rosenbloom showed in his doctoral thesis \cite{rosen:thesis} (and summarized in a separate paper \cite{rosen:distrib}) that the behavior described by Szeg\H{o} is in fact a generic property of entire functions of positive finite order (see Section \ref{ordersection}) with a certain asymptotic character.  To state this result we will require a small amount of notation.

Let $f$ be an entire function of order $0 < \rho < \infty$\phantomsection\newnot{symbol:rho} with
\[
	f(z) = \sum_{k=0}^{\infty} a_k z^k
\]
and let
\[
	s_n(f;z) = \sum_{k=0}^{n} a_k z^k
\]
be its $n^{\text{th}}$ section.  The basic conclusion Rosenbloom came to is that most of the zeros of the sections grow on the order of\phantomsection\newnot{symbol:rhon}
\[
	 \rho_n = |a_n|^{-1/n}.
\]
As such, he considers the zeros of the scaled sections $s_n(f;\rho_n z)$ (another approach for determining the appropriate scale factor is outlined in Section \ref{mloutline}). 

Let $\{N\}$ be a subsequence of the indices $\{n\}$ such that the sequence of sections $\{s_N(f;z)\}$ has a positive fraction of zeros in any sector with vertex at the origin.  That is, if $\sharp_N^{\angle}(\theta_1,\theta_2)$ denotes the number of zeros of the section $s_N(f;z)$ in the sector $\theta_1 \leq \arg z \leq \theta_2$, then
\[
	\liminf_{N \to \infty} \frac{\sharp_N^{\angle}(\theta_1,\theta_2)}{N} > 0
\]
for any fixed $\theta_1$ and $\theta_2$.  Such a subsequence $\{N\}$ is guaranteed to exist by Theorem \ref{rosentheo}.  Rosenbloom's main result is as follows.

\begin{theorem}[Rosenbloom]
Suppose that the following conditions hold:
\begin{enumerate}[label=(\arabic*)]
\item For some sequence of determinations, $f(\rho_N z)^{1/N}$ converges uniformly to a single-valued analytic function $g$ in some subdomain $X$ of the disk $|z| \leq e^{1/\rho}$;
\item $w = g(z)/z$ maps $X$ univalently onto a domain $X_1$;
\item No limit function of the sequence
\[
	T_N(z) = \frac{f(\rho_N z) - s_N(f;\rho_N z)}{z^N}
\]
is identically zero in $X$;
\item $T_N(z) \neq 0$ in $X$ for $N$ large enough.
\end{enumerate}
Then the only limit points of the zeros of $s_N(f;\rho_N z)$ in $X$ are the points on the curve $|g(z)/z| = 1$, and their images in $X_1$ under the mapping $w = g(z)/z$ are equidistributed about the unit circle $|w| = 1$; that is, the number which accumulate about any arc of length $\alpha$ contained in $X$ is asymptotically $N\alpha/2\pi$.
\label{mainrosentheo}
\end{theorem}

One interesting aspect of Rosenbloom's result is that it allows for the sequence $f(\rho_N z)^{1/N}$ to converge to different limit functions $g$ in different subregions of the disk $|z| \leq e^{1/\rho}$.  One striking example of this is the behavior of the zeros of sections of power series for exponential sums of the form
\[
	\sum_{k=1}^{p} c_j e^{\lambda_j z},
\]
where $c_j,\lambda_j \in \C$, as studied by Pavel Bleher and Robert Mallison, Jr.~\cite{mallison:expsums}.  The zeros behave differently in different sectors which are determined by the geometric properties of the parameters $\lambda_j$, as can be seen in Figure \ref{expsumplots} for the function
\begin{equation}
	f(z) = 3 e^{(8+2i)z} + (-9+12i) e^{(4+7i)z} + (2+i) e^{(-7+4i)z} - 5 e^{(-6-6i)z}.
\label{expsumeq}
\end{equation}

\begin{figure}[h!tb]
	\centering
	\begin{tabular}{cc}
		\includegraphics[width=0.45\textwidth]{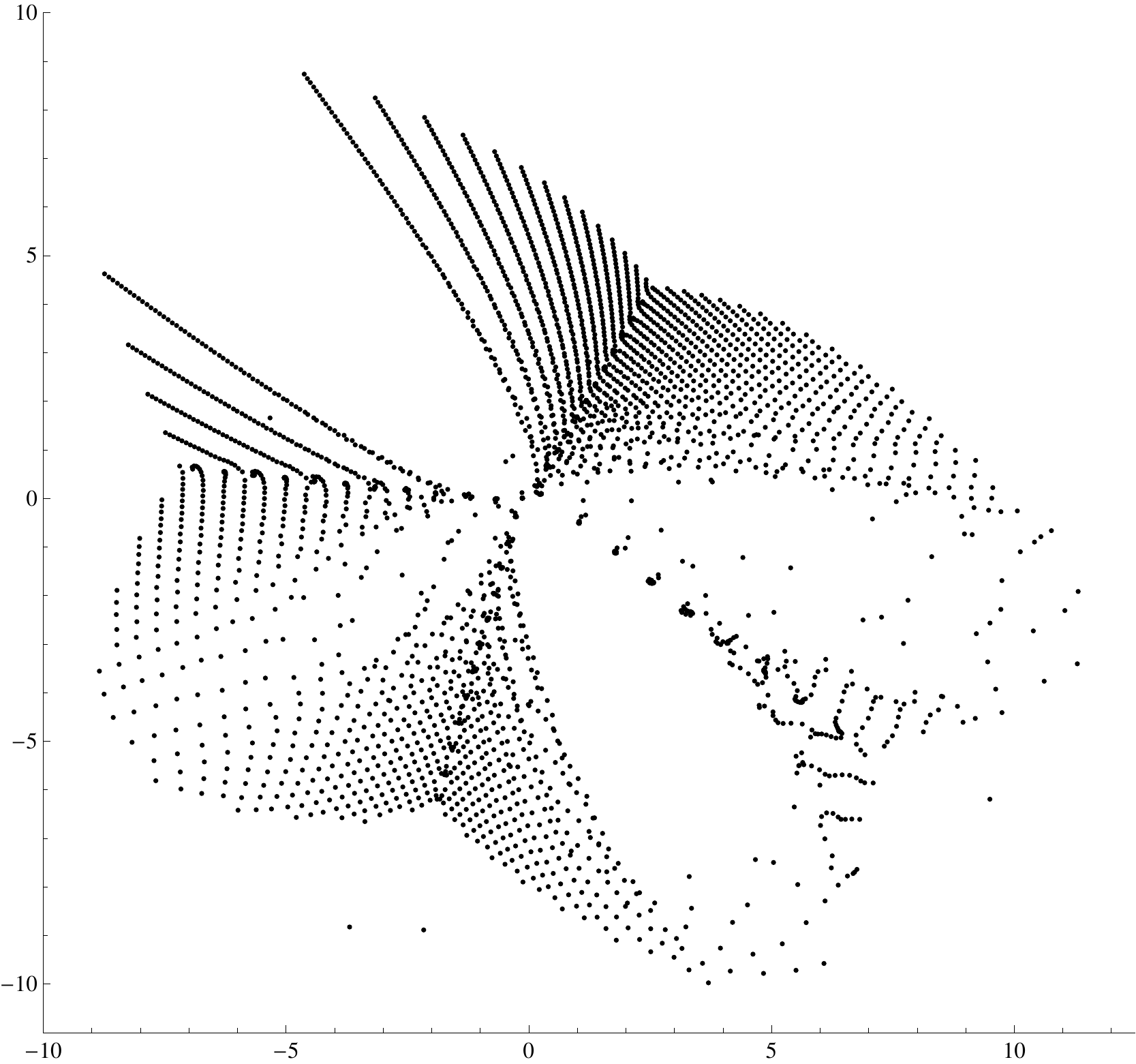}
			& \includegraphics[width=0.45\textwidth]{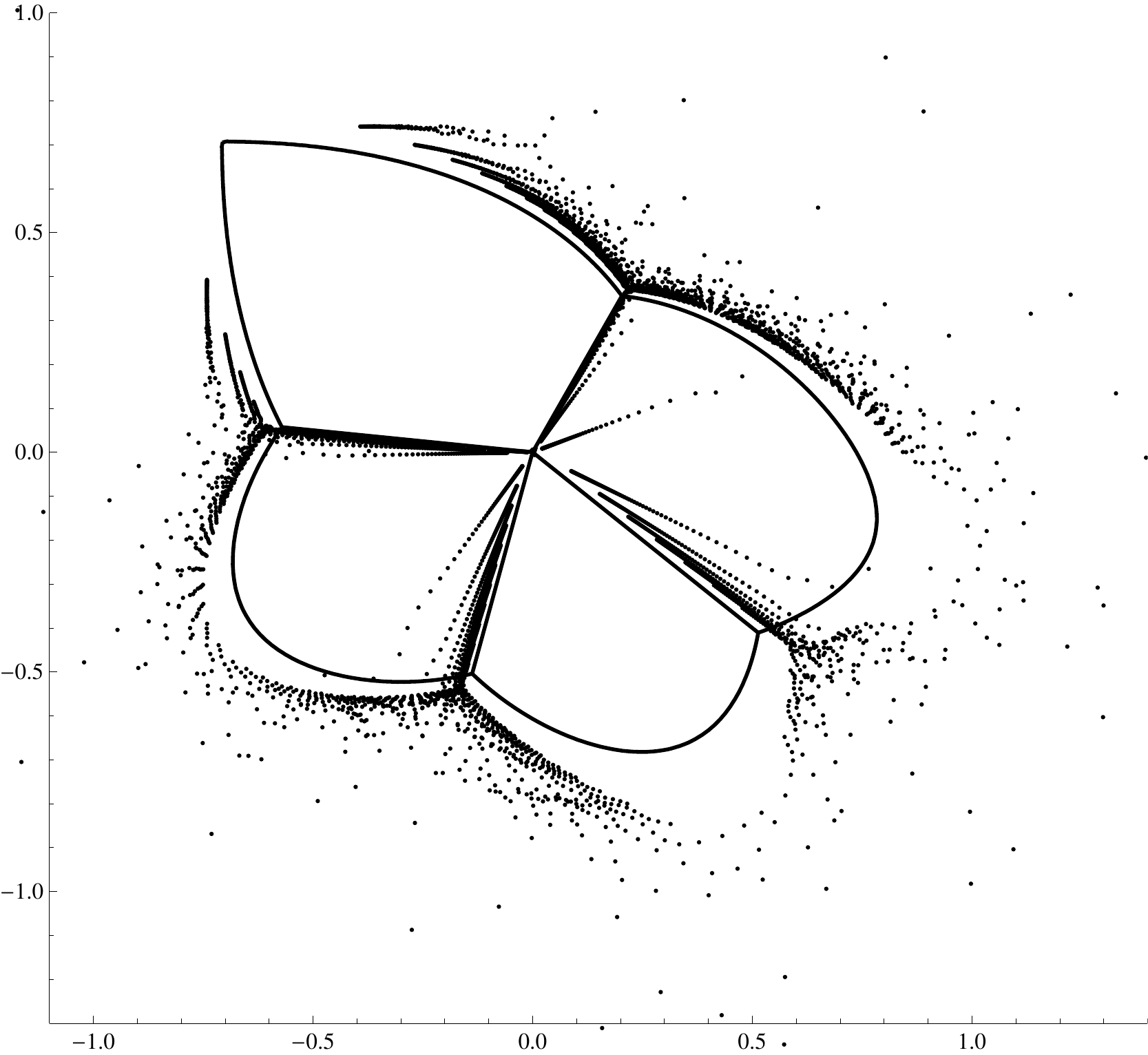}
	\end{tabular}
	\caption{LEFT: Zeros of the first $100$ sections of the exponential sum in equation \eqref{expsumeq}.  RIGHT: Zeros of the first $100$ normalized sections with their Szeg\H{o} curve.}
\label{expsumplots}
\end{figure}

We discuss how the results of Rosenbloom, Bleher, and Mallison relate to the ones obtained in this thesis in Chapter \ref{relev}.

\section{A Nudge Toward Asymptotic Analysis}

James D. Buckholtz was the first to talk about the radial position of the zeros of $s_n(\exp;z)$ in relation to the Szeg\H{o} curve $D$ in equation \eqref{szegocurve}.  He published a short paper \cite{buckholtz:expcharacter} on the subject in 1966.  His first observation was that all of the zeros of the sections lie outside the Szeg\H{o} curve.  The proof is so short and elegant that we will include it here.

\begin{theorem}[Buckholtz]
For every positive integer $n$, neither the curve $D$ nor the region it encloses contains a zero of $s_n(\exp;nz)$.
\label{buck1}
\end{theorem}

\begin{proof}
Let $z \in \C$ with $|z| \leq 1$ and $\left|z e^{1-z}\right| \leq 1$.  Then
\begin{align*}
\left|1 - e^{-nz} s_n(\exp;nz)\right| &= \left|e^{-nz} \sum_{k=n+1}^{\infty} \frac{n^k z^k}{k!} \right| \\							&= \left|\left(z e^{1-z}\right)^n e^{-n} \sum_{k=n+1}^{\infty} \frac{n^k z^{k-n}}{k!} \right| \\
			&\leq e^{-n} \sum_{k=n+1}^{\infty} \frac{n^k}{k!} \\
			&= 1 - e^{-n} s_n(\exp;n) \\
			&< 1.
\end{align*}
Having $s_n(\exp;nz) = 0$ here would contradict this inequality.
\end{proof}

Buckholtz's result is readily seen in Figure \ref{ez_zeros_norms}.

The proof of this theorem can give a slightly more general result.  For $n$ a positive integer, if $f$ is analytic in $|z| \leq n$ with
\[
	f(z) = \sum_{k=0}^{\infty} a_k z^k,
\]
let
\[
	s_n(f;z) = \sum_{k=0}^{n} a_k z^k
\]
be its $n^{\text{th}}$ section.  The method of Buckholtz shows that, if $s_n(f;n)/f(n) > 0$, then $s_n(f;nz)$ has no zeros in the region
\[
	S_n = \left\{z \in \C \,\colon |z| \leq 1,\,\,\, f(nz) \neq 0, \,\,\,\text{and}\,\,\, \left|z^n \frac{f(n)}{f(nz)} \right| \leq 1 \right\}.
\]
Note that when $f(z) = e^z$ we have $S_n = \left\{ z \in \C\,\colon |z| \leq 1 \,\,\,\text{and}\,\,\, \left|z e^{1-z}\right| \leq 1 \right\}$ for all positive integers $n$.

In addition to describing the direction from which the zeros approach the limit curve, Buckholtz used a result from a previous paper of his \cite{buckholtz:copapprox} to examine the rate at which they do so.  Refining this result would become the central focus of later work on the topic.

\begin{theorem}[Buckholtz]
For every positive integer $n$, all zeros of $s_n(\exp;nz)$ lie within a distance of $2e/\!\sqrt{n}$ of $D$.
\label{buck2}
\end{theorem}

\section{The Contribution of Newman and Rivlin}

In 1972, Donald J. Newman and Theodore J. Rivlin published a paper \cite{newriv:expzeros} in which they aimed to establish a zero-free parabolic region for the sections $s_n(\exp;z)$ described above.  However, there was an error in their proof, and so they did not actually achieve this goal until their correction \cite{newriv:expzeroscorrect} was published in 1976.  They showed that, if $c$ is any positive number satisfying $c e^c < \pi/2$, then there is no zero in the region $\left\{ z = x + iy\,\colon y^2 \leq c x \right\}$ (for related results see, e.g., \cite{sv:pade}, \cite{sv:zerofree}, \cite{sv:zerofreesharp}, and \cite{sv:zerofreesharperratum}).

Though the first paper may not have served its original purpose, the following theorem has become important to the theory we're concerned with.

For this result we will require the complementary error function $\erfc$, defined by
\[
	\erfc(z) = \frac{2}{\sqrt{\pi}} \int_{z}^{\infty} e^{-t^2} \,dt,
\]
where the path of integration begins at $z$ and travels to the right to $\infty$.

\begin{theorem}[Newman and Rivlin]
For $n > 0$, define the functions
\[
	h_n(w) = \int_{w}^{\infty}\left(1+\frac{\zeta}{\sqrt{n}}\right)^n e^{-\sqrt{n} \zeta} \,d\zeta,
\]
where the path of integration begins at $w$ and travels to the right to $\infty$.  Then the sequence $(h_n)$ converges uniformly to the function
\[
	H(w) = \int_{w}^{\infty} e^{-\zeta^2/2} \,d\zeta = \sqrt{\frac{\pi}{2}} \erfc\!\left(\frac{w}{\sqrt{2}}\right)
\]
on any compact subset of $\Im(w) \geq 0$,\phantomsection\newnot{symbol:imag} where $\erfc$ is the complementary error function.
\label{nrtheo}
\end{theorem}

The motivation for this result comes from rewriting $s_n(\exp;z)$ in a form which reveals the nature of the ``parabolic'' arcs of zeros seen in Figure \ref{ez_zeros}.  First, repeated integration by parts will verify that
\[
	s_n(\exp;z) = \int_0^\infty \frac{(z+t)^n}{n!} \,e^{-t}\,dt.
\]
Putting $z = n+w\sqrt{n}$ and using the substitution $\zeta = w+t/\sqrt{n}$ we have
\[
	\frac{s_n\!\left(\exp;n+w\sqrt{n}\right)}{e^{n+w\sqrt{n}}} = \frac{\sqrt{2 \pi n} (n/e)^n}{n!} \cdot \frac{1}{\sqrt{2 \pi}} \int_{w}^{\infty}\left(1+\frac{\zeta}{\sqrt{n}}\right)^n e^{-\sqrt{n} \zeta} \,d\zeta,
\]
where the path of integration is the horizontal line from $w$ to the right to $\infty$.  Note that the integral on the right is the function $h_n(w)$ defined above.  The conclusion of Theorem \ref{nrtheo} can thus be stated as
\begin{equation}
\frac{s_n\!\left(\exp;n+w\sqrt{n}\right)}{e^{n+w\sqrt{n}}} \longrightarrow \frac{1}{2} \erfc\!\left(\frac{w}{\sqrt{2}}\right)
\label{experfc}
\end{equation}
as \newnot{symbol:arrow} $n \to \infty$ when $w$ is restricted to a compact subset of $\Im(w) \geq 0$.  Now if $h_n$ has a zero $w = u + iv$ with $v \neq 0$, then $z = x+iy = n + (u+iv)\sqrt{n}$ is a zero of $s_n(\exp;z)$ which lies on the parabola
\begin{equation}
	x = (y/v)^2 + u(y/v).
\label{ez_parab}
\end{equation}
But if $w = u+iv$ is any zero of the limit function $\erfc\!\left(w/\sqrt{2}\right)$ in the upper half-plane, Hurwitz's theorem tells us that $h_n$ will have a zero near $w$ when $n$ is large enough, so that $s_n(\exp;z)$ will have a zero arbitrarily close to the parabola \eqref{ez_parab}.  In other words, the arcs of zeros seen in Figure \ref{ez_zeros} will tend toward parabolas of the form $x = (y/v)^2 + u(y/v)$, where $w = u+iv$ is a zero of $\erfc\!\left(w/\sqrt{2}\right)$.

This behavior can be seen in Figure \ref{ez_zeros_parab}.  There, $w=u+iv$ is chosen to be the smallest zero of $\erfc\!\left(w/\sqrt{2}\right)$ in the upper half-plane.  The parabola associated with this zero approximates the upper half of the innermost arc of zeros of $s_n(\exp;z)$.

\begin{figure}[htb]
	\centering
	\includegraphics[width=5.4in]{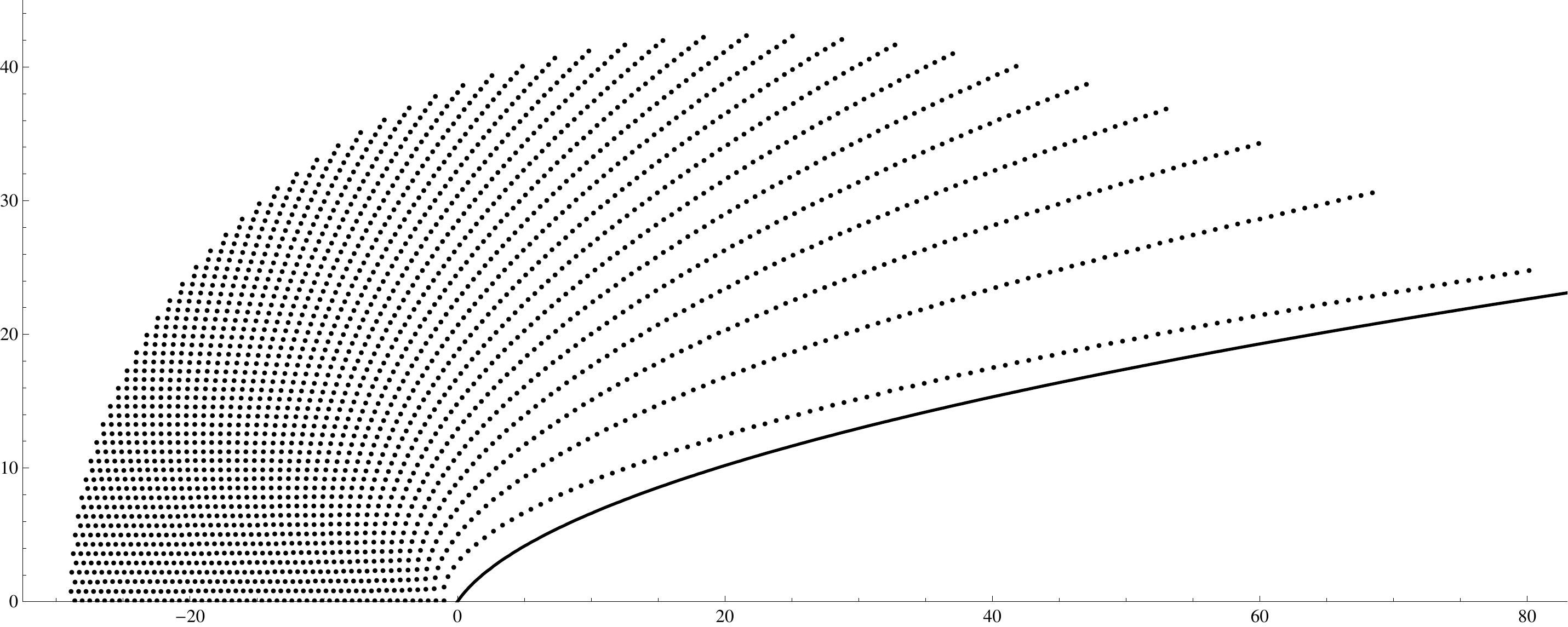}
	\caption{Zeros of $s_n(\exp;z)$ ($n=1,2,\ldots,100$) and the parabola \eqref{ez_parab} associated with the zero of smallest modulus of $\erfc\!\left(w/\sqrt{2}\right)$ in the upper half-plane.}
\label{ez_zeros_parab}
\end{figure}

The result in equation \eqref{experfc} has inspired some similar results for other functions.  In \cite{esv:sections} it is shown that
\[
	\frac{s_n\!\left(E_{1/\lambda}; R_n \left(1 + w\sqrt{\frac{2}{\lambda n}} \right)\right)}{\left(1 + w\sqrt{\frac{2}{\lambda n}} \right)^n  E_{1/\lambda}(R_n)} \longrightarrow \frac{e^{w^2}}{2} \erfc(w)
\]
as $n \to \infty$ uniformly when $w$ is restricted to a compact subset of $\C$, where $E_{1/\lambda}$ is the Mittag-Leffler function of order $\lambda$ and $R_n$ is its associated scale factor, both of which are described in Section \ref{ordersection}.  A similar result is proved for the ${\mathcal L}$-functions of order $0 < \lambda < 1$, which are described in Section \ref{orderrelated}. In \cite{norfolk:1f1} it is proved that
\[
	\frac{s_n({}_1F_1,b;n+w\sqrt{n})}{e^{w\sqrt{n}} {}_1F_1(1;b;n)} \longrightarrow \frac{1}{2} \erfc\!\left(\frac{w}{\sqrt{2}}\right)
\]
as $n \to \infty$ uniformly when $w$ is restricted to a compact subset of $\C$, where ${}_1F_1(1;b;z)$ is a confluent hypergeometric function and $s_n({}_1F_1,b;z)$ is its $n^{\text{th}}$ section---both of which are described in Section \ref{norfolkinspiration}.  Lastly in \cite{norfolk:binom} we are given that\phantomsection\newnot{symbol:approx}
\[
	\left(1-\frac{r}{n}\right)^n B_{r,n}\!\left(\frac{re^{w/\sqrt{r-r^2/n}}}{n-r}\right) e^{-rw/\sqrt{r-r^2/n}-w^2/2} \approx \frac{1}{2} \erfc\!\left(\frac{w}{\sqrt{2}}\right)
\]
when $r$ and $n$ are large with $\delta < r/n < 1-\delta$ for some $\delta > 0$, where\phantomsection\newnot{symbol:binom}
\[
	B_{r,n}(z) = \sum_{k=0}^{r} \binom{n}{k} z^k
\]
for $1 \leq r < n$.

These results are used to determine the widths of the zero-free regions---the ``openings''---exemplified in Figures \ref{ez_zeros}, \ref{ml_plots}, and \ref{1f1zeros}.

\section{The Role of the Order of the Limit Function}
\label{ordersection}

Albert Edrei, Edward B. Saff, and Richard S. Varga studied the effect of the order of the limit function on the zeros of the sections.  In 1983 they published a monograph \cite{esv:sections} in which they examined the asymptotic character of the zeros of the sections of the Mittag-Leffler functions $E_{1/\lambda}$ defined by\phantomsection\newnot{symbol:gammafunc}
\[
	E_{1/\lambda}(z) = \sum_{k=0}^{\infty} \frac{z^k}{\Gamma\!\left(\frac{k}{\lambda}+1\right)},
\]
where $0 < \lambda < \infty$.

Recall that the \textit{order} of an entire function $f$---that is, a function which is analytic on the entire complex plane $\C$---is defined to be the infimum of all real numbers $\ell$ for which
\[
	|f(z)| \leq \exp(|z|^{\ell})
\]
holds for $|z|$ large enough.  Writing
\[
	f(z) = \sum_{k=0}^{\infty} a_k z^k,
\]
we can calculate the order of $f$ directly with the formula\phantomsection\newnot{symbol:log}
\[
	\ell = \limsup_{k \to \infty} \frac{k \log k}{\log (1/|a_k|)}
\]
(see e.g. \cite[p. 9]{boas:entirefunctions} or \cite[p. 326]{saks:analyticfunctions}).

The order of the function $E_{1/\lambda}$ is seen to be $\lambda$.  As such, $E_{1/\lambda}$ is called the Mittag-Leffler function of order $\lambda$.

\subsection{Outline of the Method}
\label{mloutline}

We will outline here the approach used in the monograph.

The first step is to determine an appropriate scale factor for the sections.  Here let
\[
	a_k = \frac{1}{\Gamma\!\left(\frac{k}{\lambda}+1\right)}
\]
be the coefficient of $z^k$ in the power series, and thus let
\[
	s_n(E_{1/\lambda};z) = \sum_{k=0}^{n} a_k z^k = \sum_{k=0}^{n} \frac{z^k}{\Gamma\!\left(\frac{k}{\lambda}+1\right)}.
\]
be the $n^{\text{th}}$ section of the series.  We wish to choose the scale factor $R_n$ so that the inequality
\begin{equation}
	a_n R_n^n \geq a_k R_n^k
\label{centind}
\end{equation}
holds for all nonnegative integers $k$.  For the function $E_{1/\lambda}$, the sequence \linebreak $\{a_{k-1}/a_k\}_{k=1}^{\infty}$ is strictly increasing, so we just need to choose $R_n$ to satisfy the inequality
\begin{equation}
	\frac{a_{n-1}}{a_n} \leq R_n < \frac{a_n}{a_{n+1}}.
\label{centind2}
\end{equation}
Once we do this, inequality \eqref{centind} will in turn be satisfied.  By using Stirling's formula for the gamma function, we deduce the approximations\phantomsection\newnot{symbol:bigo}
\[
	\log\!\left(\frac{a_{n-1}}{a_n}\right) = \frac{1}{\lambda} \log\!\left(\frac{n}{\lambda}\right) + \frac{1}{2n}\left(1-\frac{1}{\lambda}\right) + O\!\left(1/n^2\right),
\]
\[
	\log\!\left(\frac{a_n}{a_{n+1}}\right) = \frac{1}{\lambda} \log\!\left(\frac{n}{\lambda}\right) + \frac{1}{2n}\left(1+\frac{1}{\lambda}\right) + O\!\left(1/n^2\right),
\]
as $n \to \infty$, so if we choose $R_n$ such that
\[
	\log R_n = \frac{1}{\lambda} \log\!\left(\frac{n}{\lambda}\right) + \frac{1}{2n},
\]
we will indeed satisfy \eqref{centind2} for $n$ large enough.

Define the functions
\begin{align*}
	& U_n(z) = \frac{E_{1/\lambda}(R_n z)}{a_n\,(R_n z)^n}, \\
	& Q_n(z) = \frac{s_n(E_{1/\lambda};R_n z)}{a_n\,(R_n z)^n} = \sum_{k=1}^{n} b_{-k}(n) \,z^{-k}, \\
	& G_n(z) = \sum_{k=1}^{\infty} \frac{a_{n+k}}{a_n}\,(R_n z)^k = \sum_{k=1}^{\infty} b_k(n) \,z^k,
\end{align*}
where
\[
	b_k(n) = \frac{a_{n+k}}{a_n} R_n^k
\]
for $k \geq -n$.  By our choice for $R_n$ and properties of the gamma function we see that $b_k(n) \to 1$ as $n \to \infty$ for any fixed $k$.  It follows that
\[
	G_n(z) \longrightarrow \frac{z}{1-z}
\]
as $n \to \infty$ uniformly on compact subsets of the open unit disk.

By our definitions we have
\[
	Q_n(z) = U_n(z) - G_n(z).
\]
As a consequence of the Enestr\"om-Kakeya Theorem (see Theorem \ref{kakene} in Section \ref{prelims:strat}), the function $Q_n$ has no zeros outside of the open unit disk.  If $n$ is large then $G_n(z)$ is approximately equal to $z/(1-z)$ for $|z|<1$, so if we can find an asymptotic representation of the limit function in question, $E_{1/\lambda}$, we can describe the asymptotic character of the zeros of $Q_n$.  These zeros are exactly the zeros of the sections $s_n(E_{1/\lambda};R_n z)$.

\subsection{Szeg\H{o} Curves for the Mittag-Leffler Functions}

The definition of the Szeg\H{o} curve for $E_{1/\lambda}$ is not as easy to state as the one for the exponential function; we will need to define it differently in different sectors.  The curve is defined as the set of all points $z = r(\theta) \,e^{i \theta}$ satisfying
\begin{enumerate}[leftmargin=1.6cm,label=(\roman*)]
\item $-\frac{\pi}{2 \lambda} \leq \theta \leq \frac{\pi}{2 \lambda}\,\,\colon\,\,r(\theta)$ is the unique solution of the equation
\[
	r(\theta)^{\lambda} \cos(\lambda \theta) - 1 - \lambda \log r(\theta) = 0
\]
in the interval $e^{-1/\lambda} \leq r(\theta) \leq 1$,
\item $\frac{\pi}{2 \lambda} < \theta < 2\pi - \frac{\pi}{2 \lambda}\,\,\colon\,\,r(\theta) = e^{-1/\lambda}$.
\end{enumerate}
Thus the Szeg\H{o} curve for $E_{1/\lambda}$ consists of a circular part and another part whose radial component is defined implicitly in terms of its argument.  Note that the order of the function, $\lambda$, plays a major role in the definition of the curve.

\begin{figure}[h!tb]
	\centering
	\begin{tabular}{cc}
		\includegraphics[width=0.45\textwidth]{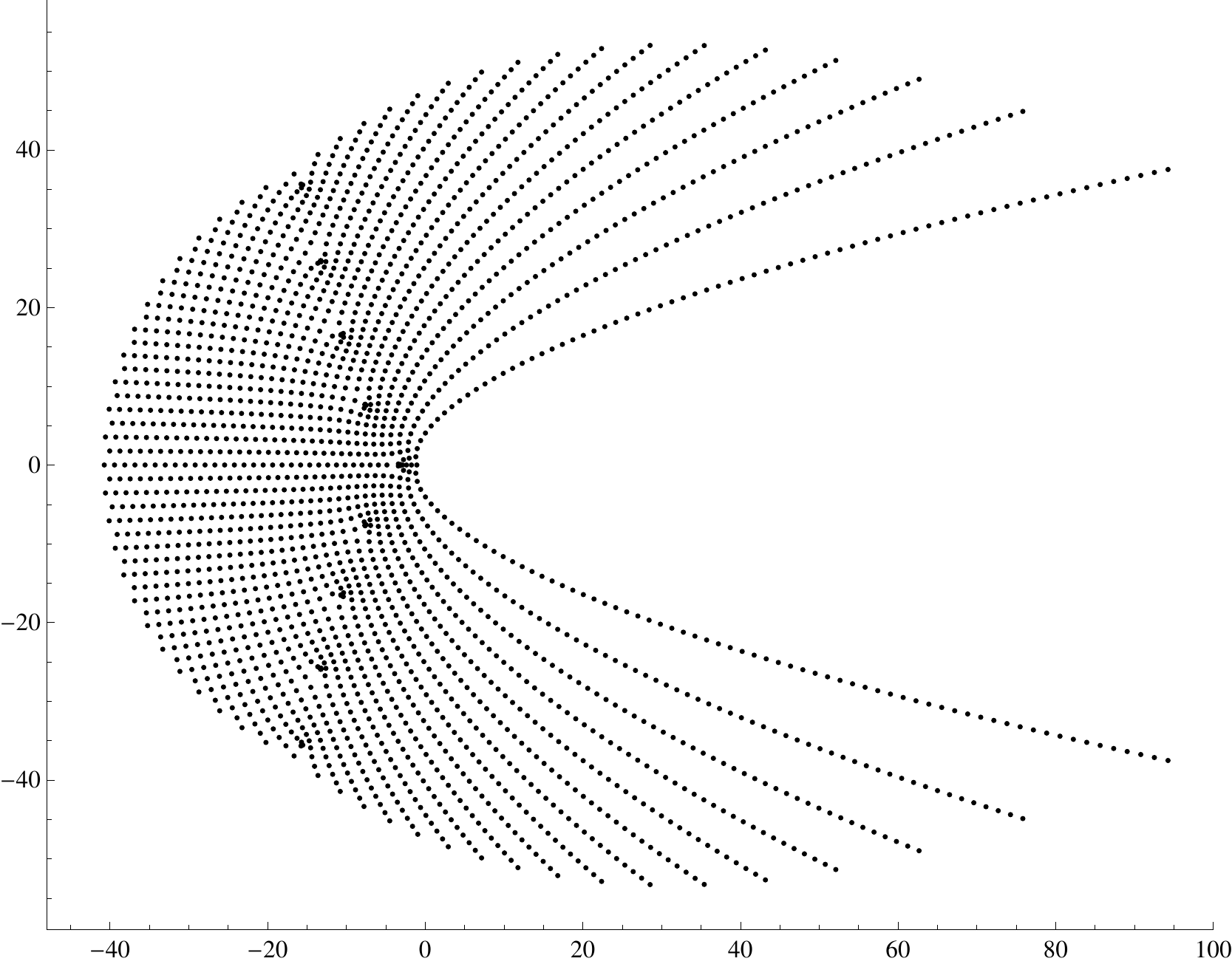}
			& \includegraphics[width=0.45\textwidth]{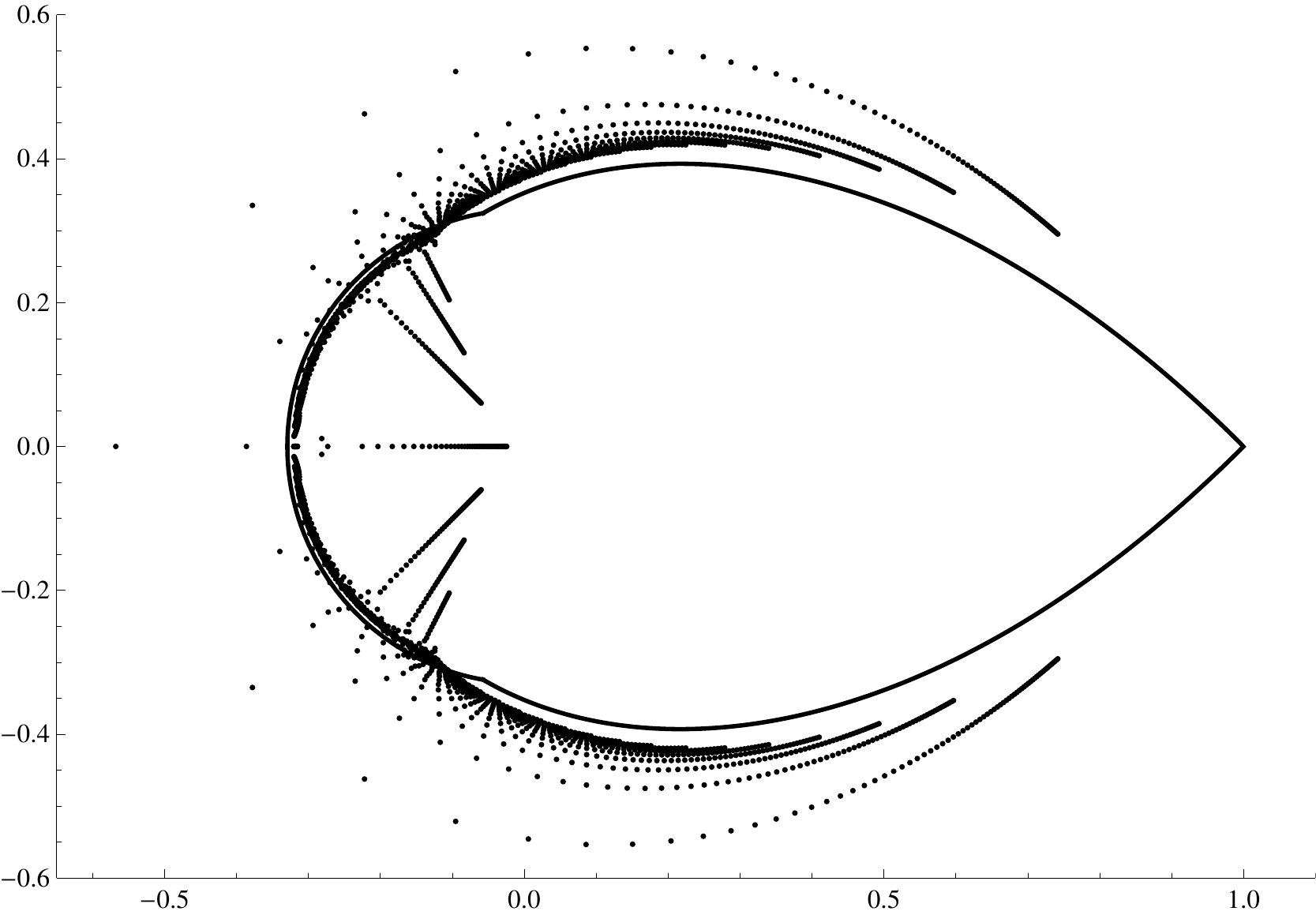} \\
		\includegraphics[width=0.45\textwidth]{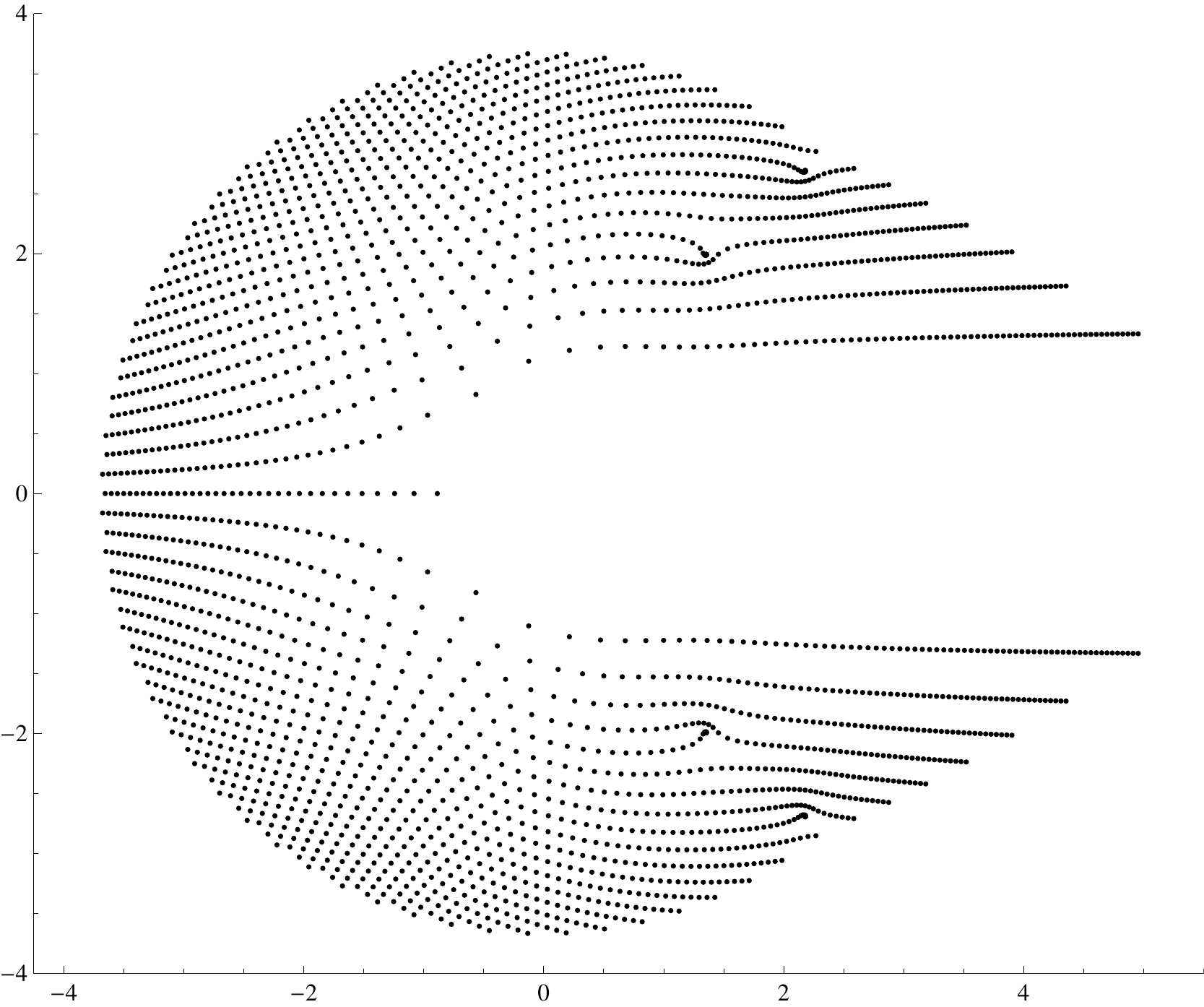}
			& \includegraphics[width=0.45\textwidth]{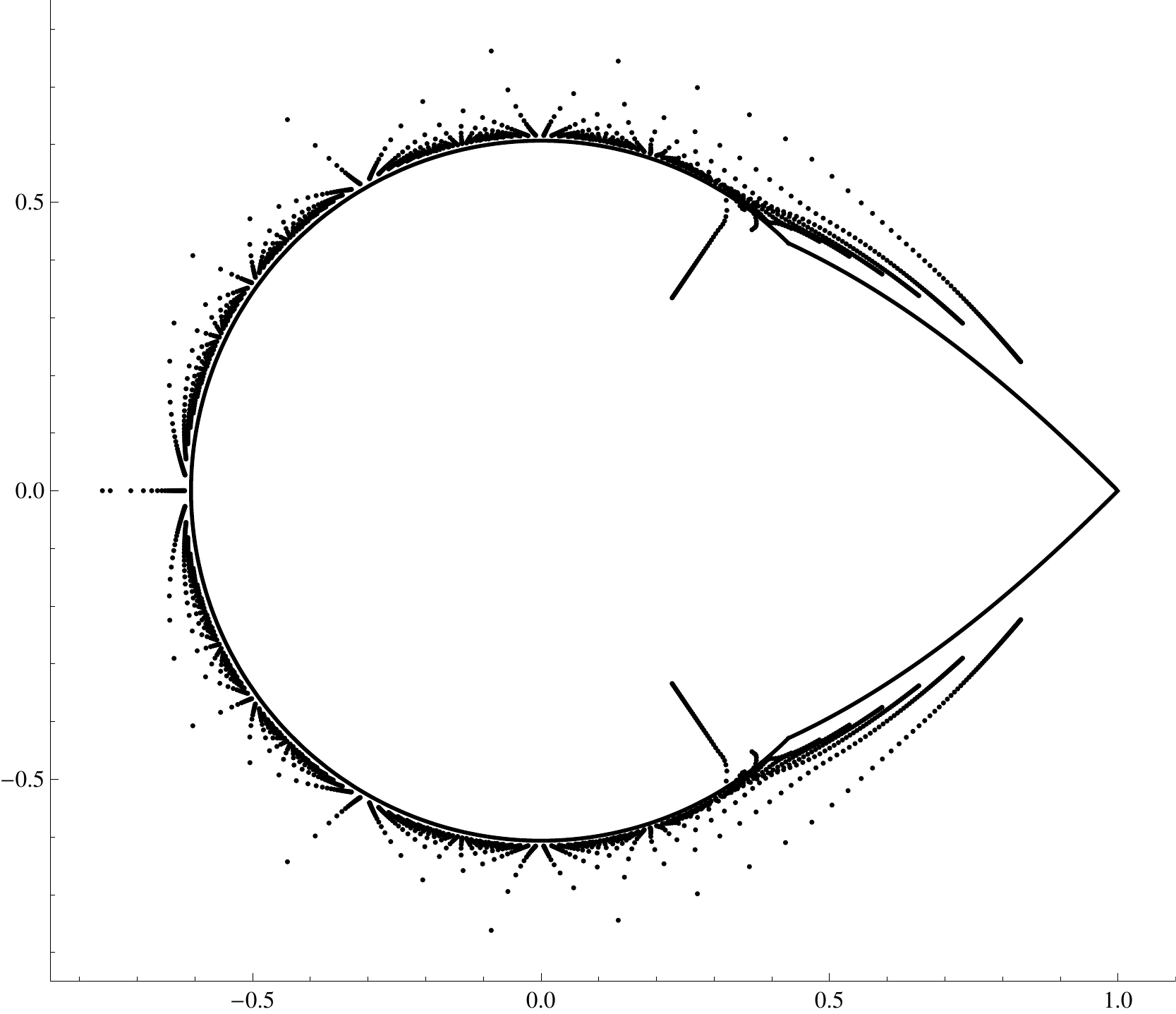}
	\end{tabular}
	\caption{LEFT: Zeros of $s_n(E_{1/\lambda};z)$ ($n=1,2,\ldots,70$) for $E_{10/9}$ (top) and $E_{1/2}$ (bottom).  RIGHT: Zeros of the normalized sections $s_n(E_{1/\lambda};R_n z)$ ($n=1,2,\ldots,70$) with their Szeg\H{o} curve for $E_{10/9}$ (top) and $E_{1/2}$ (bottom).}
\label{ml_plots}
\end{figure}

\subsection{Related Results}
\label{orderrelated}

In the same monograph, Edrei, Saff, and Varga also derived results similar to those above for ${\mathcal L}$-functions of order $0 < \lambda < 1$.  Here a function $f$ is said to be an ${\mathcal L}$-function if $f$ is entire, $f(0) > 0$, and
\[
	f(z) = f(0) \prod_{k=1}^{\infty} \left(1 + \frac{z}{x_k}\right),
\]
where $x_k > 0$ for all $k$ with
\[
	\sum_{k=1}^{\infty} \frac{1}{x_k} < \infty.
\]

Following Szeg\H{o}'s approach in equation \eqref{szegolincomb}, Natalya Zheltukhina \cite{zhel:mittaglincomb} extended the work of Edrei, Saff, and Varga on the Mittag-Leffler functions by studying the asymptotic behavior of the roots of the equation
\[
	s_n(E_{1/\lambda};R_n z) = \mu E_{1/\lambda}(R_n z)
\]
for $\mu \in \C$.  Later, Zheltukhina and Iossif Ostrovskii \cite{ostrozhel:lindeloflincomb} performed a similar analysis for the classical Lindel\"of functions, a subclass of the ${\mathcal L}$-functions studied by Edrei, Saff, and Varga.

\section{A Careful Study of the Asymptotics}
\label{cvwsection}

Five years after the publication of the monograph Varga returned to the problem for the exponential function \cite{cvw:expasympi}, this time with one of his previous PhD students, Amos J. Carpenter, and a collaborator, J\"{o}rg Waldvogel.

Their analysis began with the observation that the zeros away from the point $z = 1$ approach the limit curve much more quickly than the others.  Recall Buckholtz's result in Theorem \ref{buck2}, that the zeros approach the limit curve $D$ at a rate of $O(1/\sqrt{n})$.  The first step in the analysis here is to show that this estimate is the best possible one when the whole curve is taken into account.  Taking full advantage of Theorem \ref{nrtheo}, the authors prove the following result by tracing the behavior of a zero which approaches the point $z = 1$.

\begin{theorem}[CVW]
If $\{z_{k,n}\}_{k=1}^{n}$ are the zeros of $s_n(\exp;nz)$ and if $t_1$ is the zero of the complementary error function $\erfc$ closest to the origin in the upper half-plane, then
\[
	\liminf_{n \to \infty} \sqrt{n} \cdot \max_{k}\!\left\{ \dist\!\left(z_{k,n},\, D\right)\right\} \geq \Re(t_1) + \Im(t_1) \approx 0.636657.
\]
\label{carpvarg1}
\end{theorem}

By throwing out the zeros of $s_n(\exp;nz)$ near $z=1$ we should get a different estimate for the rate of approach.  Indeed, by defining $C_\delta$ to be the collection of all points within a distance $\delta$ of $z=1$, the authors show that the zeros outside of this set approach $D$ much more quickly.

For $\Omega \subseteq \C$, define $\maxdist\!\left(\Omega,\, C\right) = \sup_{z \in \Omega} \left\{\dist(z,C)\right\}$.

\begin{theorem}[CVW]
If $\{z_{k,n}\}_{k=1}^{n}$ are the zeros of $s_n(\exp;nz)$ and if $\delta$ is any fixed number with $0 < \delta \leq 1$, then
\[
	\maxdist\!\left(\{z_{k,n}\}_{k=1}^{n} \setminus C_\delta,\, D\right) = O\!\left(\frac{\log n}{n}\right)
\]
as $n \to \infty$.
\label{carpvarg2}
\end{theorem}

But the authors noticed something more.  As they approach $D$, the zeros seem to lie on regular curves which themselves shrink down to $D$.  Taking inspiration from Szeg\H{o}'s original analysis, the authors define the intermediate curves
\[
	D_n = \left\{z \in \C\,\colon \!\left|z e^{1-z} \right|^n = \tau_n \sqrt{2 \pi n} \left|\frac{1-z}{z}\right|,\,\, |z| \leq 1,\,\, \text{and}\,\, |\arg z| \geq \cos^{-1}\!\left(\frac{n-2}{n}\right)\right\},
\]
where
\[
	\tau_n = \frac{n!}{(n/e)^n \sqrt{2 \pi n}} \approx 1 + \frac{1}{12 n} + \frac{1}{288 n^2} - \frac{139}{51840 n^3} + \cdots
\]
as $n \to \infty$.  Indeed, the curve $D_n$ gives a much closer approximation of the zeros of $s_n(\exp;nz)$ than does the Szeg\H{o} curve $D$, as can be seen in Figure \ref{expintermed}.
\begin{theorem}[CVW]
If $\{z_{k,n}\}_{k=1}^{n}$ are the zeros of $s_n(\exp;nz)$ and if $\delta$ is any fixed number with $0 < \delta \leq 1$, then
\[
	\maxdist\!\left(\{z_{k,n}\}_{k=1}^{n} \setminus C_\delta,\, D_n\right) = O\!\left(1/n^2\right)
\]
as $n \to \infty$.
\label{carpvarg4}
\end{theorem}

\begin{figure}[htb]
	\centering
	\includegraphics[width=0.6\textwidth]{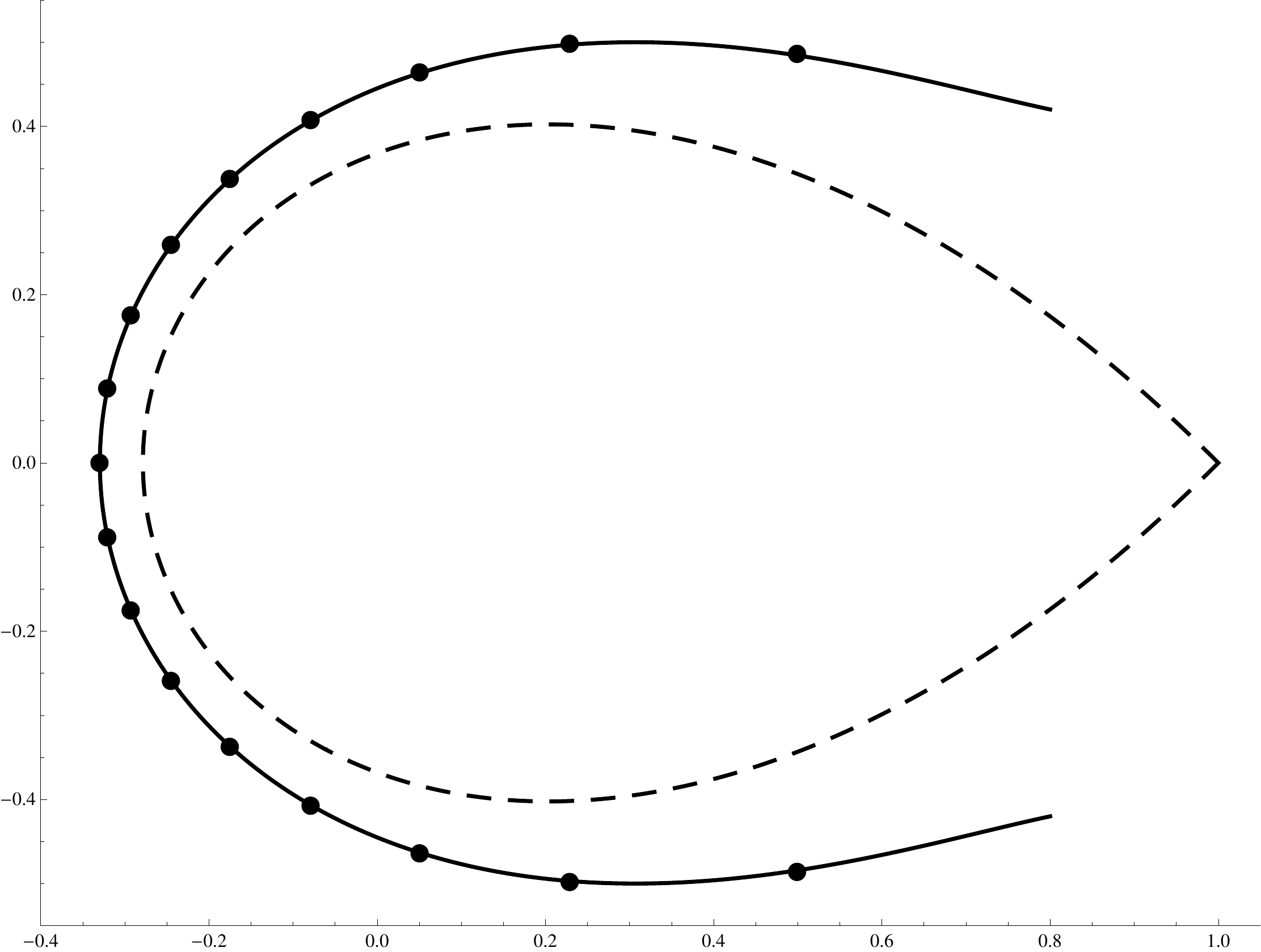}
	\caption{Zeros of $s_{17}(\exp;17 z)$ with the intermediate curve $D_{17}$ and the Szeg\H{o} curve $D$ (dotted).}
\label{expintermed}
\end{figure}

Varga and Carpenter continued their investigation of the asymptotics for the zeros of the sections of the exponential series in a second paper \cite{vc:expasympii}.

Ten years later, Varga and Carpenter published the first of three papers (\cite{vc:sincosasympi}, \cite{vc:sincosasympii}, and \cite{vc:sincosasympiii}) in which they carried out a similar analysis of the asymptotics for the zeros of the sections of sine and cosine.  Following the work of Szeg\H{o} \cite{szego:exp}, their approach essentially began with writing
\[
	2i\sin z = e^{iz} - e^{-iz}, \quad 2\cos z = e^{iz} + e^{-iz}
\]
and
\[
	2i s_n(\sin;z) = s_n(\exp;iz) - s_n(\exp;-iz), \quad 2 s_n(\cos;z) = s_n(\exp;iz) + s_n(\exp;-iz)
\]
then applying methods similar to those they used in their analysis of the exponential series, though in considerably more detail.

\section{A Divergent Power Series}
\label{divergentseries}

We shift our focus now to an example of an entirely different sort.  In 1996, Karl Dilcher and Lee A. Rubel \cite{dilcher:divergent} studied the power series
\begin{equation}
	\sum_{k=0}^{\infty} k! \,z^k.
\label{diverg}
\end{equation}
Up to this point the series we have covered have converged on the entire complex plane.  At the opposite end of the spectrum is this series, which converges nowhere but at the origin.  It is then incredibly surprising that the zeros of the sections of this series can be wrangled with using essentially the same ideas.  Here, the zeros of the partial sums
\[
	p_n(z) = \sum_{k=0}^{n} k! \,z^k
\]
move not out to infinity but in toward the origin.  So, instead of choosing a scale factor which increases with $n$, one is chosen which decreases with $n$.

\begin{theorem}[Dilcher and Rubel]
If $a \neq 1$ is a complex number, then for all positive integers $n$ satisfying
\[
	n \geq \max\!\left\{\frac{16}{9}e^4 |1-a|^4,\left(\frac{3}{|1-a|}\right)^2\right\},
\]
the roots of the equation $p_n(e z/n) = a$ lie in the annulus
\[
	1-\frac{3}{|1-a|}\cdot\frac{1}{\sqrt{n}} < |z| < 1 + (1+|a|)\sqrt{\frac{2}{\pi n}}.
\]
\end{theorem}

As a special case of this result, we may use the sharpened form of the Enestr\"om-Kakeya theorem in \cite{asv:ke} to see that the zeros of the normalized $n^{\text{th}}$ section $p_n(e z/n)$ all lie in the annulus
\[
	1-\frac{3}{\sqrt{n}} < |z| < 1.
\]
for $n > 97$.

\begin{figure}[h!tb]
	\centering
	\begin{tabular}{cc}
		\includegraphics[width=0.45\textwidth]{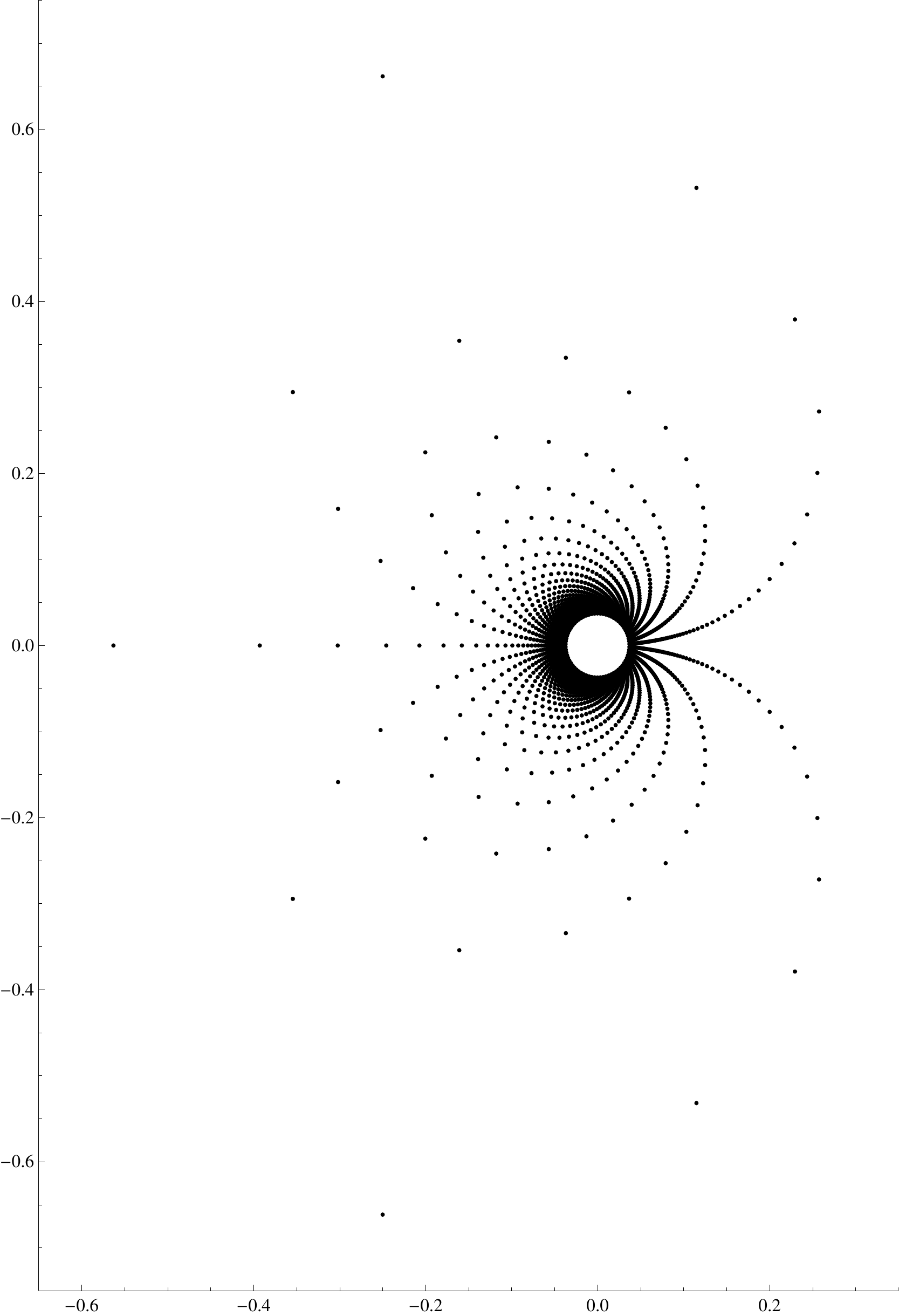}
			& \includegraphics[width=0.45\textwidth]{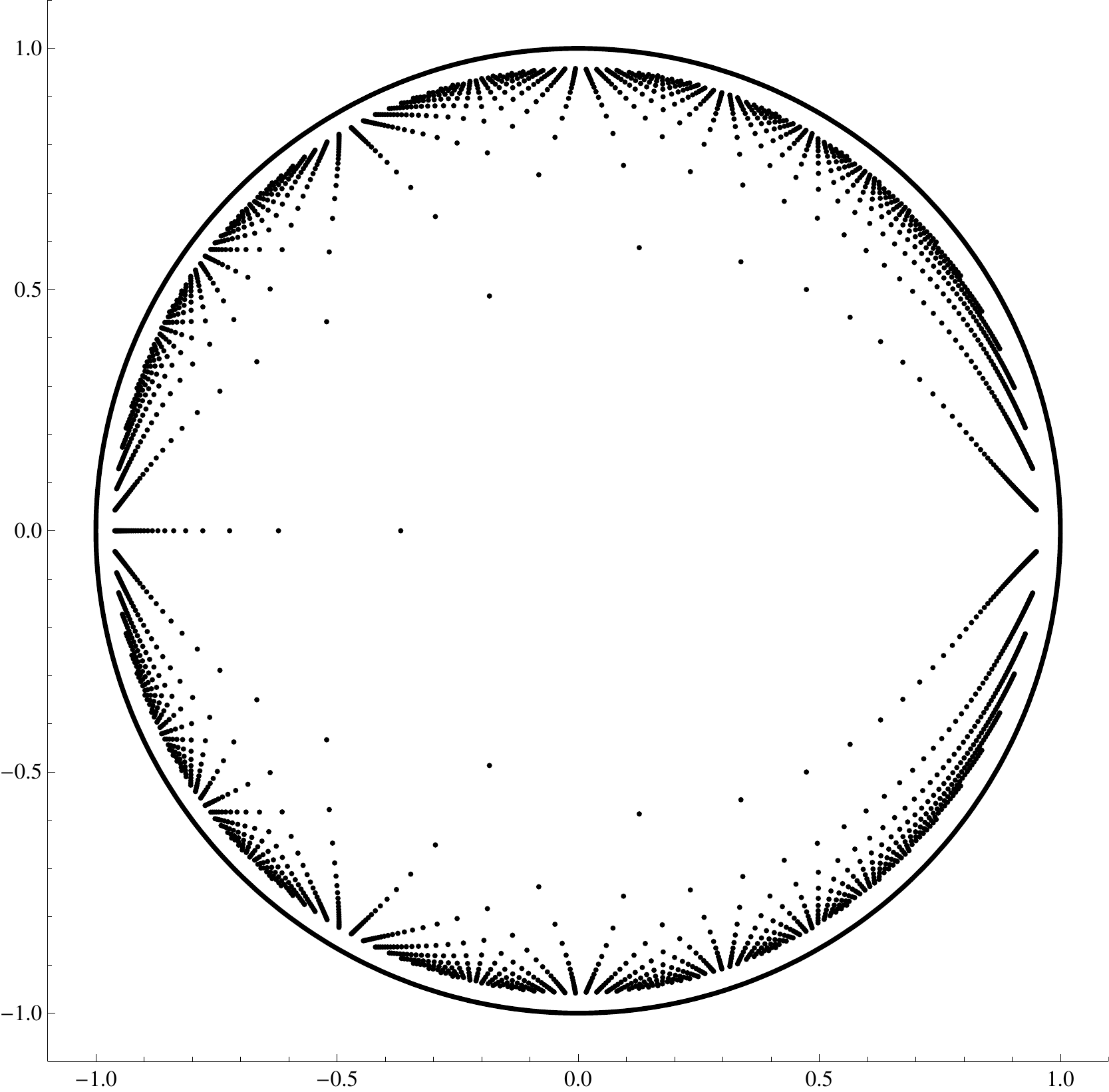}
	\end{tabular}
	\caption{LEFT: Zeros of $p_n(z) = \sum_{k=0}^{n} k! z^k$ ($n=1,2,\ldots,70$).  RIGHT: Zeros of the normalized sections $p_n(e z/n)$ ($n=1,2,\ldots,70$) plotted with their Szeg\H{o} curve, the unit circle.}
\label{diverg_plots}
\end{figure}

\section{Inspiration for Studying Exponential Integrals}
\label{norfolkinspiration}

The primary inspiration for this current work is a paper by Timothy S. Norfolk \cite{norfolk:1f1} published in 1998.  Studying the confluent hypergeometric functions
\[
	{}_1F_1(1;b;z) = \Gamma(b) \sum_{k=0}^{\infty} \frac{z^k}{\Gamma(k+b)},
\]
where $b \neq 1, 0, -1, -2, \ldots$, Norfolk derived results analogous to those of Carpenter, Varga, and Waldvogel outlined in Section \ref{cvwsection}.  His primary tool was the integral representation
\begin{equation}
	{}_1F_1(1;b;z) = (b-1) \int_0^1 (1-t)^{b-2} e^{zt} \,dt,
\label{1f1integralrep}
\end{equation}
valid for $b > 1$.

If
\[
	s_n({}_1F_1,b;z) = \Gamma(b) \sum_{k=0}^{n} \frac{z^k}{\Gamma(k+b)}
\]
is the $n^{\text{th}}$ section of ${}_1F_1(1;b;z)$, we have
\[
	s_n({}_1F_1,b;z) = (b-1) \int_0^1 (1-t)^{b-2} s_n(\exp;zt) \,dt,
\]
where $s_n(\exp;z)$ is the $n^{\text{th}}$ section of the exponential function.  Replacing $z$ with $nz$ and subtracting this from equation \eqref{1f1integralrep} we get the expression
\begin{align}
{}_1F_1(1;b;nz) - s_n({}_1F_1,b;nz) &= (b-1) \int_0^1 (1-t)^{b-2} \Bigl(e^{nzt} - s_n(\exp;nzt)\Bigr) dt \nonumber \\
			&= (b-1) \int_0^1 (1-t)^{b-2} e^{nzt} g_n(zt) \,dt,
\label{1f1difference}
\end{align}
where
\[
	g_n(z) = 1 - e^{-nz} s_n(\exp;nz).
\]
This quantity $g_n(z)$ was studied by Szeg\H{o} \cite{szego:exp}, who derived asymptotic approximations as $n \to \infty$ for $z$ in different regions of the complex plane (see also \cite{cvw:expasympi} and \cite{boyergoh:euler}).  In particular he showed that
\[
	g_n(z) = \frac{\left(ze^{1-z}\right)^n}{\sqrt{2 \pi n}} \cdot \frac{z}{1-z} \Bigl(1 - \epsilon_n(z)\Bigr),
\]
where $\epsilon_n(z) = O(1/n)$ as $n \to \infty$ uniformly when $z$ is restricted to a compact subset of $\Re(z) < 1$.  Norfolk applies this approximation of Szeg\H{o}'s to obtain an asymptotic estimate for the tail \eqref{1f1difference} in the region $|z| \leq R < 1$.

Norfolk's main result is that the zeros of the normalized sections $s_n({}_1F_1,b;nz)$ have as their limit points the set
\begin{align*}
	&\left\{z \in \C \,\colon \Re(z) \geq 0, \,\,\,|z| \leq 1, \,\,\,\text{and}\,\,\, \left|ze^{1-z}\right| = 1\right\} \\
	&\quad \cup \{z \in \C \,\colon \Re(z) = 0 \,\,\,\text{and}\,\,\, |z| \leq 1/e\} \\
	&\quad \cup \{z \in \C \,\colon \Re(z) \leq 0 \,\,\,\text{and}\,\,\, |z| = 1/e\}.
\end{align*}
This curve and some of the zeros can be seen in Figure \ref{1f1zeros}.

\begin{figure}[h!tb]
	\centering
	\begin{tabular}{cc}
		\includegraphics[width=0.45\textwidth]{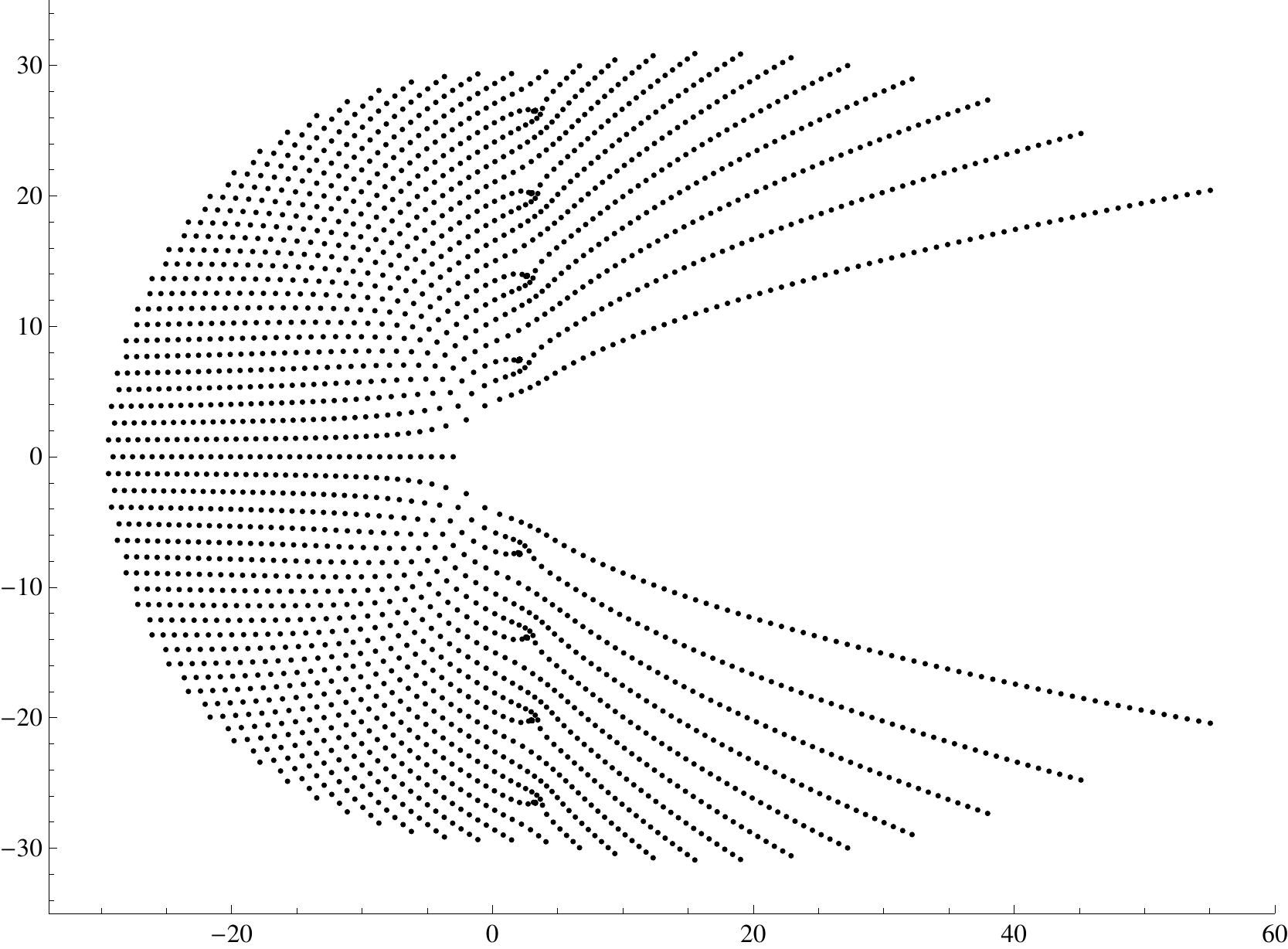}
			& \includegraphics[width=0.45\textwidth]{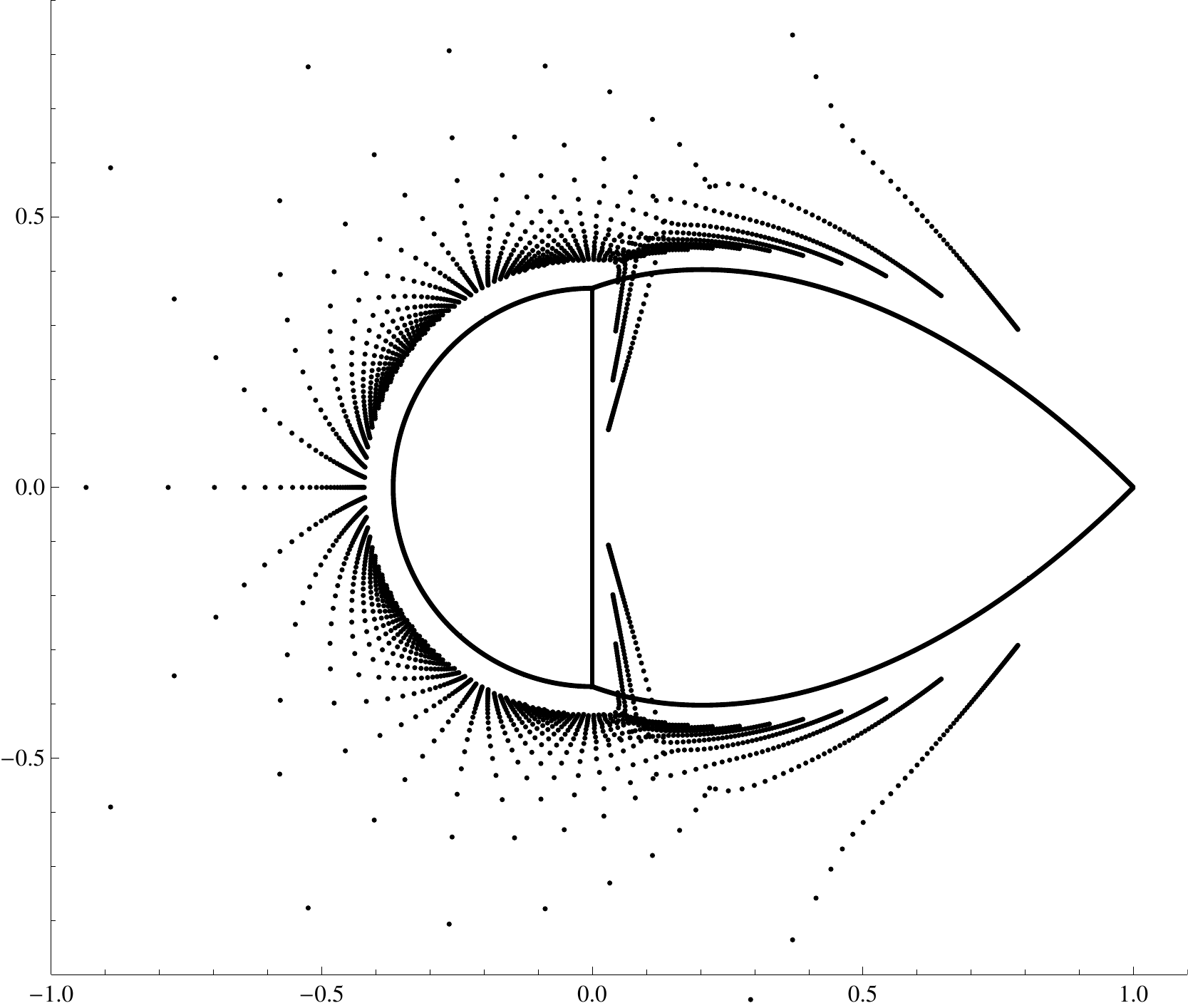}
	\end{tabular}
	\caption{LEFT: Zeros of the sections $s_n({}_1F_1,3;z)$ $(n=1,2,\ldots,70)$.  RIGHT: Zeros of the normalized sections $s_n({}_1F_1,3;nz)$ $(n=1,2,\ldots,70)$ with their Szeg\H{o} curve.}
\label{1f1zeros}
\end{figure}

We follow essentially the same process in Chapter \ref{res}, generalizing Norfolk's result to a much larger class of functions.  These functions are defined in Chapter \ref{prelims} and a few special cases are listed in Section \ref{relev:cases}.

\section{Other Families of Polynomials}
\label{otherpolys}

Given a power series with finite radius of convergence, Jentzsch's Theorem states that every point on the circle of convergence will be a limit point of the zeros of the sections of the series.  In this sense, the result of Jentzsch is a bridge between the results of Dilcher and Rubel in section \ref{divergentseries} and the results described earlier in the chapter for entire functions.  We consider a couple simple examples of power series with finite radius of convergence in Sections \ref{prelims:strat:mob} and \ref{gen}.

The problem of asymptotic zero distribution has also been treated for a number of families of polynomials which are not sections of some power series.  We reference a few of them here.

The Daubechies polynomials
\[
	B_p(z) = \sum_{k=0}^{p-1} \binom{k+p-1}{k} z^k
\]
were studied by Djalil Kateb and Pierre Gilles Lemarie-Rieusset \cite{kateb:daubechies} and independently by Jianhong Shen and Gilbert Strang \cite{shenstrang:daubechies}.  Later, Norfolk and Svante Janson \cite{norfolk:binom} studied the related polynomials
\[
	B_{r,n}(z) = \sum_{k=0}^{r} \binom{n}{k} z^k,
\]
where $1 \leq r < n$.  The partition polynomials, defined by
\[
	H_n(z) = \sum_{k=1}^{n} p_k(n) z^k,
\]
where $p_k(n)$ is the number of partitions of $n$ with exactly $k$ parts, were studied by Robert P. Boyer and William M. Y. Goh \cite{boyergoh:partition}.

If, in the hypergeometric functions
\[
	{}_2F_1(a_1,a_2;b_1;z) = 1 + \sum_{k=1}^{\infty} \frac{(a_1)_k (a_2)_k}{k! (b_1)_k} \,z^k
\]
and
\[
	{}_3F_2(a_1,a_2,a_3;b_1,b_2;z) = 1 + \sum_{k=1}^{\infty} \frac{(a_1)_k (a_2)_k (a_3)_k}{k! (b_1)_k (b_2)_k} \,z^k,
\]
where
\[
	(\alpha)_k = \frac{\Gamma(\alpha+k)}{\Gamma(\alpha)},
\]
we let $a_1 = -n$ be a negative integer, the series in question are finite and we obtain sequences of polynomials. Hypergeometric polynomials of this type been studied by various authors (see, e.g., \cite{duren:hypergeom}, \cite{driver:hypergeom1}, \cite{driver:hypergeom2}, \cite{driver:hypergeom3}, and \cite{zsw:hypergeom}).  Some of these results have been used to obtain similar results for the zeros of classical Jacobi orthogonal polynomials.

The problem has also been studied directly for a number of other families of orthogonal polynomials (see, e.g., \cite{dilcher:bernoullieuler}, \cite{carpenter:bessel}, \cite{orive:laguerrebessel}, \cite{orive:jacobi}, and \cite{boyergoh:euler}).  Of note are the papers of Arno B. J. Kuijlaars, Kenneth D. T.-R. McLaughlin, and Peter D. Miller, who apply the methods of Riemann-Hilbert analysis to achieve their results.  In 2008, Kuijlaars, McLaughlin, Miller, and Thomas Kriecherbauer turned these tools back on the original problem of studying the zeros of sections of the exponential series \cite{kuij:expriemannhilbert}.  The authors derive full asymptotic series for the zeros in terms of roots of appropriately chosen auxiliary equations.

Though it is only tangentially related to the current theory, we also refer the reader to a paper of Christopher P. Hughes and Ashkan Nickeghbali \cite{hughesnick:randompolys} for an interesting result about the clustering behavior of zeros of random polynomials.



\chapter{Preliminaries}
\label{prelims}

\ifpdf
    \graphicspath{{Chapter2/Chapter2Figs/PNG/}{Chapter2/Chapter2Figs/PDF/}{Chapter2/Chapter2Figs/}}
\else
    \graphicspath{{Chapter2/Chapter2Figs/EPS/}{Chapter2/Chapter2Figs/}}
\fi

In this chapter we will discuss the terms and ideas which are relevant to the work.  Some of these have already been introduced in Chapter \ref{intro} and are collected here for convenience.

Our focus will be on the sequence of polynomials given by the partial sums---the sections---of a convergent power series.  In particular, if $\Omega$ is some open subset of the complex plane $\C$ which contains the origin and $f \,\colon \Omega \to \C$ is a function which is analytic at the origin, then $f$ can be represented by a power series,
\begin{equation}
	f(z) = \sum_{k=0}^{\infty} a_k z^k,
\label{fdef}
\end{equation}
in some neighborhood of the origin.  We define the $n^{\text{th}}$ section of this power series for $f$ to be the polynomial \newnot{symbol:section}
\begin{equation}
	s_n(f;z) = \sum_{k=0}^{n} a_k z^k.
\label{sndef}
\end{equation}
In later sections we will rely heavily on properties of the exponential function
\[
	\exp(z) = \sum_{k=0}^{\infty} \frac{z^k}{k!}
\]
and its sections
\[
	s_n(\exp;z) = \sum_{k=0}^{n} \frac{z^k}{k!}.
\]

\section{Strategy and General Tools}
\label{prelims:strat}

Here we will give a brief description of the method we will use in Chapter \ref{res} to obtain our results.  We will also collect a few tools for later use.

Let $f$ and $s_n(f;z)$ be as in equations \eqref{fdef} and \eqref{sndef}.  For $z$ in the radius of convergence of the power series for $f$, we have
\begin{equation}
	f(z) - s_n(f;z) = t_n(f;z),
\label{sectail}
\end{equation}
where \newnot{symbol:tail}
\[
	t_n(f;z) = \sum_{k=n+1}^{\infty} a_k z^k
\]
is the tail of the power series.  Suppose we can find a real, positive sequence $\{R_n\}$ such that some subset of the zeros of $s_n(f;z)$ grow on the order of $R_n$ and let $\{R_n z_n\}$ be such a sequence of zeros which eventually lies within the radius of convergence of the power series.  By equation \eqref{sectail} this gives us
\[
	f(R_n z_n) = t_n(f;R_n z_n)
\]
for $n$ large enough.  Theoretically, analytic relations for the scaled zeros $z_n$ can be derived by determining the asymptotic character of the sequences $f(R_n z_n)$ and $t_n(f;R_n z_n)$ as $n \to \infty$.

An example of this method is given in Section \ref{prelims:strat:mob}.

We will need to bound the modulus of the zeros of the sections to simplify the estimation of the tails of the power series in Chapter \ref{res}.  To do this we will invoke a particular result of Rosenbloom (\cite{rosen:thesis} and \cite{rosen:distrib}) which will give us precisely the bound we need.

Recall from Chapter \ref{intro} that a sequence of sections $\{s_N(f;z)\}$ is said to have a positive fraction of zeros in any sector with vertex at the origin if
\[
	\liminf_{N \to \infty} \frac{\sharp_N^{\angle}(\theta_1,\theta_2)}{N} > 0
\]
for any fixed $\theta_1$ and $\theta_2$, where $\sharp_N^{\angle}(\theta_1,\theta_2)$ is the number of zeros of the section $s_N(f;z)$ in the sector $\theta_1 \leq \arg z \leq \theta_2$.  Rosenbloom's result is as follows.

\begin{theorem}[Rosenbloom]
Let
\[
	f(z) = \sum_{k=0}^{\infty} a_k z^k
\]
be an entire function of finite positive order $\rho$ and let $s_n(f;z)$ be as in equation \eqref{sndef}.  Define $\sharp_{n}^{\circ}(R)$ \newnot{symbol:numdisk} to be the number of zeros of $s_n(f;z)$ in the disk $|z| \leq R$ and define $\rho_n = |a_n|^{-1/n}$.  For any increasing sequence of indices $\{N\}$ \newnot{symbol:Nindices} such that the sequence of sections $\{s_N(f;z)\}$ have a positive fraction of zeros in any sector with vertex at the origin, the existence of which is guaranteed, and for $0 \leq r < 1$ and $\epsilon > 0$ we have
\[
	\liminf_{N \to \infty} \frac{\sharp_N^{\circ}\Bigl((e^{1/\rho} + \epsilon)\rho_N\Bigr) - \sharp_N^{\circ}(\rho_N r)}{N} \geq 1 - r^{\rho} > 0.
\]
Further, the number of zeros of $s_N(f;z)$ satisfying
\[
	|\rho_N z| > e^{1/\rho} + \epsilon
\]
is bounded.
\label{rosentheo}
\end{theorem}

Part of this result was inspired by the work of Carlson \cite{carlson:entiresector} on the angular distribution of the zeros.

Many of the recent approaches to the problem (e.g. \cite{norfolk:1f1}) have made use of the Enestr\"om-Kakeya Theorem to find this bound.  Marden's book \cite[ch. 7]{marden:geom} is a good resource for more information on this theorem as well as for more general results in in the same vein.

\begin{theorem}[Enestr\"om-Kakeya]
All zeros of the polynomial
\[
	p(z) = a_0 + a_1 z + \cdots + a_n z^n
\]
having real, positive coefficients $a_j$ lie in the ring $\alpha \leq |z| \leq \beta$, where
\[
	\alpha = \min\{a_k/a_{k+1}\} \quad \text{and} \quad \beta = \max\{a_k/a_{k+1}\}
\]
for $k=0,1,\ldots,n-1$.
\label{kakene}
\end{theorem}

In \cite{asv:ke}, Anderson, Saff, and Varga give necessary and sufficient conditions for when zeros of $p$ lie on the circles $|z| = \alpha$ or $|z| = \beta$.  Of particular interest is the maximum modulus of the zeros, and as a corollary to their main theorems the authors formulate a helpful sufficient condition that the polynomial have no roots on the outer circle $|z| = \beta$.

\begin{theorem}[Anderson, Saff, Varga]
For $p(z)$ and $\beta$ as defined in Theorem \ref{kakene}, if $\beta a_1 - a_0 > 0$, then all zeros of $p$ satisfy $|z| < \beta$.
\label{asvkakene}
\end{theorem}

The requirement in the above two theorems that all coefficients be positive is fairly restrictive.  As we mentioned, there are more flexible versions of the Enestr\"om-Kakeya theorem that do not assume this, but they are often not powerful enough or simply too difficult to apply to this problem in full generality.  For example, the best bound these more general theorems can give for the zeros of sections of the exponential integrals we will study is precisely twice the bound we need.  For this reason we instead rely on Rosenbloom's Theorem \ref{rosentheo}.

\subsection{A Class of Linear Fractional Transformations}
\label{prelims:strat:mob}

To illustrate the general strategy described above let us pause to treat a toy example.

Let $a_0$, $A$, and $B$ all be positive and real.  Define $a_1 = A a_0 + B$ and $a_n = A^{n-1} a_1$ for $n \geq 2$.  Then the series $\sum a_k z^k$ converges to the linear fractional transformation
\begin{equation}
	f(z) = \frac{a_0 + B z}{1 - A z}
\label{mobdef}
\end{equation}
for all $z \in \C$ satisfying $|z| < 1/A$.  The sections of this series can be explicitly expressed as
\[
	s_n(f;z) = a_0 + a_1 z \,\frac{1 - A^n z^n}{1-A z}.
\]
By Theorem \ref{asvkakene}, all zeros of $s_n(f;z)$ satisfy $|z| < 1/A$.  We thus consider the normalized sections $s_n(f;z/A)$, all of whose zeros lie strictly inside the unit circle.

Given a power series with finite radius of convergence, Jentzsch's Theorem states that every point on the circle of convergence will be a limit point of the zeros of the sections of the series.  Thus every point on the unit circle is a limit point of the zeros of the sections $s_n(f;z/A)$.  If $f(z/A$) has a zero satisfying $|z| < 1$ (which would occur at $z = -A a_0/B$), Hurwitz's Theorem (see, e.g., \cite[p. 4]{marden:geom}) tells us that $z = -A a_0/B$ will also be a limit point of the zeros of the sections.  However, we are only concerned with the sections' spurious zeros---that is, sequences of zeros which do not converge to zeros of the limit function---so we will exclude these limit points in our calculations.

We prove the following result.

\begin{theorem}
Let $C$ be the unit circle, $R_{\delta}$ the open region which consists of all points within a distance $\delta > 0$ of the negative real axis, and $\{z_{k,n}\}_{k=1}^{n}$ the zeros of the section $s_n(f;z/A)$.  For $\Omega \subseteq \C$, define $\maxdist\!\left(\Omega,\, C\right) = \sup_{z \in \Omega} \left\{\dist(z,C)\right\}$.  Then
\[
	\maxdist\!\left(\{z_{k,n}\}_{k=1}^{n} \!\setminus\! R_{\delta} \,;\, C \right) = O(1/n).
\]
\end{theorem}

\begin{proof}
Because $f(z/A)$ has no zeros in the set $\Delta = \{z \in \C\,\colon |z| < 1\} \setminus R_{\delta}$, we may conclude that the points $\{z \in \C \,\colon |z|=1\} \setminus R_{\delta}$ are the only limit points in the closure of $\Delta$.

For $z \in \Delta$ we calculate
\[
	\frac{s_n(f;z/A)}{f(z/A)} = 1 - z^{n+1} \frac{a_1}{A a_0 + B z}.
\]
So, if $z$ is a zero of $s_n(f;z/A)$, it must satisfy the relationship
\begin{equation}
	z^{n+1} = (A a_0 + B z)/a_1.
\label{mobroots}
\end{equation}
For $z \in \Delta$ it is clear that there exist positive constants $C_1$ and $C_2$ such that \linebreak $C_1 < |(A a_0 + B z)/a_1| < C_2$, so that on taking absolute values and $(n+1)^{\text{th}}$ roots in \eqref{mobroots} we find that the spurious zeros of $s_n(f;z/A)$ satisfy
\[
	|z| = 1 + O(1/n),
\]
as desired.
\end{proof}

This result is illustrated in Figures \ref{mobplot} and \ref{mobintermed}.

\begin{figure}[htb]
	\centering
	\includegraphics[width=0.6\textwidth]{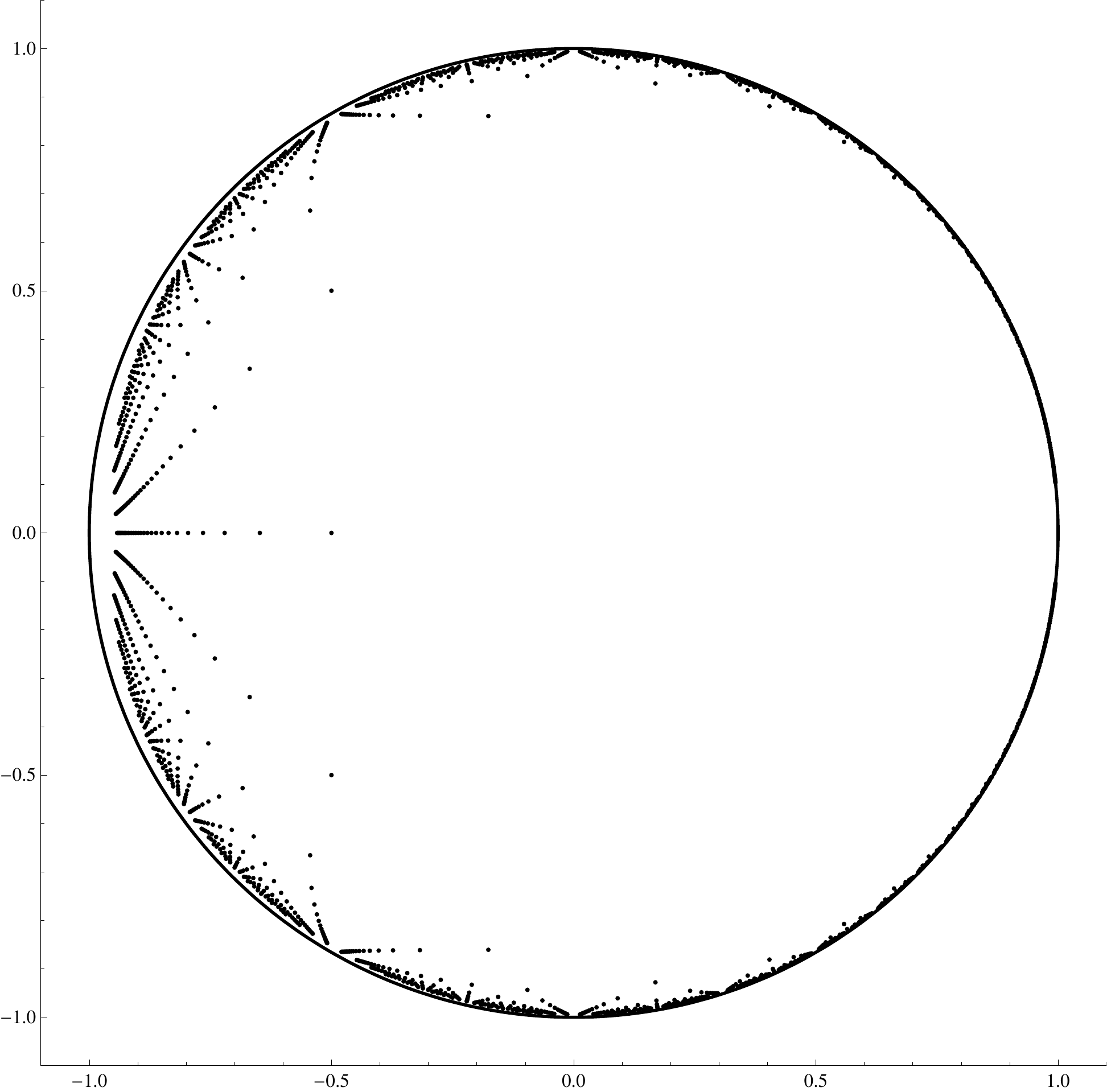}
	\caption{Zeros of $s_n(f;z)$ ($n=1,2,\ldots,60$), where $f$ is the linear fractional transformation in \eqref{mobdef} with $a_0 = A = B = 1$, and their Szeg\H{o} curve, the unit circle.}
\label{mobplot}
\end{figure}

\begin{figure}[htb]
	\centering
	\includegraphics[width=0.6\textwidth]{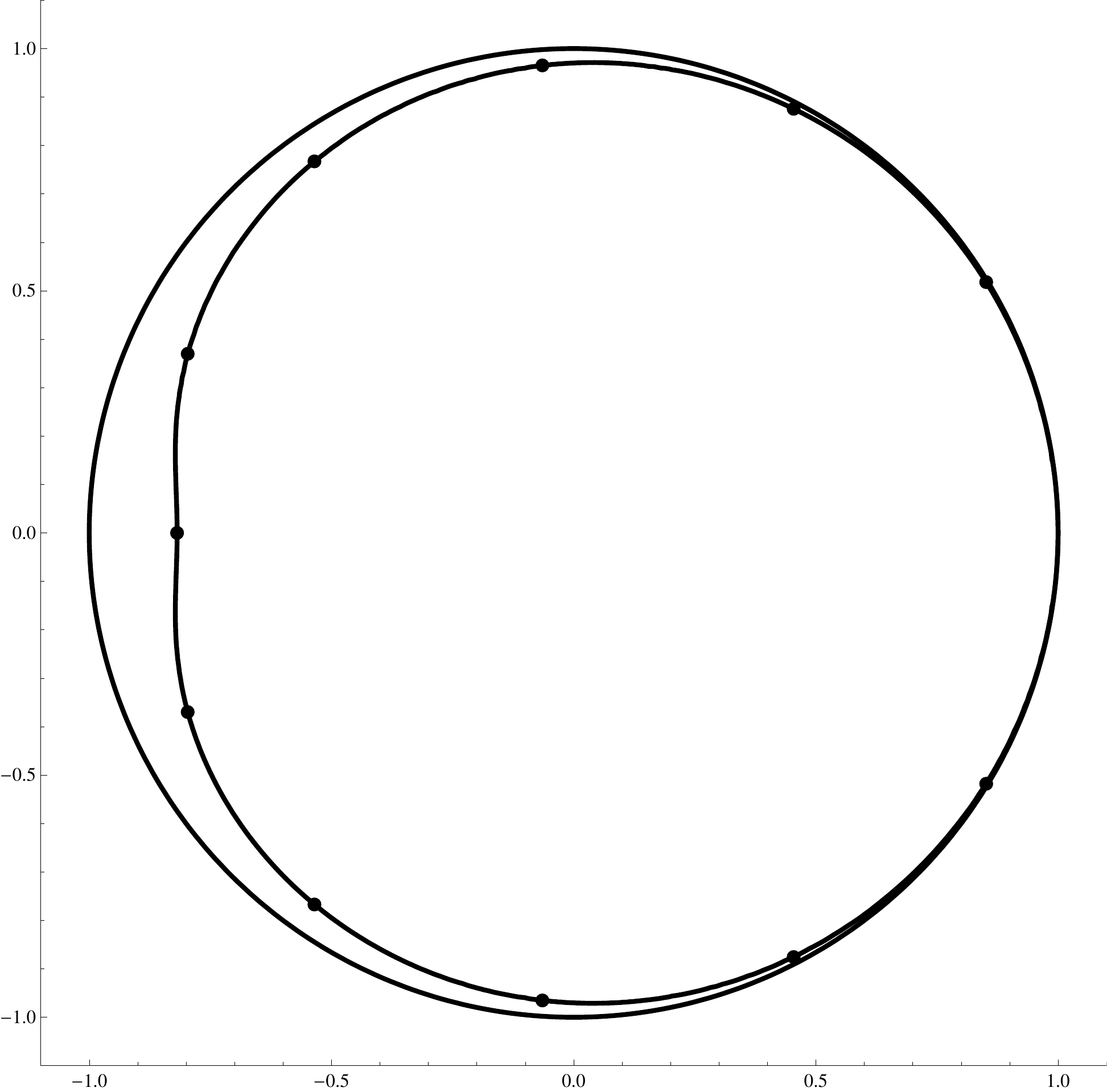}
	\caption{Zeros of $s_{11}(f;z)$, where $f$ is as in \eqref{mobdef} with $a_0 = A = B = 1$, with the intermediate curve (the modulus of equation \eqref{mobroots}) and the unit circle.}
\label{mobintermed}
\end{figure}

\section{The Exponential Integral Functions}
\label{prelims:expints}

The functions we will study will be defined by integrals of the form $\int_{-a}^{b} \varphi(t) e^{zt}\,dt$.  The restrictions we will place on the function $\varphi$ are determined essentially by the abilities of Watson's Lemma, which are discussed in the next section.

Suppose $0 \leq a,b < \infty$ and let $\varphi \,\colon [-a,b] \to \C \cup \{\infty\}$ be a function satisfying
\[
	\int_{-a}^{b} |\varphi(t)|\,dt < \infty
\]
and $\varphi(t) = (t+a)^{\mu} f_1(t+a) = (b-t)^{\nu} f_2(b-t)$, where
\begin{enumerate}[leftmargin=1.6cm,label=(\arabic*)]
\item $\mu,\nu \in \C$ with $\Re(\mu) > -1$ and $\Re(\nu) > -1$,
\item $f_1,f_2 \,\colon [0,a+b] \to \C \cup \{\infty\}$ with $f_1(0)$ and $f_2(0)$ both finite and nonzero,
\item in a neighborhood of $t=0$, both $f_1'(t)$ and $f_2'(t)$ exist and are bounded.
\end{enumerate}
Define
\[
	F(z) = \int_{-a}^{b} \varphi(t) e^{zt}\,dt.
\]
The function $F$ is entire, and its sections are given by the formula
\[
	s_n(F;z) = \sum_{k=0}^{n} \frac{z^k}{k!} \int_{-a}^{b} \varphi(t) t^k\,dt.
\]
Notice that $F(z)$ is the exponential generating function of these particular integral moments of $\varphi$.

Properties $(1)$ and $(2)$ above essentially serve to ensure that the integral $F(z)$ converges.  Property $(3)$ allows us to determine a simple error term in the asymptotic expansion of $F$ in Watson's Lemma.

An important example of functions of this type are the Bessel functions of the first kind $J_{\alpha}$ with $\Re(\alpha) > -1/2$.  Indeed, in this case we have
\[
	\Gamma\!\left(\alpha + \frac{1}{2}\right) \Gamma\!\left(\frac{1}{2}\right) \left(\frac{2i}{z}\right)^{\alpha} J_{\alpha}(-iz) = \int_{-1}^{1} \left(1-t^2\right)^{\alpha-1/2} e^{zt} \,dt.
\]
We discuss in detail how the main result of Chapter \ref{res} applies to these Bessel functions in Section \ref{relev:cases}.

\section{Watson's Lemma}
\label{prelims:watson}

The primary tool in this work is Watson's Lemma.  For a thorough discussion of this result in a general setting see \cite{miller:aaa}.

In the following, $\lambda$ is a complex parameter.

\begin{theorem}[Watson's Lemma]
Suppose $0 < T \leq \infty$ and $\varphi \,\colon [0,T] \to \C \cup \{\infty\}$ is a function satisfying
\[
	\int_0^T |\varphi(t)|\,dt < \infty
\]
and $\varphi(t) = t^{\sigma}h(t)$, where $\Re(\sigma) > -1$, $h(0) \neq 0$, and $h'(t)$ exists and is bounded in a neighborhood of $t=0$.  Then the exponential integral
\[
	\Phi(\lambda) = \int_0^T \varphi(t) e^{-\lambda t}\,dt
\]
is finite for all $\Re(\lambda) > 0$, and
\[
	\Phi(\lambda) = \frac{h(0) \Gamma(\sigma+1)}{\lambda^{\sigma+1}} + O\!\left(\lambda^{-\sigma-2}\right)
\]
as $\lambda \to \infty$ with $|\arg \lambda| \leq \theta$ for any fixed $0 \leq \theta < \pi/2$.
\end{theorem}

Though this form of Watson's Lemma only gives an asymptotic for $\Phi(\lambda)$ as $\lambda \to \infty$ to the right, it can easily be extended to address the case when $\lambda \to \infty$ to the left.  However, to do so we must assume that $T$ is finite.

\begin{corollary}
Suppose $0 < T < \infty$ and $\varphi \,\colon [0,T] \to \C \cup \{\infty\}$ is a function satisfying
\[
	\int_0^T |\varphi(t)|\,dt < \infty
\]
and $\varphi(t) = (T-t)^{\sigma}h(T-t)$, where $\Re(\sigma) > -1$, $h(0) \neq 0$, and $h'(t)$ exists and is bounded in a neighborhood of $t=0$.  Then the exponential integral
\[
	\Phi(\lambda) = \int_0^T \varphi(t) e^{\lambda t}\,dt
\]
is finite for all $\Re(\lambda) > 0$, and
\[
	\Phi(\lambda) = \frac{h(0) \Gamma(\sigma+1)}{\lambda^{\sigma+1}} \,e^{T\lambda} + O\!\left(\lambda^{-\sigma-2} e^{T\lambda}\right)
\]
as $\lambda \to \infty$ with $|\arg \lambda| \leq \theta$ for any fixed $0 \leq \theta < \pi/2$.
\label{watsoncoro}
\end{corollary}

\begin{proof}
We have
\begin{align*}
e^{-T\lambda} \Phi(\lambda) &= \int_0^T \varphi(t) e^{-\lambda(T-t)}\,dt \\
			&= \int_0^T \varphi(T-u) e^{-\lambda u}\,du \\
			&= \int_0^T u^{\sigma} h(u) e^{-\lambda u}\,du \\
			&= \frac{h(0) \Gamma(\sigma+1)}{\lambda^{\sigma+1}} + O\!\left(\lambda^{-\sigma-2}\right)
\end{align*}
as $\lambda \to \infty$ with $|\arg \lambda| \leq \theta$ for any fixed $0 \leq \theta < \pi/2$, by Watson's Lemma.
\end{proof}

Watson's Lemma and its corollary will be used in Chapter \ref{res} to determine the asymptotic character of the exponential integrals described in Section \ref{prelims:expints} as well as the coefficients of their power series.



\chapter{Main Results}
\label{res}

\ifpdf
    \graphicspath{{Chapter3/Chapter3Figs/PNG/}{Chapter3/Chapter3Figs/PDF/}{Chapter3/Chapter3Figs/}}
\else
    \graphicspath{{Chapter3/Chapter3Figs/EPS/}{Chapter3/Chapter3Figs/}}
\fi

Recall from Section \ref{prelims:expints} that we are concerned with functions of the form
\[
	F(z) = \int_{-a}^{b} \varphi(t) e^{zt} \,dt
\]
with $\varphi$ satisfying some light requirements.  The $n^\text{th}$ section of $F$ is the polynomial
\[
	s_n(F;z) = \sum_{k=0}^{n} \frac{z^k}{k!} \int_{-a}^{b} \varphi(t) t^k \,dt.
\]

The statement and proof of the main result depends on the relative sizes of $a$ and $b$ and of $\Re(\mu)$ and $\Re(\nu)$.  To this end, define $c = \max\{a,b\}$ and
\[
	\xi = \begin{cases}
			\Re(\mu) & \text{if } a > b, \\
			\Re(\nu) & \text{if } a < b, \\
			\min\{\Re(\mu),\Re(\nu)\} & \text{if } a = b.
		  \end{cases}
\]

Lastly, let $\{N\}$ be an increasing subsequence of the indices $\{n\}$ for which the sequence of sections $\{s_N(F;z)\}$ has a positive fraction of zeros in any sector with vertex at the origin (see Section \ref{prelims:strat}).  If $a=b$ and $\Re(\mu) = \Re(\nu)$, we also impose the condition that the indices $\{N\}$ are chosen so that quantity
\begin{equation}
	(-1)^N f_1(0) \Gamma(\mu+1) + f_2(0) \Gamma(\nu+1) a^{\nu-\mu} N^{\mu-\nu}
\label{Ncond}
\end{equation}
is bounded away from $0$.  Such a subsequence is guaranteed to exist by Theorem \ref{rosentheo}.  The condition in \eqref{Ncond} ensures that we can use the asymptotic representations derived in Lemmas \ref{coeffasymp2} and \ref{tailasymp2} without incident.

The main theorem is as follows.

\begin{theorem}
It is true that
\begin{enumerate}[label=(\roman*)]
\item Every point on the curve
\begin{align*}
	D_{a,b} &= \left\{z \in \C \,\colon \Re(z) \leq 0,\,\,\, |z| \leq \frac{1}{c},\,\,\, \text{and } \left|cze^{1+az}\right| = 1\right\} \\
		&\qquad \cup \left\{z \in \C \,\colon \Re(z) \geq 0,\,\,\, |z| \leq \frac{1}{c},\,\,\, \text{and } \left|cze^{1-bz}\right| = 1\right\}
\end{align*}
is a limit point of the zeros of the sections $s_N(F;Nz)$.
\item The only limit points of the zeros of the sections $s_N(F;Nz)$ on the imaginary axis are those on the line segment
\[
	D_{\text{imag}} = \left\{z \in \C \,\colon \Re(z) = 0 \,\,\,\text{and}\,\,\, |z| \leq \frac{1}{ec} \right\}.
\]
The limit points of the zeros of $F(Nz)$, if there are any, are a subset of the imaginary axis.  If every point on the imaginary axis is a limit point of these zeros, then every point on $D_{\text{imag}}$ is a limit point of the zeros of the sections $s_N(F;Nz)$.
\item Let $\{z_N\}$ be a sequence of complex numbers such that $s_N(F;Nz_N) = 0$ for all $N$ and such that the sequence has a limit point in the region
\[
	\left\{z \in \C \,\colon \Re(z) < 0 \,\,\,\text{and}\,\,\, z \neq -\frac{1}{a}\right\}.
\]
Then the elements of the sequence satisfy
\[
	\left|cz_Ne^{1+az_N}\right| = 1 + \left(\xi - \Re(\mu) + \frac{1}{2}\right) \frac{\log N}{N} + O(1/N)
\]
as $N \to \infty$.
\item Let $\{z_N\}$ be a sequence of complex numbers such that $s_N(F;Nz_N) = 0$ for all $N$ and such that the sequence has a limit point in the region
\[
	\left\{z \in \C \,\colon \Re(z) > 0 \,\,\,\text{and}\,\,\, z \neq \frac{1}{b}\right\}.
\]
Then the elements of the sequence satisfy
\[
	\left|cz_Ne^{1-bz_N}\right| = 1 + \left(\xi - \Re(\nu) + \frac{1}{2}\right) \frac{\log N}{N} + O(1/N)
\]
as $N \to \infty$.
\end{enumerate}
\label{maintheo}
\end{theorem}

It is interesting to note that the zeros will eventually approach the limit curve $D_{a,b}$ from the inside or from the outside depending on the signs of the quantities $\xi - \Re(\mu) + 1/2$ and $\xi - \Re(\nu) + 1/2$.  For example, if $\xi - \Re(\mu) + 1/2 > 0$ then the zeros will eventually approach the part of $D_{a,b}$ in the left half-plane from outside the curve.  If either of these quantities is zero then this theorem does not give any information about the direction from which the zeros approach the relevant part of the curve.

Figures \ref{eiplot1} through \ref{eiplot3norm} showcase the limit curve $D_{a,b} \cup D_{\text{imag}}$ and the zeros of the sections for a few exponential integrals.

The remainder of this chapter is dedicated to the proof of this theorem.  We begin by finding two asymptotic estimates we will require in our calculations.

\begin{lemma}
As $n \to \infty$,
\[
	F(nz) = f_1(0) \Gamma(\mu+1) (-nz)^{-\mu-1} e^{-anz} \Bigl(1 + O(1/n) \Bigr)
\]
when $z$ is restricted to a compact subset of $\Re(z) < 0$, and
\[
	F(nz) = f_2(0) \Gamma(\nu+1) (nz)^{-\nu-1} e^{bnz} \Bigl(1 + O(1/n) \Bigr)
\]
when $z$ is restricted to a compact subset of $\Re(z) > 0$.
\label{Fasymp}
\end{lemma}

\begin{proof}
This follows from a direct application of Corollary \ref{watsoncoro}.  To see this, suppose first that $z$ is restricted to a compact subset of $\Re(z) < 0$, and make the substitution $t = b - s$ in the integral for $F(nz)$ to get
\[
	F(nz) = \int_{-a}^{b} \varphi(t) e^{nzt}\,dt = e^{bnz} \int_{0}^{a+b} \varphi(b-s) e^{-nzs}\,ds,
\]
which, after replacing $z$ with $-z$, is of the form required by the corollary.  Next, suppose that $z$ is restricted to a compact subset of $\Re(z) > 0$, and make the substitution $t = s - a$ in the definition of $F(nz)$ to get
\[
	F(nz) = \int_{-a}^{b} \varphi(t) e^{nzt}\,dt = e^{-anz} \int_{0}^{a+b} \varphi(s-a) e^{nzs}\,ds,
\]
which is also of the required form.
\end{proof}

\begin{lemma}
We have
\begin{align*}
	\int_{-a}^{b} \varphi(t) t^n \,dt &= (-1)^n f_1(0) \Gamma(\mu+1) n^{-\mu-1} a^{n+\mu+1} + O\!\left(n^{-\mu-2} a^n\right) \\
			&\qquad + f_2(0) \Gamma(\nu+1) n^{-\nu-1} b^{n+\nu+1} + O\!\left(n^{-\nu-2} b^n\right).
\end{align*}
as $n \to \infty$.
\label{coeffasymp}
\end{lemma}

\begin{proof}
If $a \neq 0$ we calculate
\begin{align*}
\int_{-a}^{0} \varphi(t) t^n \,dt &= (-a)^n \int_{-a}^{0} \varphi(t) e^{n \log(-t/a)} \,dt \\
			&= (-a)^n \int_0^a \varphi(s-a) e^{n \log(1-s/a)} \,ds \\
			&= (-a)^n \int_0^a s^{\mu} f_1(s) e^{n \log(1-s/a)} \,ds.
\end{align*}
Letting $s = a(1-e^{-r})$ gives
\begin{align*}
\int_{-a}^{0} \varphi(t) t^n \,dt &= (-1)^n a^{n+\mu+1} \int_{0}^{\infty} (1-e^{-r})^{\mu} f_1(a-ae^{-r}) e^{-r} e^{-nr} \,dr \\
			&= (-1)^n a^{n+\mu+1} \int_{0}^{\infty} r^{\mu} \psi_a(r) e^{-nr} \,dr,
\end{align*}
where
\[
	\psi_a(r) = \left(\frac{1-e^{-r}}{r}\right)^{\mu} f_1(a-ae^{-r}) e^{-r}
\]
has a bounded derivative in a neighborhood of $r=0$.  We may now apply Watson's Lemma to conclude that
\begin{align*}
\int_{-a}^{0} \varphi(t) t^n \,dt &= (-1)^n \psi_a(0) \Gamma(\mu+1) n^{-\mu-1} a^{n+\mu+1} + O\!\left(n^{-\mu-2} a^n\right) \\
			&= (-1)^n f_1(0) \Gamma(\mu+1) n^{-\mu-1} a^{n+\mu+1} + O\!\left(n^{-\mu-2} a^n\right).
\end{align*}

Using an identical argument we find that
\[
	\int_0^b \varphi(t) t^n \,dt = f_2(0) \Gamma(\nu+1) n^{-\nu-1} b^{n+\nu+1} + O\!\left(n^{-\nu-2} b^n\right),
\]
which completes the proof.
\end{proof}

The actual asymptotic character of the integral in the above lemma depends on the relative sizes of $a$ and $b$ and of $\Re(\mu)$ and $\Re(\nu)$, so for convenience we'll organize the possible outcomes in a separate lemma.

\begin{lemma}
If $a>b$ then
\[
	\int_{-a}^{b} \varphi(t) t^n \,dt = (-1)^{n} f_1(0) \Gamma(\mu+1) \,n^{-\mu-1} a^{n+\mu+1} \Bigl(1 + O(1/n)\Bigr),
\]
if $a<b$ then
\[
	\int_{-a}^{b} \varphi(t) t^n \,dt = f_2(0) \Gamma(\nu+1) \,n^{-\nu-1} b^{n+\nu+1} \Bigl(1 + O(1/n)\Bigr),
\]
and if $a=b$ we have three cases:
\begin{enumerate}[label=(\roman*)]
\item if $\Re(\mu) > \Re(\nu)$ then
\[
	\int_{-a}^{b} \varphi(t) t^n \,dt = f_2(0) \Gamma(\nu+1) \,n^{-\nu-1} a^{n+\nu+1} \Bigl(1 + O\!\left(n^{\Re(\nu) - \Re(\mu)}\right) + O(1/n)\Bigr),
\]
\item if $\Re(\mu) < \Re(\nu)$ then
\begin{align*}
	&\int_{-a}^{b} \varphi(t) t^n \,dt \\
	&\qquad = (-1)^{n+1} f_1(0) \Gamma(\mu+1) \,n^{-\mu-1} a^{n+\mu+1} \Bigl(1 + O\!\left(n^{\Re(\mu) - \Re(\nu)}\right) + O(1/n)\Bigr),
\end{align*}
\item and if $\Re(\mu) = \Re(\nu)$ then
\[
	\int_{-a}^{b} \varphi(t) t^n \,dt = G_1(a,\mu,\nu,n,z) \,n^{-\mu-1} a^{n+\mu+1} \Bigl(1 + O(1/n)\Bigr),
\]
where
\[
	G_1(a,\mu,\nu,n) = (-1)^{n+1} f_1(0) \Gamma(\mu+1) + f_2(0) \Gamma(\nu+1) a^{\nu-\mu} n^{\mu-\nu},
\]
if $G_1(a,\mu,\nu,n) \neq 0$,
\end{enumerate}
as $n \to \infty$.
\label{coeffasymp2}
\end{lemma}

\begin{proof}
This follows directly from the statement of Lemma \ref{coeffasymp}.
\end{proof}

\section{Restricting the Zeros}
\label{restrict}

In this section we will prove the following lemma.

\begin{lemma}
The limit points of the zeros of the sections $s_N(F;Nz)$ lie in the disk $|z| \leq 1/c$, where $c = \max\{a,b\}$.
\label{restrictlemma}
\end{lemma}

From Stirling's formula\phantomsection\newnot{symbol:sim}
\[
	n! \sim \left(\frac{n}{e}\right)^n \sqrt{2\pi n}
\]
we have
\[
	(n!)^{1/n} \sim \frac{n}{e}
\]
as $n \to \infty$, and with the aid of Lemma \ref{coeffasymp2} we calculate
\[
	\left|\int_{-a}^{b} \varphi(t) t^N \,dt\right|^{-1/N} \longrightarrow \frac{1}{c}
\]
as $N \to \infty$, where the subsequence of indices $\{N\}$ is as defined above Theorem \ref{maintheo}.  Combining these we see that the power series coefficients of $F(z)$, which are given by
\[
	a_n = \frac{1}{n!} \int_{-a}^{b} \varphi(t) t^n \,dt,
\]
satisfy
\[
	\rho_N = |a_N|^{-1/N} \sim \frac{N}{ec},
\]
where $\rho_N$ is as defined in Theorem \ref{rosentheo}.  Thus the order $\rho$ of $F$ is calculated to be
\[
	\rho = \limsup_{n \to \infty} \frac{\log n}{\log \rho_n} = 1,
\]
where all indices $n$ are taken into account (see, e.g., \cite[p. 9]{boas:entirefunctions} or \cite[p. 326]{saks:analyticfunctions}).

By Theorem \ref{rosentheo} we have, for any $\epsilon > 0$,
\[
	\liminf_{N \to \infty} \frac{\sharp_N^{\circ}\Bigl((1 + \epsilon)N/c\Bigr)}{N} = 1,
\]
where $\sharp_N^{\circ}\bigl((1 + \epsilon)N/c\bigr)$ is the number of zeros of $s_N(F;z)$ in the disk $|z| \leq (1 + \epsilon)N/c$. From this we conclude that the limit points of the zeros of the sections $s_N(F;Nz)$ lie in the disk $|z| \leq 1/c$, as desired.

\section{Gathering the Formulas}
\label{setup}

By definition we have
\[
	F(nz) = \int_{-a}^{b} \varphi(t) e^{nzt} \,dt,
\]
and
\[
	s_n(F;nz) = \int_{-a}^{b} \varphi(t) s_n(\exp;nzt) \,dt.
\]
Subtracting these we get
\begin{align}
	F(nz) - s_n(F;nz) &= \int_{-a}^{b} \varphi(t) \left(e^{nzt} - s_n(\exp;nzt)\right)\,dt \nonumber \\
					  &= \int_{-a}^{b} \varphi(t) e^{nzt} g_n(zt)\,dt,
\label{Fsndiff}
\end{align}
where
\[
	g_n(z) = 1 - e^{-nz}s_n(\exp;nz).
\]
It was shown by Szeg\H{o} in \cite{szego:exp} (see also \cite{cvw:expasympi}, \cite{boyergoh:euler}, and \cite{norfolk:1f1}) that
\[
	g_n(z) = \frac{\left(ze^{1-z}\right)^n}{\sqrt{2 \pi n}} \cdot \frac{z}{1-z} \Bigl(1 - \epsilon_n(z)\Bigr),
\]
where $\epsilon_n(z) = O(1/n)$ as $n \to \infty$ uniformly when $z$ is restricted to a compact subset of $\Re(z) < 1$.  Upon substituting this into equation \eqref{Fsndiff} we get
\begin{equation}
	F(nz) - s_n(F;nz) = \frac{e^n z^{n+1}}{\sqrt{2 \pi n}} \int_{-a}^{b} \frac{\varphi(t)}{1-zt} \,t^{n+1} \Bigl(1 - \epsilon_n(zt)\Bigr) \,dt.
\label{maindifference}
\end{equation}
Our next step is to estimate this integral.  The following lemma is proved by an argument similar to the one used to prove Lemma \ref{coeffasymp}.

\begin{lemma}
\begin{align*}
\int_{-a}^{b} \frac{\varphi(t)}{1-zt} \,t^{n+1} \,dt &= (-1)^{n+1} \frac{f_1(0) \Gamma(\mu+1)}{1+az} \,n^{-\mu-1} a^{n+\mu+2} + O\!\left(n^{-\mu-2} a^n\right) \\
			&\qquad + \frac{f_2(0) \Gamma(\nu+1)}{1-bz} \,n^{-\nu-1} b^{n+\nu+2} + O\!\left(n^{-\nu-2} b^n\right)
\end{align*}
as $n \to \infty$ uniformly when $z$ is restricted to a compact subset of the doubly-punctured plane $\{z \in \C \,\colon z \neq -1/a \,\,\,\text{and}\,\,\, z \neq 1/b\}$.
\label{tailasymp}
\end{lemma}

As in Lemma \ref{coeffasymp}, the actual asymptotic character of the integral in the above lemma depends on the relative sizes of $a$ and $b$ and of $\Re(\mu)$ and $\Re(\nu)$, so we will again organize the possible outcomes in a separate lemma.

\begin{lemma}
If $a>b$ then
\[
	\int_{-a}^{b} \frac{\varphi(t)}{1-zt} \,t^{n+1} \,dt = (-1)^{n+1} \frac{f_1(0) \Gamma(\mu+1)}{1+az} \,n^{-\mu-1} a^{n+\mu+2} \Bigl(1 + O(1/n)\Bigr),
\]
if $a<b$ then
\[
	\int_{-a}^{b} \frac{\varphi(t)}{1-zt} \,t^{n+1} \,dt = \frac{f_2(0) \Gamma(\nu+1)}{1-bz} \,n^{-\nu-1} b^{n+\nu+2} \Bigl(1 + O(1/n)\Bigr),
\]
and if $a=b$ we have three cases:
\begin{enumerate}[label=(\roman*)]
\item if $\Re(\mu) > \Re(\nu)$ then
\[
	\int_{-a}^{b} \frac{\varphi(t)}{1-zt} \,t^{n+1} \,dt = \frac{f_2(0) \Gamma(\nu+1)}{1-az} \,n^{-\nu-1} a^{n+\nu+2} \Bigl(1 + O\!\left(n^{\Re(\nu) - \Re(\mu)}\right) + O(1/n)\Bigr),
\]
\item if $\Re(\mu) < \Re(\nu)$ then
\begin{align*}
	&\int_{-a}^{b} \frac{\varphi(t)}{1-zt} \,t^{n+1} \,dt \\
	&\qquad = (-1)^{n+1} \frac{f_1(0) \Gamma(\mu+1)}{1+az} \,n^{-\mu-1} a^{n+\mu+2} \Bigl(1 + O\!\left(n^{\Re(\mu) - \Re(\nu)}\right) + O(1/n)\Bigr),
\end{align*}
\item and if $\Re(\mu) = \Re(\nu)$ then
\[
	\int_{-a}^{b} \frac{\varphi(t)}{1-zt} \,t^{n+1} \,dt = \frac{G_2(a,\mu,\nu,n,z)}{1-a^2 z^2} \,n^{-\mu-1} a^{n+\mu+2} \Bigl(1 + O(1/n)\Bigr),
\]
where
\begin{align*}
	&G_2(a,\mu,\nu,n,z) \\
	&\qquad = (-1)^{n+1} f_1(0) \Gamma(\mu+1) (1-az) + f_2(0) \Gamma(\nu+1) (1+az) a^{\nu-\mu} n^{\mu-\nu},
\end{align*}
if $G_2(a,\mu,\nu,n,z) \neq 0$,
\end{enumerate}
as $n \to \infty$ uniformly when $z$ is restricted to a compact subset of the doubly-punctured plane $\{z \in \C \,\colon z \neq -1/a \,\,\,\text{and}\,\,\, z \neq 1/b\}$.
\label{tailasymp2}
\end{lemma}

\begin{proof}
This follows directly from the statement of Lemma \ref{tailasymp}.
\end{proof}

As a consequence of these two lemmas, equation \eqref{maindifference} becomes
\begin{equation}
	F(nz) = \frac{e^n z^{n+1}}{\sqrt{2 \pi n}} \int_{-a}^{b} \frac{\varphi(t)}{1-zt} \,t^{n+1} \,dt \,\Bigl(1 - O(1/n)\Bigr)
\label{maindifferencezero}
\end{equation}
when $z$ is a zero of the scaled section $s_n(F;nz)$ in the region $|\Re(z)| < 1/c$, where \linebreak $c = \max\{a,b\}$.

\section{Concluding the Argument}
\label{argend}

Let $\{N\}$ be a sequence of indices as defined above Theorem \ref{maintheo}.  We showed in Section \ref{restrict} that the limit points of the zeros of the sections $s_N(F;Nz)$ lie in the disk $|z| \leq 1/c$, where $c = \max\{a,b\}$.  Define
\[
	\Delta = \left\{z \in \C \,\colon |z| \leq \frac{1}{c}, \,\,\,z \neq -\frac{1}{a}, \,\,\,\text{and}\,\,\, z \neq \frac{1}{b}\right\}.
\]

Suppose first that $\{z_N\}$ is a sequence in $\C$ such that $s_N(F;Nz_N) = 0$ for all $N$ and such that the sequence has a limit point in $\Delta \cap \{z \in \C \,\colon \Re(z) < 0\}$.  Note that this implies there is a $\delta > 0$ such that $|z_N + 1/a| > \delta$ for $N$ large enough.

It follows from Lemma \ref{Fasymp} that
\begin{equation}
	|F(Nz_N)|^{1/N} = |e^{-az_N}| \left(1 - \left(\Re(\mu) + 1\right) \frac{\log N}{N} + O(1/N)\right)
\label{leftasymp1}
\end{equation}
as $N \to \infty$, and if
\[
	\xi = \begin{cases}
			\Re(\mu) & \text{if } a > b, \\
			\Re(\nu) & \text{if } a < b, \\
			\min\{\Re(\mu),\Re(\nu)\} & \text{if } a = b,
		  \end{cases}
\]
then we have from Lemma \ref{tailasymp2} that
\begin{equation}
	\left|\frac{e^N z^{N+1}}{\sqrt{2\pi N}} \int_{-a}^{b} \frac{\varphi(t)}{1-zt} \,t^{N+1} \,dt\right|^{1/N} = |ecz_N| \left(1 - \left(\xi + \frac{3}{2}\right) \frac{\log N}{N} + O(1/N)\right)
\label{rightasymp}
\end{equation}
as $N \to \infty$.  Upon substituting equations \eqref{leftasymp1} and \eqref{rightasymp} into equation \eqref{maindifferencezero} we see that these zeros $z_N$ satisfy
\[
	\left|c z_N e^{1+az_N}\right| = 1 + \left(\xi - \Re(\mu) + \frac{1}{2}\right) \frac{\log N}{N} + O(1/N)
\]
as $N \to \infty$, which proves part (iii) of Theorem \ref{maintheo}.

Suppose now that $\{z_N\}$ is a sequence such that $s_N(F;Nz_N) = 0$ for all $N$ and such that the sequence has a limit point in $\Delta \cap \{z \in \C \,\colon \Re(z) > 0\}$.  Note that this implies there is a $\delta > 0$ such that $|z_N - 1/b| > \delta$ for $N$ large enough.

Here it follows from Lemma \ref{Fasymp} that
\begin{equation}
	|F(Nz_N)|^{1/N} = |e^{bz_N}| \left(1 - \left(\Re(\nu) + 1\right) \frac{\log N}{N} + O(1/N)\right)
\end{equation}
as $N \to \infty$.  Subtituting this and equation \eqref{rightasymp} into equation \eqref{maindifferencezero} we see that these zeros $z_N$ satisfy
\[
	\left|c z_N e^{1-bz_N}\right| = 1 + \left(\xi - \Re(\nu) + \frac{1}{2}\right) \frac{\log N}{N} + O(1/N)
\]
as $N \to \infty$, which proves part (iv) of Theorem \ref{maintheo}.

We have so far shown that the limit points of the zeros of the sections $s_N(F;Nz)$ with $\Re(z) \neq 0$ must lie on the curve $D_{a,b}$ as defined in part (i) of Theorem \ref{maintheo}.  That every point on $D_{a,b}$ is such a limit point follows from the choice of the subsequence $\{N\}$ which ensures that the sequence of sections $\{s_N(F;z)\}$ has a positive fraction of zeros in any sector with vertex at the origin.  It is straightforward to show that for any $0 \leq \theta < 2\pi$ there is a unique $r > 0$ such that $re^{i\theta} \in D_{a,b}$.  This proves part (i) of Theorem \ref{maintheo}.

Finally we will prove part (ii).  From the asymptotic expansion for $F(nz)$ in Lemma \ref{Fasymp} we see that the limit points of the zeros of $F(nz)$ must lie on the imaginary axis.

If $\{z_N\}$ is a sequence of complex numbers such that $s_N(F;Nz_N) = 0$ for all $N$ and such that the sequence has a limit point in $\{z \in \C \,\colon \Re(z) = 0 \,\,\,\text{and}\,\,\, \Im(z) > 1/(ec)\}$ then by equation \eqref{maindifferencezero} and Lemma \ref{tailasymp2} we must have $F(Nz_N) \to \infty$.  However, if $y \in \R$ then
\[
	|F(iny)| = \left|\int_{-a}^{b} \varphi(t) e^{inyt} \,dt\right| \leq \int_{-a}^{b} |\varphi(t)| \,dt,
\]
so that such a sequence of zeros cannot exist.  Hence any limit points on the imaginary axis must satisfy $\Im(z) \leq 1/(ec)$.

Further, if $\{z_N\}$ is a sequence of zeros which has a limit point in $|z| \leq 1/(ec)$, then by equation \eqref{maindifferencezero} and Lemma \ref{tailasymp2} we must have $F(Nz_N) \to 0$.  In other words, the zeros $\{z_N\}$ must approximate the zeros of $F(Nz)$.  Thus part (ii) is proved, completing the proof of Theorem \ref{maintheo}.

\section{A Few Examples}
\label{res:examp}

In this section we'll look at what Theorem \ref{maintheo} says about a few example exponential integrals.

In Figures \ref{eiplot1} and \ref{eiplot1norm} we plot the zeros of the sections and of the normalized sections, respectively, of the function
\[
	F_1(z) = \int_{-2}^{3/2} (3/2-t)^{1/2+i} e^{zt} \,dt.
\]
In the notation of Theorem \ref{maintheo} we have $a=2$, $b=3/2$, $\mu = 0$, and $\nu = 1/2+i$.  The Szeg\H{o} curve associated with the normalized zeros is the set
\begin{align*}
	D_{2,3/2} \cup D_{\text{imag}} &= \left\{z \in \C \,\colon \Re(z) \leq 0,\,\,\, |z| \leq \frac{1}{2},\,\,\, \text{and } \left|2ze^{1+2z}\right| = 1\right\} \\
		&\qquad \cup \left\{z \in \C \,\colon \Re(z) \geq 0,\,\,\, |z| \leq \frac{1}{2},\,\,\, \text{and } \left|2ze^{1-\frac{3}{2}z}\right| = 1\right\} \\
		&\qquad \cup \left\{z \in \C \,\colon \Re(z) = 0 \,\,\,\text{and}\,\,\, |z| \leq \frac{1}{2e}\right\}.
\end{align*}
Zeros of $s_N(F_1;Nz)$ which converge in $\Re(z) < 0$ to a point different from $z=-1/2$ satisfy
\[
	\left|2ze^{1+2z}\right| = 1 + \frac{\log N}{2 N} + O(1/N)
\]
and thus eventually approach the curve $D_{2,3/2}$ from the outside.  Zeros which converge in $\Re(z) > 0$ satisfy
\[
	\left|2ze^{1-\frac{3}{2}z}\right| = 1 + O(1/N).
\]

In Figures \ref{eiplot2} and \ref{eiplot2norm} we plot the zeros of the sections and of the normalized sections, respectively, of the function
\[
	F_2(z) = \int_{-1}^{1} (1-t)^4 (t+1)^{-1/2-2i} e^{zt} \,dt.
\]
In the notation of Theorem \ref{maintheo} we have $a=b=1$, $\mu = -1/2-2i$, and $\nu = 4$.  The Szeg\H{o} curve associated with the normalized zeros is the set
\begin{align*}
	D_{1,1} \cup D_{\text{imag}} &= \left\{z \in \C \,\colon \Re(z) \leq 0,\,\,\, |z| \leq 1,\,\,\, \text{and } \left|ze^{1+z}\right| = 1\right\} \\
		&\qquad \cup \left\{z \in \C \,\colon \Re(z) \geq 0,\,\,\, |z| \leq 1,\,\,\, \text{and } \left|ze^{1-z}\right| = 1\right\} \\
		&\qquad \cup \left\{z \in \C \,\colon \Re(z) = 0 \,\,\,\text{and}\,\,\, |z| \leq 1/e\right\}.
\end{align*}
Zeros of $s_N(F_2;Nz)$ which converge in $\Re(z) < 0$ to a point different from $z=-1$ satisfy
\[
	\left|ze^{1+z}\right| = 1 + \frac{\log N}{2 N} + O(1/N)
\]
and thus eventually approach the curve $D_{1,1}$ from the outside.  Zeros which converge in $\Re(z) > 0$ to a point different from $z=1$ satisfy
\[
	\left|ze^{1-z}\right| = 1 - \frac{4 \log N}{N} + O(1/N)
\]
and thus eventually approach $D_{1,1}$ from the inside.

In Figures \ref{eiplot3} and \ref{eiplot3norm} we plot the zeros of the sections and of the normalized sections, respectively, of the function
\[
	F_3(z) = \int_{-17/36}^{19/36} (t-1/2)^2 e^{zt} \,dt.
\]
In the notation of Theorem \ref{maintheo} we have $a = 17/36$, $b=19/36$, and $\mu = \nu = 0$.  The Szeg\H{o} curve associated with the normalized zeros is the set
\begin{align*}
	D_{17/36,19/36} \cup D_{\text{imag}} &= \left\{z \in \C \,\colon \Re(z) \leq 0,\,\,\, |z| \leq \frac{36}{19},\,\,\, \text{and } \left|\frac{19}{36}ze^{1+\frac{17}{36}z}\right| = 1\right\} \\
		&\qquad \cup \left\{z \in \C \,\colon \Re(z) \geq 0,\,\,\, |z| \leq \frac{36}{19},\,\,\, \text{and } \left|\frac{19}{36}ze^{1-\frac{19}{36}z}\right| = 1\right\} \\
		&\qquad \cup \left\{z \in \C \,\colon \Re(z) = 0 \,\,\,\text{and}\,\,\, |z| \leq \frac{36}{19e}\right\}.
\end{align*}
Zeros of $s_N(F_3;Nz)$ which converge in $\Re(z) < 0$ satisfy
\[
	\left|\frac{19}{36}ze^{1+\frac{17}{36}z}\right| = 1 + \frac{\log N}{2N} + O(1/N)
\]
and thus approach the curve $D_{17/36,19/36}$ from the outside.  Zeros which converge in $\Re(z) > 0$ to a point different from $z=36/19$ satisfy
\[
	\left|\frac{19}{36}ze^{1-\frac{19}{36}z}\right| = 1 + \frac{\log N}{2N} + O(1/N)
\]
and also approach $D_{17/36,19/36}$ from the outside.

\begin{landscape}

\begin{figure}[htb]
	\centering
	\includegraphics[height=0.89\textheight]{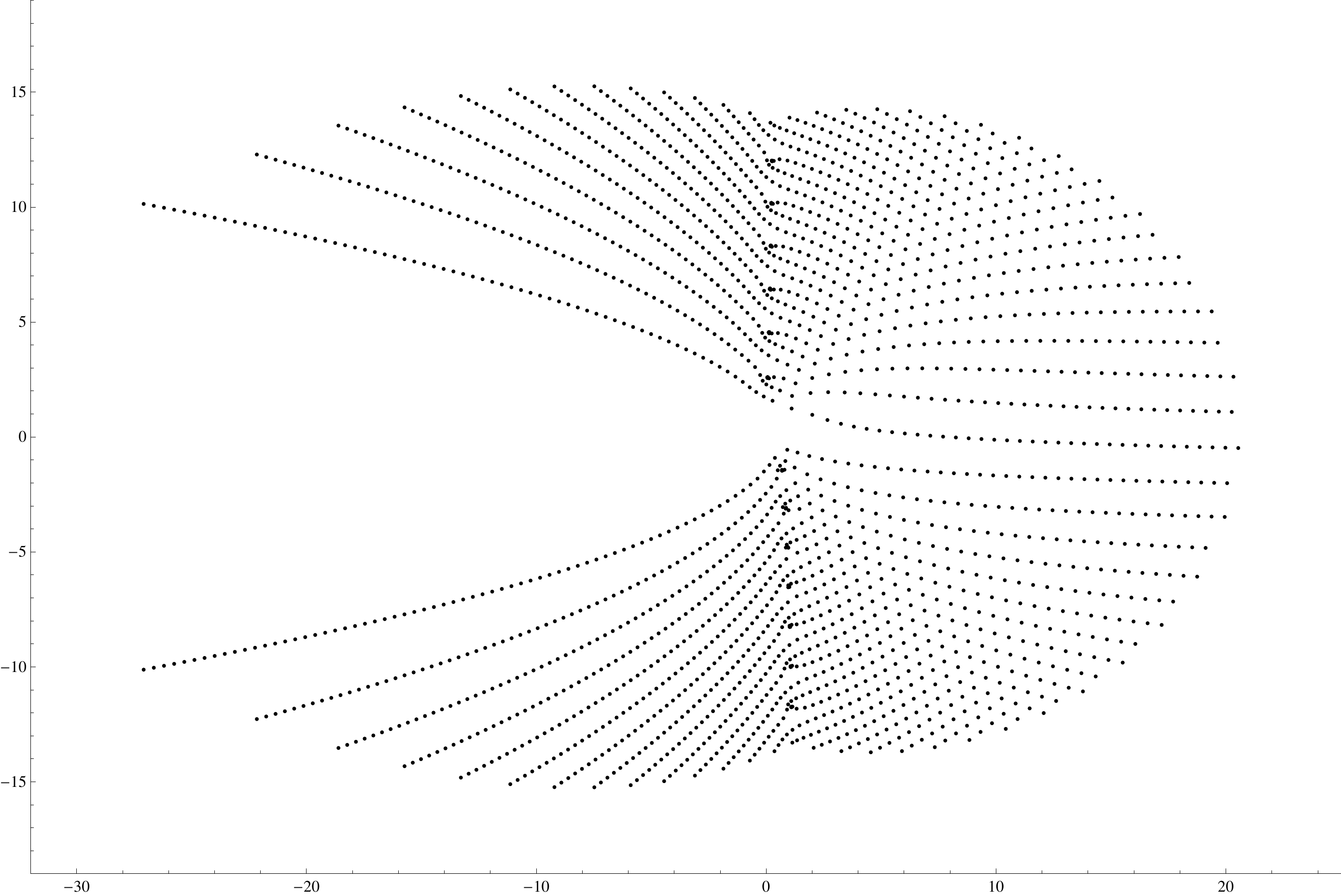}
	\caption{Zeros of the sections $s_n(F_1;z)$ ($n=1,2,\ldots,70$) for $F_1(z) = \int_{-2}^{3/2} (3/2-t)^{1/2+i} e^{zt} \,dt$.}
\label{eiplot1}
\end{figure}

\begin{figure}[htb]
	\centering
	\includegraphics[height=0.92\textheight]{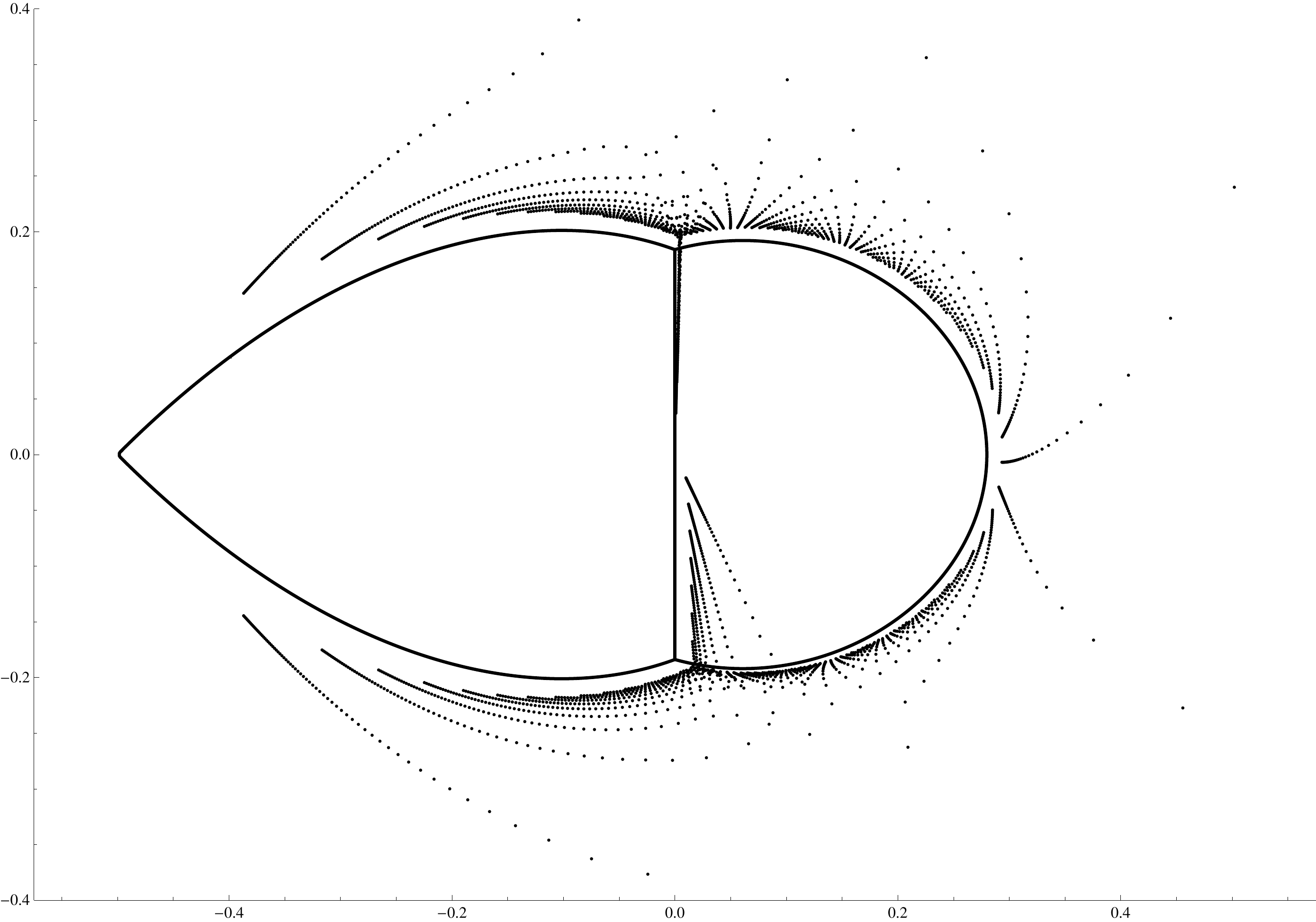}
	\caption{Zeros of the normalized sections $s_n(F_1;nz)$ ($n=1,2,\ldots,70$) for $F_1(z) = \int_{-2}^{3/2} (3/2-t)^{1/2+i} e^{zt} \,dt$.}
\label{eiplot1norm}
\end{figure}

\begin{figure}[htb]
	\centering
	\includegraphics[height=0.89\textheight]{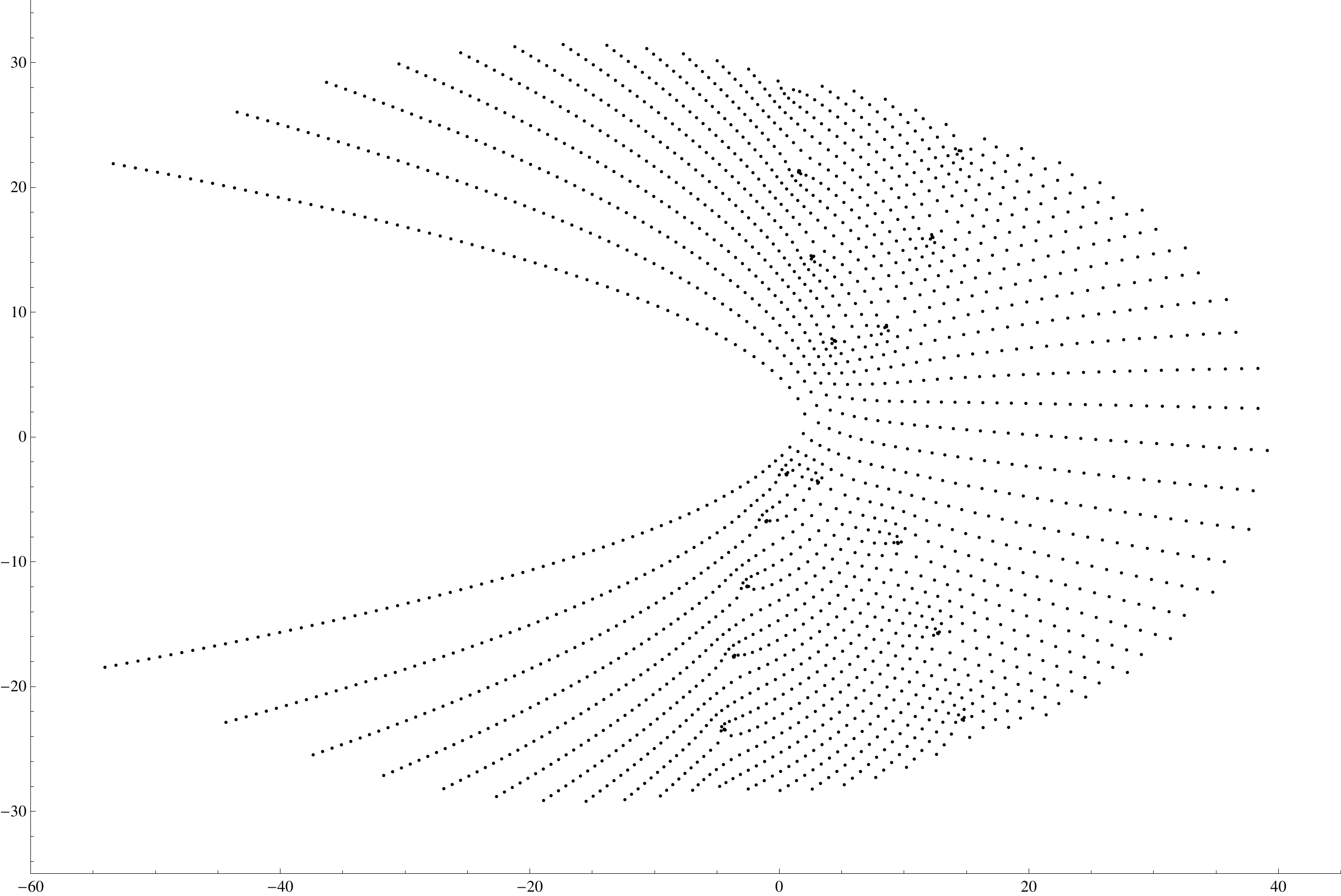}
	\caption{Zeros of the sections $s_n(F_2;z)$ ($n=1,2,\ldots,70$) for $F_2(z) = \int_{-1}^{1} (1-t)^4 (t+1)^{-1/2-2i} e^{zt} \,dt$.}
\label{eiplot2}
\end{figure}

\begin{figure}[htb]
	\centering
	\includegraphics[height=0.84\textheight]{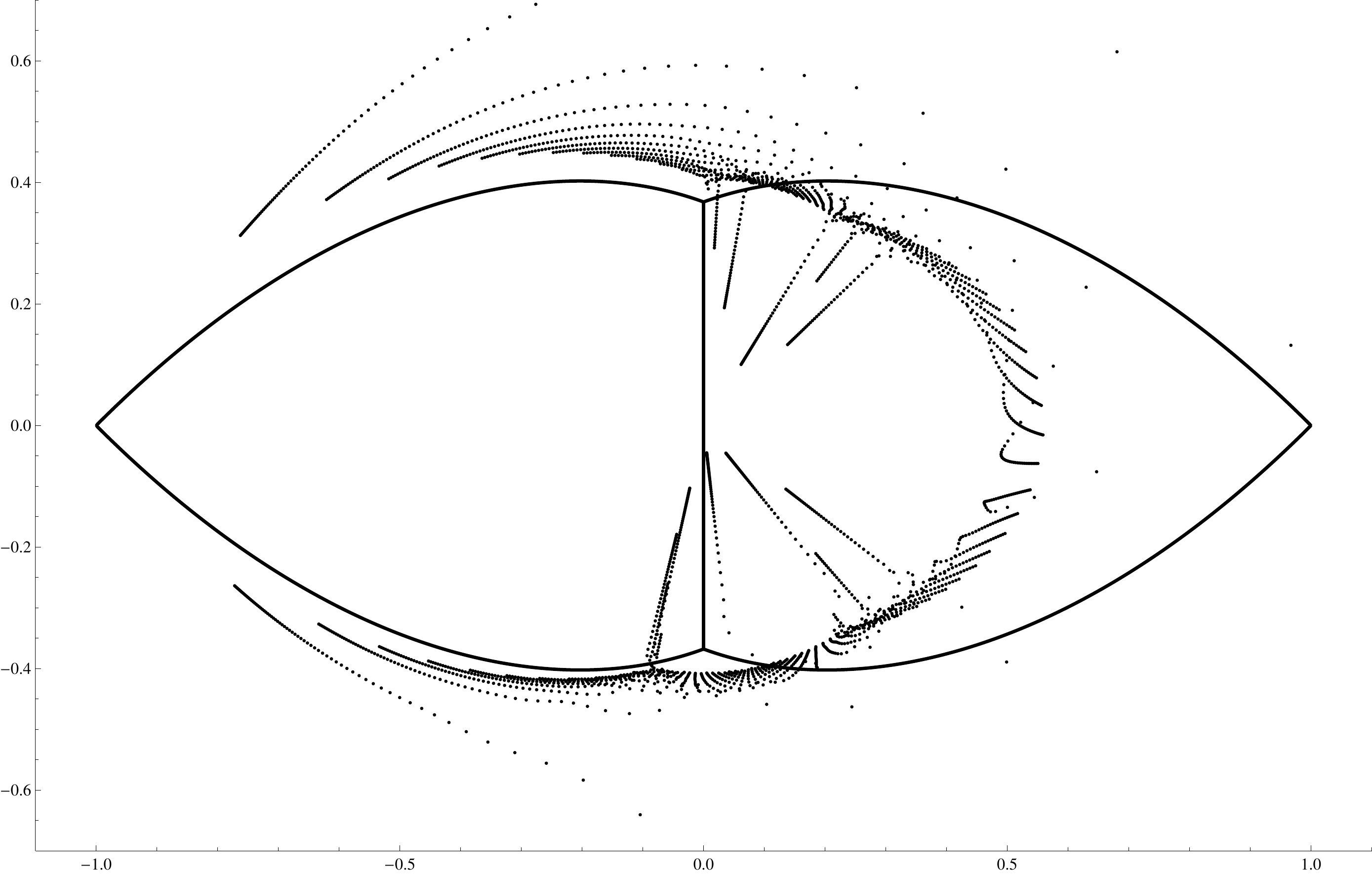}
	\caption{Zeros of the normalized sections $s_n(F_2;nz)$ ($n=1,2,\ldots,70$) for $F_2(z) = \int_{-1}^{1} (1-t)^4 (t+1)^{-1/2-2i} e^{zt} \,dt$.}
\label{eiplot2norm}
\end{figure}

\begin{figure}[htb]
	\centering
	\includegraphics[height=0.80\textheight]{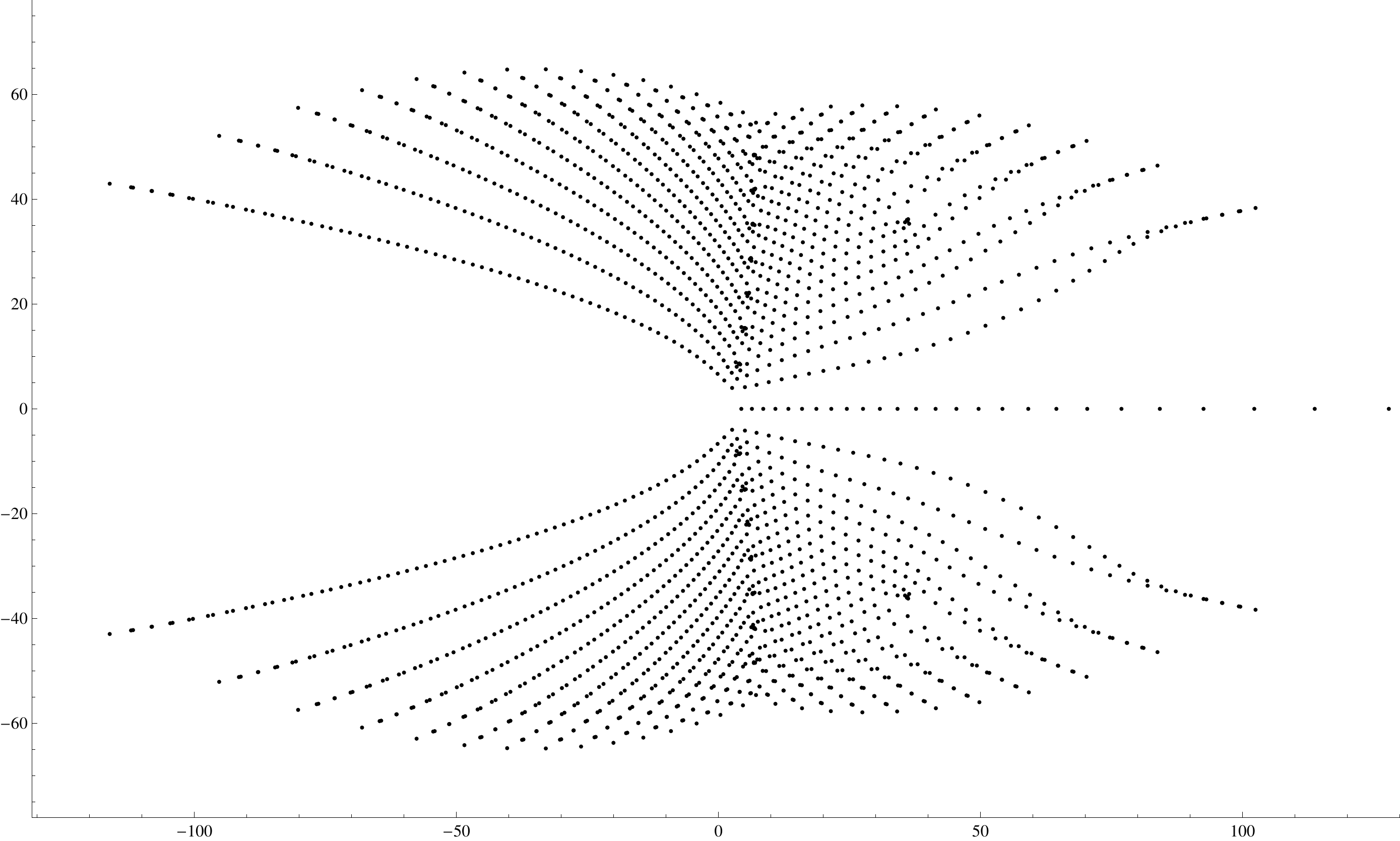}
	\caption{Zeros of the sections $s_n(F_3;z)$ ($n=1,2,\ldots,70$) for $F_3(z) = \int_{-17/36}^{19/36} (t-1/2)^2 e^{zt} \,dt$.}
\label{eiplot3}
\end{figure}

\begin{figure}[htb]
	\centering
	\includegraphics[height=0.92\textheight]{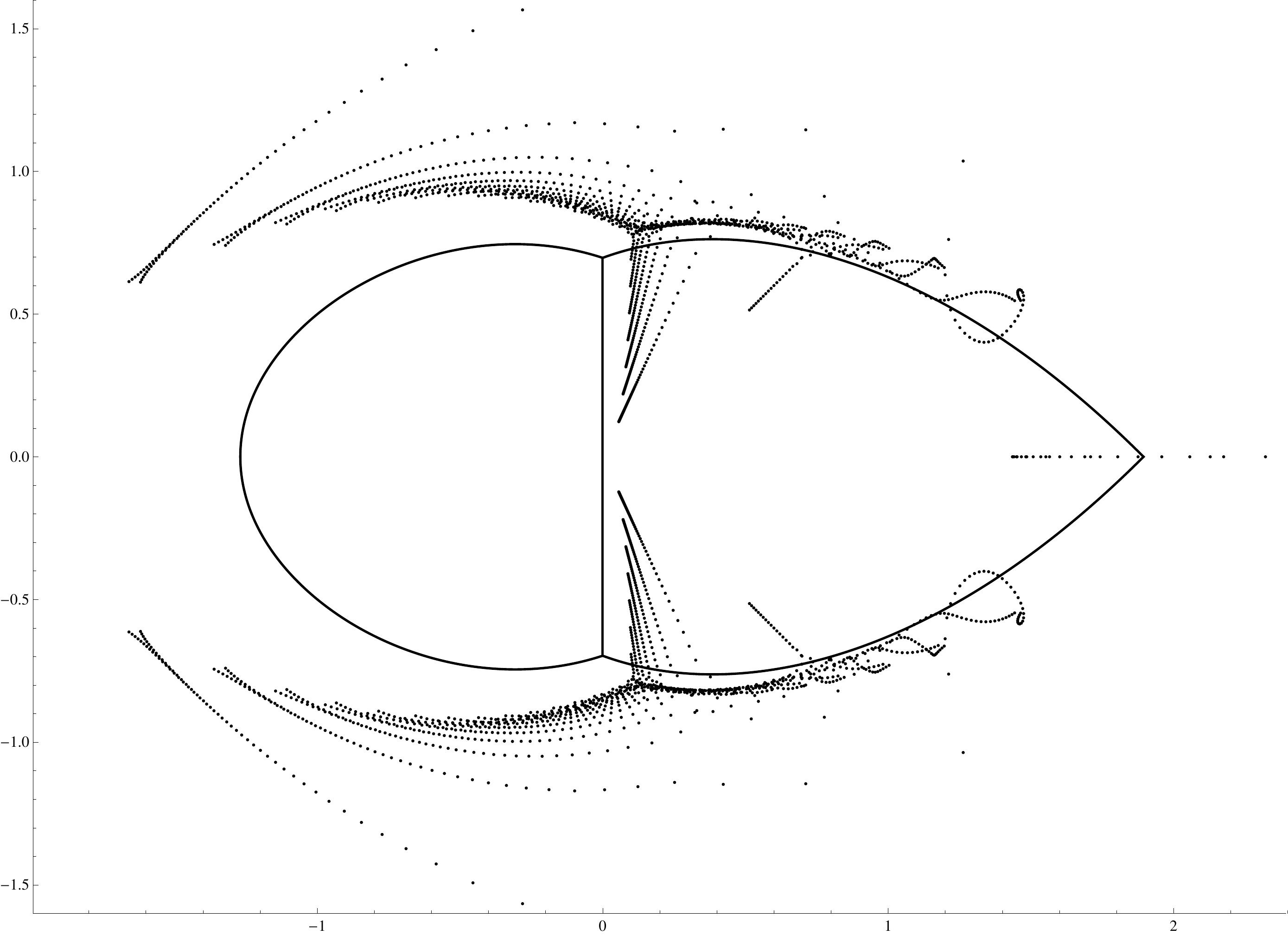}
	\caption{Zeros of the normalized sections $s_n(F_3;nz)$ ($n=1,2,\ldots,70$) for $F_3(z) = \int_{-17/36}^{19/36} (t-1/2)^2 e^{zt} \,dt$.}
\label{eiplot3norm}
\end{figure}

\end{landscape}



\chapter{Discussion}
\label{relev}

\ifpdf
    \graphicspath{{Chapter4/Chapter4Figs/PNG/}{Chapter4/Chapter4Figs/PDF/}{Chapter4/Chapter4Figs/}}
\else
    \graphicspath{{Chapter4/Chapter4Figs/EPS/}{Chapter4/Chapter4Figs/}}
\fi

Rosenbloom's main theorem, Theorem \ref{mainrosentheo}, is quite general.  For instance, it can be used in conjunction with Theorem \ref{rosentheo} to determine the limit curve described in part (i) of our Theorem \ref{maintheo}.  The new result in this current thesis is the derivation of the asymptotic order at which the zeros approach this limit curve as well as information about the direction at which they do so in most cases.  This information is given in parts (iii) and (iv) of Theorem \ref{maintheo}.

In light of Theorem \ref{mainrosentheo}, one consequence of Lemma \ref{Fasymp} is that the zeros of the sections of the exponential integrals we have studied have the same limit curves as the zeros of particular exponential sums.  Indeed, if $a,b > 0$ and
\[
	F(z) = \int_{-a}^{b} \varphi(t) e^{zt} \,dt,
\]
where $\varphi$ is an appropriate extended-complex-valued function as described in Chapter \ref{prelims}, the zeros of the sections $s_n(F;nz)$ have the same limit curve as the zeros of the sections of the sum
\begin{equation}
	e^{-az} + e^{bz}.
\label{equivexpsum}
\end{equation}
Zeros of sections of exponential sums such as these were previously studied by Bleher and Mallison \cite{mallison:expsums} (see Figure \ref{expsumplots} and the paragraph immediately following Theorem \ref{mainrosentheo}).

One interesting aspect of the current work is that the asymptotic rate of approach of the zeros of the sections $s_n(F;nz)$ depends on the order of the critical points of $\varphi(t)$ at $t=-a$ and $t=b$ (see parts (iii) and (iv) of Theorem \ref{maintheo}).  In most cases the zeros approach the limit curve at a rate of $c \log n/n$ for some nonzero constant $c$ (either positive or negative), and in other cases at a rate of $O(1/n)$.  This stands in contrast with the work of Bleher and Mallison, who showed that the zeros of sections of exponential sums, such as the one in equation \eqref{equivexpsum}, always approach the arcs of the limit curve at a rate of $c \log n/n$ for some positive constant $c$.

\section{Special Cases of the Exponential Integrals}
\label{relev:cases}

As discussed in Section \ref{norfolkinspiration}, the confluent hypergeometric functions
\[
	{}_1F_1(1;b;z) = (b-1) \int_0^1 (1-t)^{b-2} e^{zt} \,dt
\]
with $b > 1$ studied by Norfolk \cite{norfolk:1f1} are a special case of the exponential integrals studied in this thesis.  Our result extends some of Norfolk's results to the case of complex $b$ with $\Re(b) > 1$.  A few of his results, notably the asymptotic rate of approach of the zeros to the ``corners'' of the Szeg\H{o} curve, were not replicated and will be treated in future work.

Another important special case of the exponential integrals is the class of Bessel functions of the first kind
\[
	J_{\alpha}(z) = \left(\frac{z}{2}\right)^{\alpha} \sum_{k=0}^{\infty} \frac{(-1)^k}{\Gamma(k+1)\Gamma(k+\alpha+1)} \left(\frac{z}{2}\right)^{2k}.
\]
For $\Re(\alpha) > -1/2$ we have Poisson's integral representation (see, e.g., \cite{watson:bessel})
\[
	J_{\alpha}(z) = \frac{\left(\frac{z}{2}\right)^{\alpha}}{\Gamma\!\left(\alpha + \frac{1}{2}\right) \Gamma\!\left(\frac{1}{2}\right)} \int_{-1}^{1} \left(1-t^2\right)^{\alpha - 1/2} e^{izt} \,dt,
\]
which is of the required form once $z$ is replaced with $-iz$.  As far as we can tell, asymptotics for the zeros of sections of the Bessel functions have not been previously studied.  We state the result formally as a corollary to Theorem \ref{maintheo}.  This corollary is illustrated in Figure \ref{besselzeros}.

\begin{corollary}
Let $J_\alpha$ be the Bessel function of order $\alpha$ with $\Re(\alpha) > -1/2$, and for $n$ even let
\[
	s_n(J_{\alpha};z) = \left(\frac{z}{2}\right)^{\alpha} \sum_{k=0}^{n/2} \frac{(-1)^k}{\Gamma(k+1)\Gamma(k+\alpha+1)} \left(\frac{z}{2}\right)^{2k}
\]
be its $n^{\text{th}}$ section.  Then, for a subsequence $\{N\}$ of the indices $\{n\}$ as defined in Theorem \ref{rosentheo}, the zeros of the normalized sections $s_N(J_{\alpha};Nz)$ have as their limit points the set
\begin{align*}
	D(J_{\alpha}) &= \left\{z \in \C \,\colon \Im(z) \geq 0,\,\,\, |z| \leq 1, \,\,\,\text{and}\,\,\, \left|ze^{1+iz}\right| = 1 \right\} \\
			&\qquad \cup \left\{z \in \C \,\colon \Im(z) \leq 0,\,\,\, |z| \leq 1, \,\,\,\text{and}\,\,\, \left|ze^{1-iz}\right| = 1 \right\} \\
			&\qquad \cup \left\{x \in \R \,\colon -1/e \leq x \leq 1/e \right\}.
\end{align*}
Conversely, every point of $D(J_{\alpha})$ is a limit point of zeros.  Zeros of $s_N(J_{\alpha};Nz)$ in $\Im(z) > 0$ which do not converge to the point $z=i$ satisfy
\[
	\left|z e^{1+iz}\right| = 1 + \frac{\log N}{2N} + O(1/N)
\]
as $N \to \infty$, and zeros in $\Im(z) < 0$ which do not converge to $z=-i$ satisfy
\[
	\left|z e^{1-iz}\right| = 1 + \frac{\log N}{2N} + O(1/N)
\]
as $N \to \infty$.  Furthermore, if $\alpha$ is real, the modified indices $\{N\}$ in this result may be replaced everywhere by the original even indices $\{n\}$.
\label{besselcorollary}
\end{corollary}

\begin{proof}
The only parts of Corollary \ref{besselcorollary} which do not follow immediately from Theorem \ref{maintheo} are the claims
\begin{enumerate}[label=(\roman*)]
\item Every point of the line segment $\{x \in \R \,\colon -1/e \leq x \leq 1/e\}$ is a limit point of zeros, and
\item If $\alpha$ is real, the modified indices $\{N\}$ may be replaced by the original indices $\{n\}$.
\end{enumerate}

To prove claim (i) we cite the known result \cite{watson:bessel} that the large zeros of the Bessel function $J_{\alpha}$ are given by the asymptotic expansion
\[
	\left(k+\frac{\alpha}{2}-\frac{1}{4}\right)\pi - \frac{4\alpha^2 - 1}{8 \left(k+\frac{\alpha}{2}-\frac{1}{4}\right)\pi} - \frac{(4\alpha^2 - 1)(28\alpha^2-31)}{384 \left[\left(k+\frac{\alpha}{2}-\frac{1}{4}\right)\pi\right]^2} - \cdots.
\]
Every point of the imaginary axis is thus a limit point of the zeros of $J_{\alpha}(Nz)$, so that every point of the segment $\{x \in \R \,\colon -1/e \leq x \leq 1/e\}$ is a limit point of the zeros of the normalized sections $s_N(J_{\alpha};Nz)$.  Thus claim (i) is proved.

To prove claim (ii) we write
\[
	s_n(J_{\alpha};-i n z) = \left(-\frac{i n z}{2}\right)^{\alpha} P_n\!\left(z^2\right),
\]
where
\[
	P_n(z) = \sum_{k=0}^{n/2} \frac{n^{2k}}{4^k \Gamma(k+1)\Gamma(k+\alpha+1)} \,z^k.
\]
By the Enestr\"om-Kakeya theorem (Theorem \ref{kakene}), all zeros of $P_n(z)$ satisfy
\[
	|z| \leq \frac{2n+4}{n^2} \cdot \frac{\Gamma\left(\frac{n}{2} + \alpha + 2\right)}{\Gamma\left(\frac{n}{2}+\alpha+1\right)}.
\]
We apply Stirling's formula for the Gamma function to see that
\[
	\frac{\Gamma\left(\frac{n}{2} + \alpha + 2\right)}{\Gamma\left(\frac{n}{2}+\alpha+1\right)} \sim \frac{n}{2},
\]
whence
\[
	\frac{2n+4}{n^2} \cdot \frac{\Gamma\left(\frac{n}{2} + \alpha + 2\right)}{\Gamma\left(\frac{n}{2}+\alpha+1\right)} \longrightarrow 1
\]
as $n \to \infty$.  This tells us that the limit points of the zeros of the polynomials $P_n$, and hence of the sections $s_n(J_{\alpha};n z)$, lie in the closed unit disk.  This replaces the use of Theorem \ref{rosentheo} in the proof of Theorem \ref{maintheo}, which is the origin of the restriction of the indices to the subsequences $\{N\}$.  This proves claim (ii), completing the proof of the corollary.
\end{proof}

\begin{figure}[h!tb]
	\centering
	\begin{tabular}{cc}
		\includegraphics[width=0.45\textwidth]{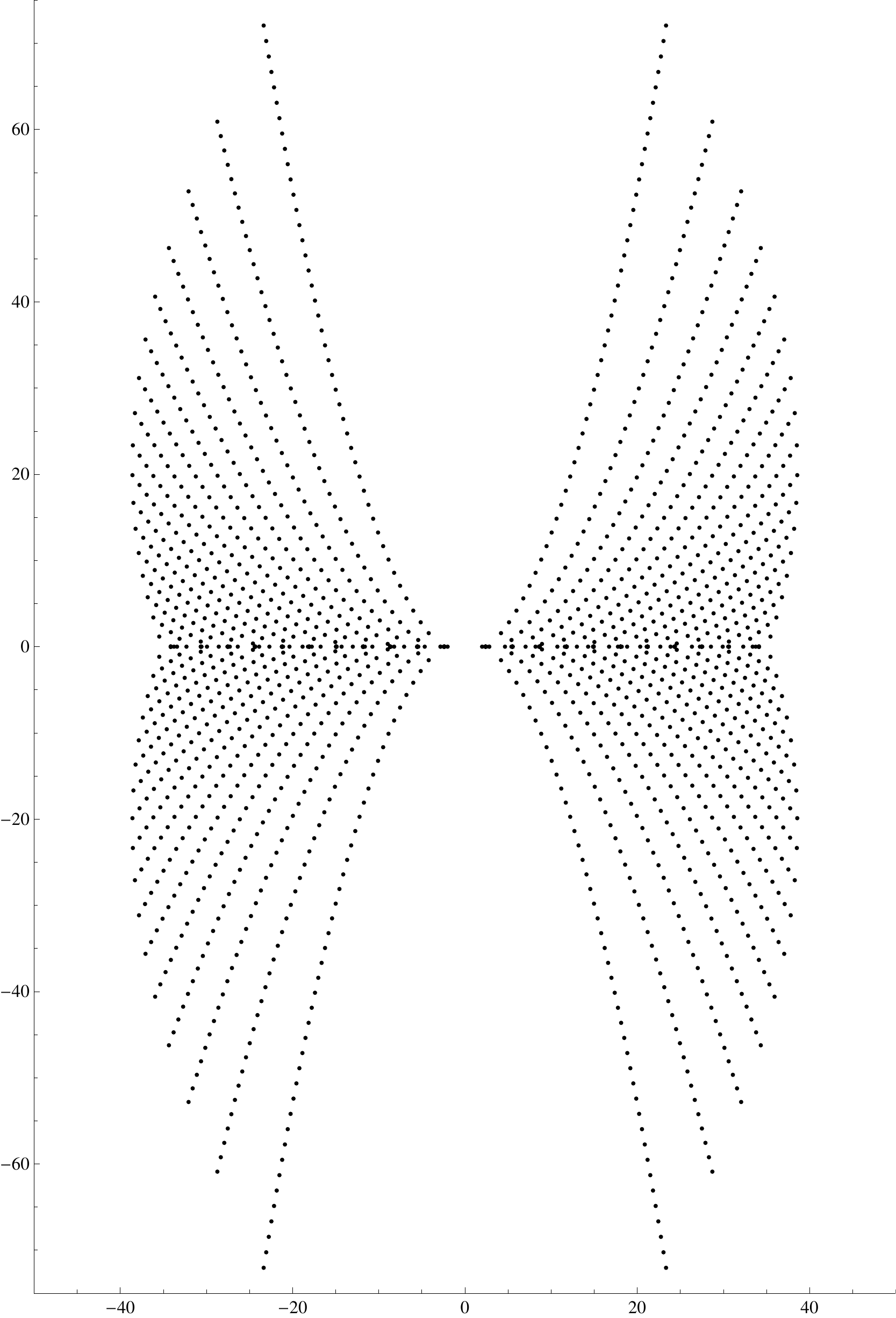}
			& \includegraphics[width=0.45\textwidth]{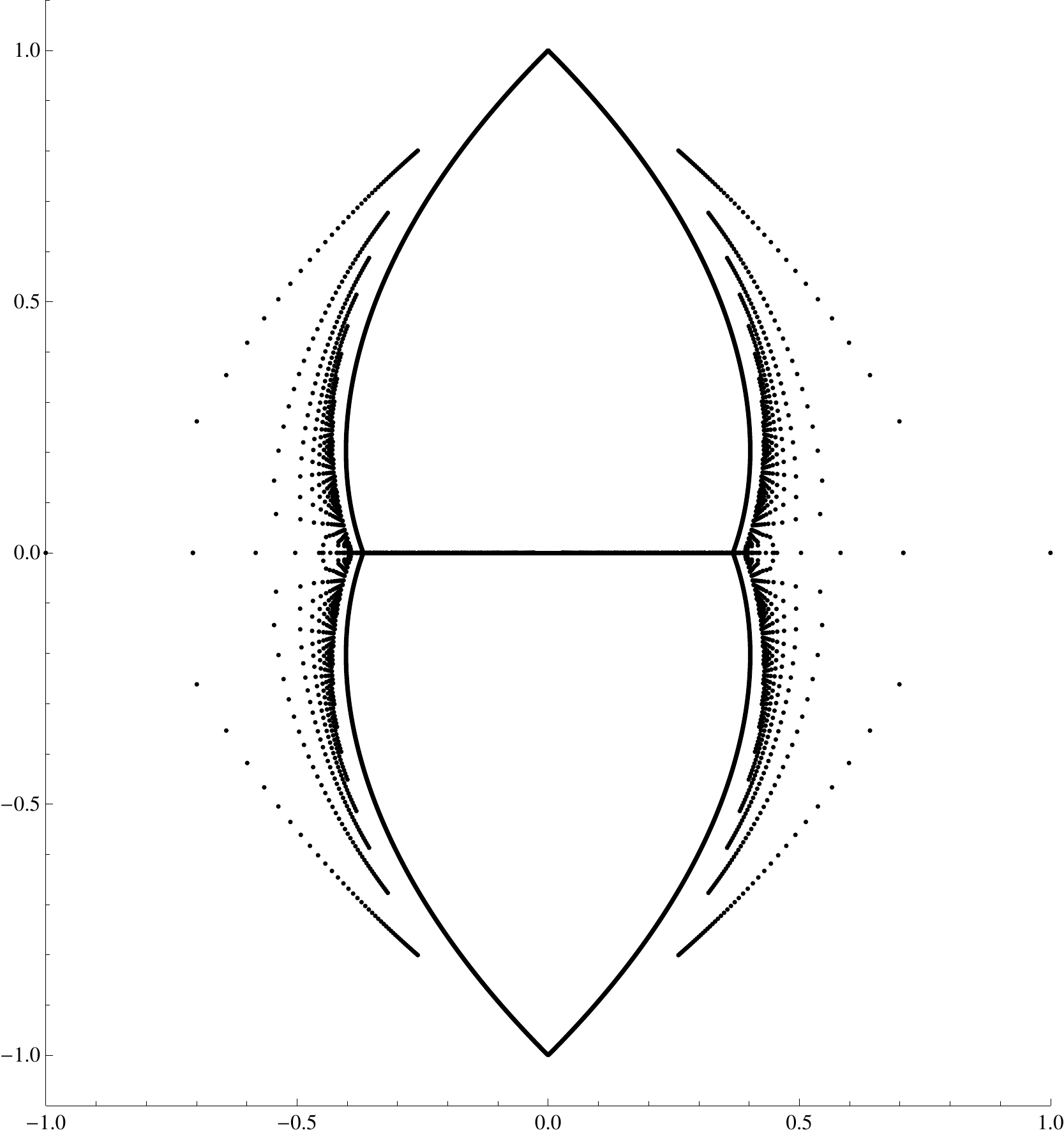}
	\end{tabular}
	\caption{LEFT: Zeros of the sections $s_n(J_0;z)$ $(n=2,4,\ldots,90)$.  RIGHT: Zeros of the normalized sections $s_n(J_0;nz)$ $(n=2,4,\ldots,90)$ with their Szeg\H{o} curve.}
\label{besselzeros}
\end{figure}

\section{Generalizing the Method}
\label{gen}

Though most of this thesis has dealt with sections of entire functions, we showed in Section \ref{prelims:strat:mob} that it is sometimes possible to perform this analysis on sections of series with positive finite radius of convergence.  As another example, if
\begin{equation}
	f(z) = \frac{1}{(1-z)^2} = \sum_{k=0}^{\infty} (k+1) z^k,
\label{finiteradfunc}
\end{equation}
we can use a similar approach to show that the zeros of the sections $s_n(f;z)$ which do not converge to the point $z=1$ satisfy
\[
	|z| = 1 - \frac{\log n}{n} + O(1/n)
\]
as $n \to \infty$ (see Figure \ref{finiteradplot}).  Usually this analysis will require the use of the Enestr\"om-Kakeya theorem or one of its variants to ensure that the zeros eventually lie within the radius of convergence, but somtimes tricks can be employed if this is not actually the case.  If, say, the relevant zeros all lie outside the radius of convergence, we could try studying the zeros of the polynomials $z^n s_n(f;1/z)$ instead.

\begin{figure}[htb]
	\centering
	\includegraphics[width=0.6\textwidth]{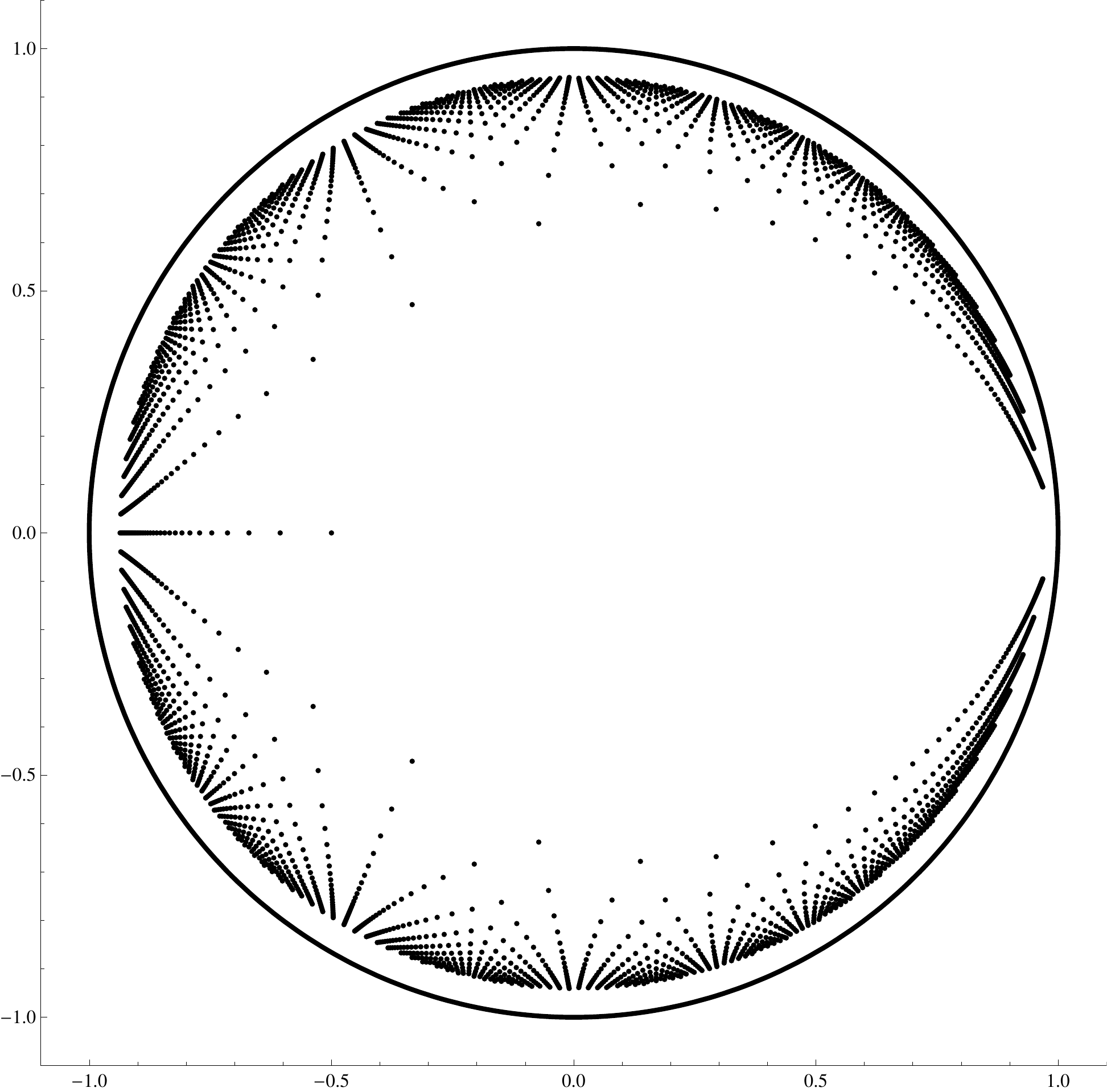}
	\caption{Zeros of $s_n(f;z)$ ($n=1,2,\ldots,75$), where $f$ is the rational function in \eqref{finiteradfunc}, and their Szeg\H{o} curve, the unit circle.}
\label{finiteradplot}
\end{figure}

In the future we would like to investigate the possibility of studying exponential integrals of integer order greater than one.  For example, it would be interesting to see if we could simply replace $z$ with $z^m$ for some integer $m$ and derive analogous results for the new function
\[
	F_m(z) = F\!\left(z^m\right) = \int_{-a}^{b} \varphi(t) e^{z^m t} \,dt,
\]
which would have exponential order $m$.

Lastly we speculate that it may be possible to iterate the process described in this work.  If $f$ is an entire function amenable to the current method---that is, if reasonably-detailed asymptotics for $f(z)$ and its tail $t_n(f;z) = f(z) - s_n(f;z)$ are known---then it may be possible to determine useful asymptotics for a new function $\hat{f}(z)$, defined by
\[
	\hat{f}(z) = \int_a^b \varphi(t) f(zt) \,dt,
\]
as well as for its tail $t_n(\hat{f};z)$.  The zeros of the sections of the new function $\hat{f}$ could then be studied.  It would be interesting to see what, if any, characteristics of the zero distribution associated with $f$ carry over to the one associated with $\hat{f}$.




\bibliographystyle{amsplain}
\renewcommand{\bibname}{Bibliography} 
\bibliography{References/biblio} 

\appendix
\chapter{Mathematica Code for the Plots}

\ifpdf
    \graphicspath{{Appendix1/Appendix1Figs/PNG/}{Appendix1/Appendix1Figs/PDF/}{Appendix1/Appendix1Figs/}}
\else
    \graphicspath{{Appendix1/Appendix1Figs/EPS/}{Appendix1/Appendix1Figs/}}
\fi

Here we give examples of the code used to create the various plots in this thesis.  All plots were created in Mathematica 7.

The zeros of the sections of any function $f(z)$ which is analytic at $z=0$ can be plotted using code like the following.
\acode\begin{lstlisting}[title={Code Snippet 1.}]
f[x_]:=E^x;
s[n_,x_]:=Normal[Series[f[z],{z,0,n}]]/.z->x;

numpolys=20;
start=1;
allzeros={};
For[k=start,k<numpolys+start,k++,
    newzeros=x/.NSolve[s[k,x]==0,x,70];
    For[j=1,j<=Length[newzeros],j++,
        AppendTo[allzeros,
            {Re[newzeros[[j]]],Im[newzeros[[j]]]}
        ];
    ];
];

Show[Graphics[{Point[allzeros]}]]
\end{lstlisting}\zcode
Here, the function in question is $f(z) = e^z$, and the zeros of the first 20 (from {\bf \footnotesize \ttfamily start} to {\bf \footnotesize \ttfamily numpolys+start-1}) sections are plotted, producing an image like Figure \ref{ex1plot}.

\begin{figure}[htb]
	\centering
	\includegraphics[width=0.6\textwidth]{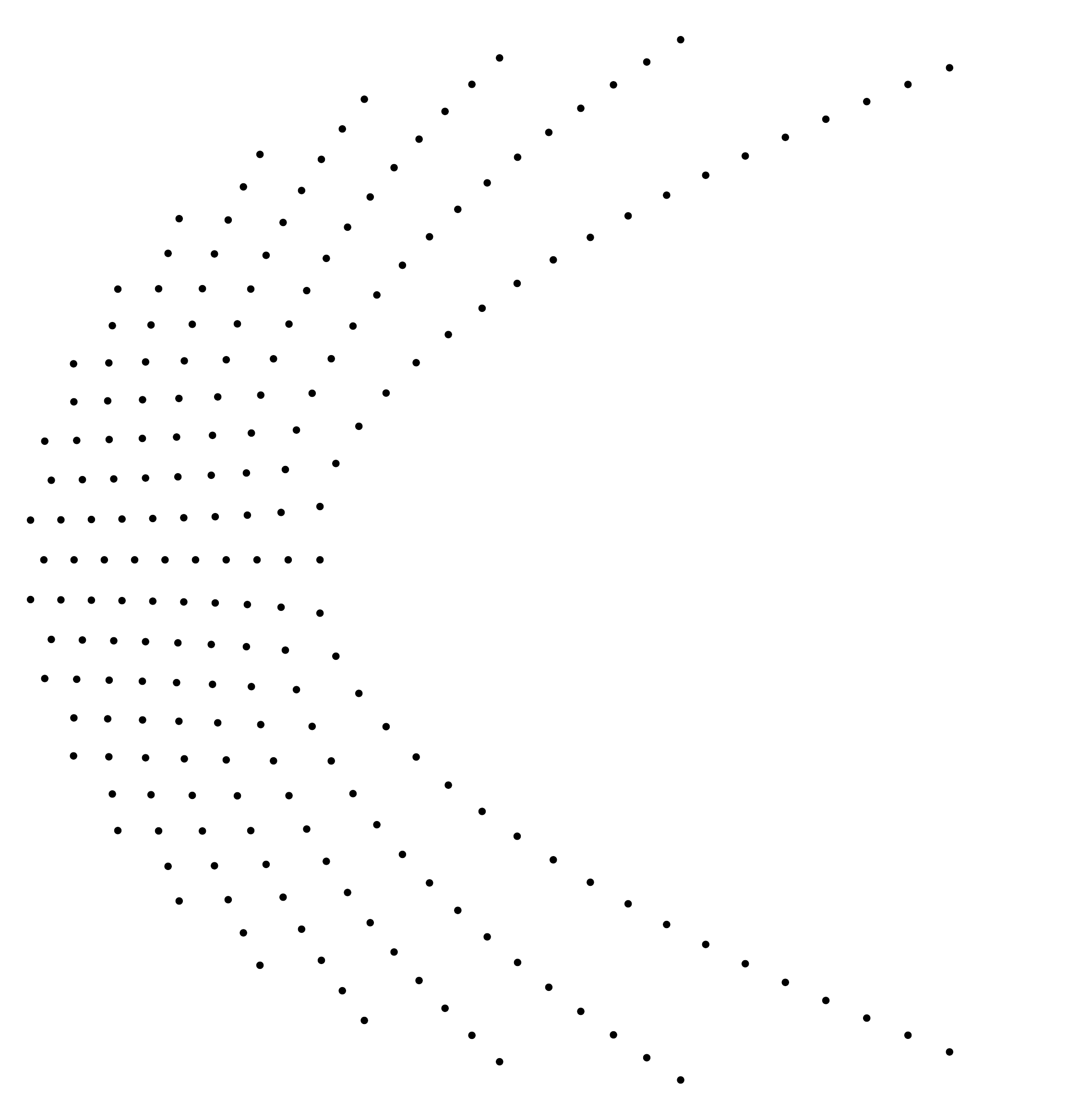}
\caption{Output of Code Snippet 1.}
\label{ex1plot}
\end{figure}

\newpage

To plot the zeros of the normalized sections, simply replace {\bf \footnotesize \ttfamily x} with {\bf \footnotesize \ttfamily k x} in the {\bf \footnotesize \ttfamily NSolve} function.  The modified block from Code Snippet 1 follows.
\acode\begin{lstlisting}[title={Code Snippet 2.}]
numpolys=20;
start=1;
allzeros={};
For[k=start,k<numpolys+start,k++,
    newzeros=x/.NSolve[s[k,k x]==0,x,70];
    For[j=1,j<=Length[newzeros],j++,
        AppendTo[allzeros,
            {Re[newzeros[[j]]],Im[newzeros[[j]]]}
        ];
    ];
];
\end{lstlisting}\zcode

\begin{figure}[htb]
	\centering
	\includegraphics[width=0.6\textwidth]{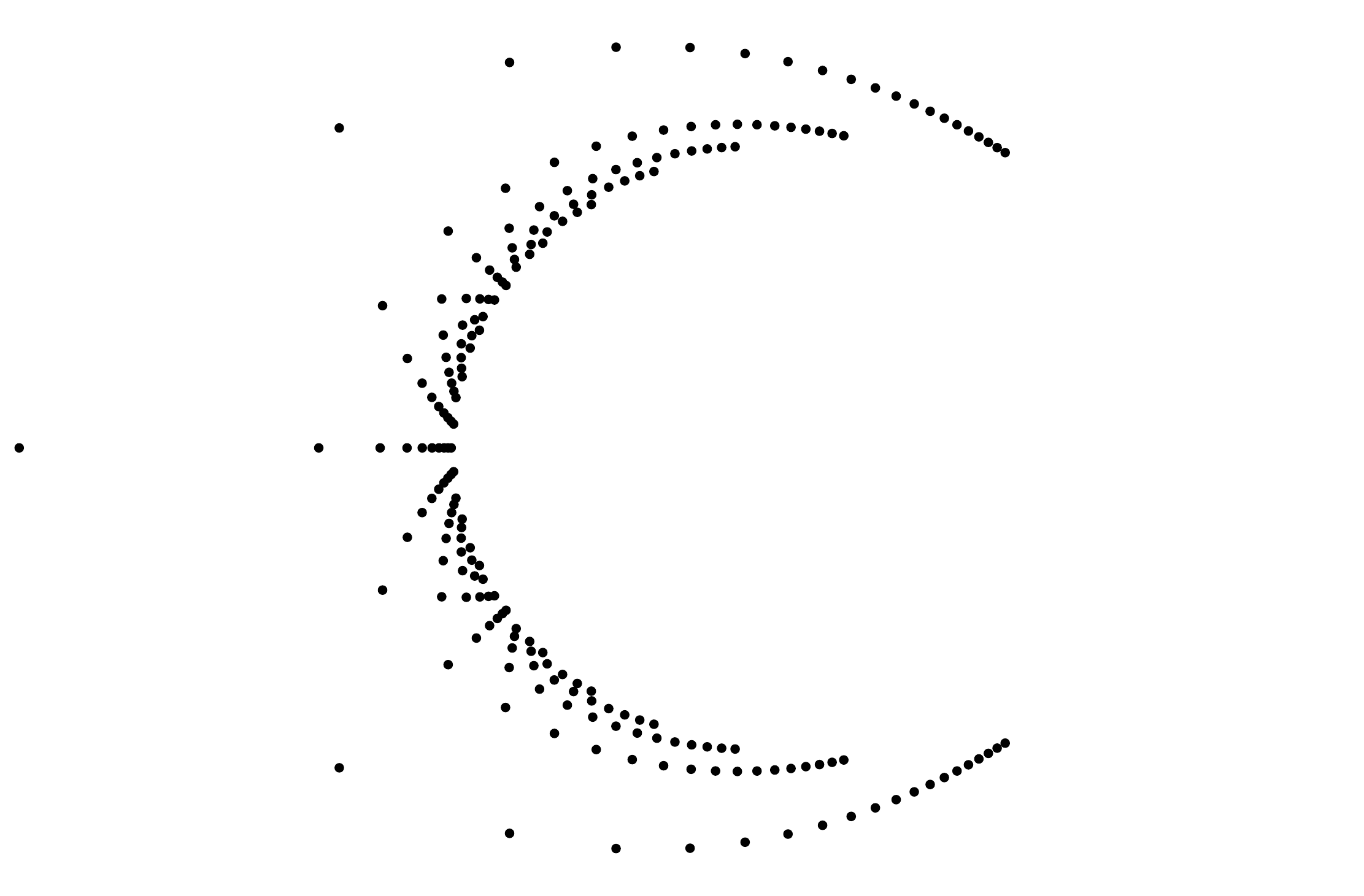}
\caption{Output of Code Snippet 2.}
\end{figure}

\newpage

The Szeg\H{o} curves can be drawn with the {\bf \footnotesize \ttfamily ContourPlot} function.  The code to draw the Szeg\H{o} curve for the exponential function is given below.
\acode\begin{lstlisting}[title={Code Snippet 3.}]
Show[
    Graphics[{Point[allzeros]}],
    ContourPlot[
        Abs[(x+I y) E^(1-(x+I y))]==1,
        {x,-1/E,1},{y,-1,1},
        ContourStyle->{Black},
        PlotPoints->40
    ]
]
\end{lstlisting}\zcode

\begin{figure}[htb]
	\centering
	\includegraphics[width=0.6\textwidth]{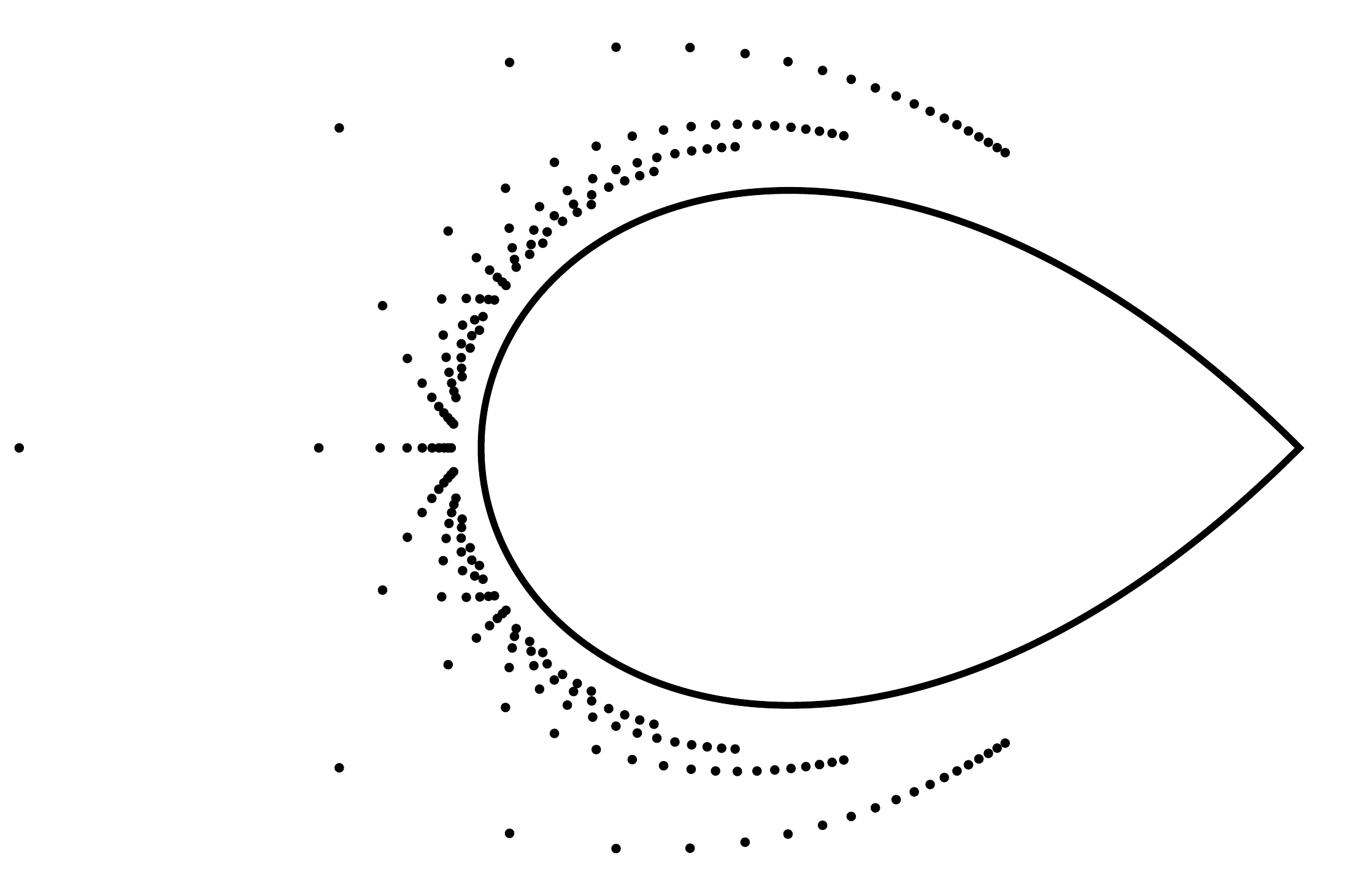}
\caption{Output after running Code Snippet 2 then Code Snippet 3.}
\end{figure}



\end{document}